\documentclass[11pt,twoside]{preprint}
\usepackage[a4paper,innermargin=1.4in,outermargin=1.4in,bottom=1.5in,marginparwidth=1.5in,marginparsep=3mm]{geometry}
\usepackage{amsmath,amsthm,amssymb,enumerate}
\usepackage{hyperref}
\usepackage{mathrsfs}
\usepackage{breakurl}

\usepackage[osf]{newtxtext}
\usepackage[varqu,varl]{inconsolata}
\usepackage[cal=boondoxo]{mathalfa}
\usepackage[bigdelims,vvarbb,cmintegrals]{newtxmath}

\usepackage{xcolor}
\usepackage{mhequ}
\usepackage{comment}
\usepackage{microtype}

\colorlet{darkblue}{blue!90!black}
\colorlet{darkred}{red!90!black}

\makeatletter
\newenvironment{claim}[1][MM]
               {\list{$\bullet$}{%
  \setbox\@tempboxa\hbox{#1}\@tempdima\wd\@tempboxa%
  \setlength{\labelwidth}{\@tempdima}
  \advance\@tempdima by 1em%
  \setlength{\leftmargin}{\@tempdima}
  \setlength{\parsep}{1mm}\setlength{\itemindent}{0mm}%
  \setlength{\labelsep}{2mm}\setlength{\itemsep}{0mm}%
  \setlength{\topsep}{1mm}%
}}{\endlist}
\makeatother

\renewcommand{\theequation}{\thesection.\arabic{equation}}

\newtheorem{theorem}{Theorem}[section]
\newtheorem{lemma}[theorem]{Lemma}
\newtheorem{proposition}[theorem]{Proposition}

\theoremstyle{definition}

\newtheorem{definition}[theorem]{Definition}
\newtheorem{example}[theorem]{Example}

\theoremstyle{remark}
\newtheorem{remark}[theorem]{Remark}

\def\scal#1{\langle #1 \rangle}

\newcommand{\vn}[1]{{\vert\kern-0.23ex\vert\kern-0.23ex\vert #1 
    \vert\kern-0.23ex\vert\kern-0.23ex\vert}}

\newcommand{\bn}[1]{{[\kern-0.5ex] #1 
    [\kern-0.5ex]}}

\allowdisplaybreaks

\newcommand\bR{\mathbf{R}}
\newcommand\bone{\mathbf{1}}

\newcommand\cR{\mathcal{R}}
\newcommand\CR{{\mathcal{R}}}
\newcommand\cK{\mathcal{K}}
\newcommand\cD{\mathcal{D}}
\newcommand\CD{\mathcal{D}}
\newcommand\cB{\mathcal{B}}

\newcommand\cC{\mathcal{C}}
\newcommand\CC{{\mathcal{C}}}
\newcommand\cN{\mathcal{N}}
\newcommand\CN{\mathcal{N}}
\newcommand\cI{\mathcal{I}}
\newcommand\cJ{\mathcal{J}}
\newcommand\cQ{\mathcal{Q}}

\newcommand\cS{\mathcal{S}}
\newcommand\cG{\mathcal{G}}
\newcommand\cV{\mathcal{V}}
\newcommand\CV{\mathcal{V}}

\newcommand\cW{\mathcal{W}}
\newcommand\CW{\mathcal{W}}

\newcommand\cP{\mathcal{P}}

\newcommand\CO{\mathcal{O}}

\newcommand\scK{\mathscr{K}}
\newcommand\scZ{\mathscr{Z}}
\newcommand\scD{\mathscr{D}}

\newcommand\scT{\mathscr{T}}

\newcommand\frK{\mathfrak{K}}
\newcommand\frs{\mathfrak{s}}
\newcommand\frm{\mathfrak{m}}

\def\half{\textstyle{1\over 2}}
\def\eps{\varepsilon}
\def\Erf{\mathop{\mathrm{Erf}}}

\newcommand{\R}{{\mathbb{R}}}
\newcommand{\E}{\mathbf{E}}
\newcommand{\Var}{\mathbf{Var}}

\def\supp{\mathop{\mathrm{supp}}}
\def\d{\partial}
\def\id{\mathrm{id}}

\begin{document}
\title{Singular SPDEs in domains with boundaries}

\author{M\'at\'e Gerencs\'er${}^1$ and Martin Hairer${}^2$}
\institute{IST Austria, \email{mate.gerencser@ist.ac.at} \and 
University of Warwick, \email{M.Hairer@Warwick.ac.uk}}
\date{\today}

\maketitle
\begin{abstract}
We study spaces of modelled distributions with singular behaviour near the boundary of a domain that, in the context of the 
theory of regularity structures, allow one to give robust solution theories for singular stochastic PDEs with boundary 
conditions. The calculus of modelled distributions established in Hairer (\emph{Invent. Math.} \textbf{198}, 2014) is extended to this 
setting. We formulate and solve fixed point problems in these spaces with a class of kernels that is sufficiently large to 
cover in particular the Dirichlet and Neumann heat kernels. These results are then used to provide solution theories 
for the KPZ equation with Dirichlet and Neumann boundary conditions and for 
the 2D generalised parabolic Anderson model with Dirichlet boundary conditions.

In the case of the KPZ equation with Neumann boundary conditions, we show that, depending
on the class of mollifiers one considers, a ``boundary renormalisation'' takes place. In other words,
there are situations in which a certain boundary condition is applied to an approximation to the KPZ
equation, but the limiting process is the Hopf-Cole solution to the KPZ equation with 
a \textit{different} boundary condition.
\end{abstract}

\setcounter{tocdepth}{2}
\tableofcontents

\section{Introduction}

The theory of regularity structures, recently developed in \cite{H0}, was in large part motivated by, and very successful in dealing with, singular stochastic partial differential equations (SPDEs). These SPDEs are typically semilinear perturbations of the stochastic heat equation, with their formal right-hand side including expressions that are not well-defined even for functions that are as regular as the solution of the linear part. One well-known example is the KPZ equation
$$
\partial_t u=\Delta u+(\partial_x u)^2+\xi,
$$
where $\xi$ is the $1+1$-dimensional space-time white noise. From the linear theory we know that $u$ is not expected to have better (parabolic) regularity than $1/2$, so its spatial derivative is a distribution, which, in general, one cannot take the square of. The theory developed in \cite{H0} provided a robust concept of solution to equations like KPZ \cite{H_KPZ}, $\Phi^4_3$, the parabolic Anderson model in both two \cite{H0} and three \cite{HP15} dimensions, the dynamical Sine-Gordon model \cite{HShen} on the torus, or such equations on the whole Euclidean space \cite{HL15}.
As neither the torus nor the whole space has boundaries, the spatial behaviour in these examples are `uniform', and the only blow-up of the generalised abstract Taylor expansions - also referred to as `modelled distributions' - that describe the solutions  occur at the $\{t=0\}$ hyperplane of the initial time.

The aim of the present article is to provide a framework within the context of this theory, with which one can provide 
solution theories for initial-boundary problems for singular SPDEs. The appropriate spaces of modelled distributions 
introduced here are flexible enough to account for singular behaviour at the spatial boundary. These are similar to the 
singularities at the initial time treated in \cite{H0} and indeed a similar calculus can be built on them. One could hope 
that, provided such a generalisation of the abstract calculus is obtained, coupling it with rest of the theory automatically 
gives solution theories of the same equations that were previously considered without or with periodic boundary conditions, 
now with for instance Dirichlet or Neumann boundary conditions. However, a subtle-looking but notable difference is that the codimension $2$ of the initial time hyperplane is replaced by the codimension $1$ of the spatial boundary, and therefore dual elements of spaces of test functions supported away from the boundary which are uniformly `locally in $\cC^\alpha$' for $\alpha<-1$  have no canonical extensions as bona fide distributions - a simple example for such situation is the function $1/|x|$, considered as an element of $\cD'(\R\setminus\{0\})$. As elements with (local) regularity less than $-1$ are quite common in applications (unlike elements with regularity less than $-2$), for each such object one has to make sense of their extensions, in a consistent way so that the sufficient continuity properties are preserved.
Although, unlike the rest of the theory, the treatment of this issue is not performed in a systematic way, 
the methods used to treat the examples discussed in the next section are likely to be relevant to different situations. 

\subsection{Applications}\label{subsec:applications}

We now give a few examples of singular SPDEs to which the framework developed in this article can be applied. 
The proofs of the results stated here are postponed to Section~\ref{sec:applications}.

Our first example is the Dirichlet problem for the two-dimensional generalised parabolic Anderson model given by
\begin{equs}[eq:0PAM]
        \partial_t u&=\Delta u+f_{ij}(u)\partial_iu\partial_j u+g(u)\xi\quad & \text{on } &\R_+\times D,\\
        u&=0, & \text{on } &\R_+\times\partial D,\\
        u&=u_0.&\text{on }&\{0\}\times D
\end{equs}
Here $\xi$ denotes two-dimensional spatial white noise, $g$ and $f_{ij}$, $i,j=1,2$ are smooth functions, $D$ is the square $(-1,1)^2$, and $u_0$ belongs to $\cC^\delta(\bar D)$ for some $\delta>0$ and vanishes on $\partial D$.

Take a smooth compactly supported function $\rho$ on $\R^2$ integrating to 1, define $\rho_\varepsilon(x)=\varepsilon^{-2}\rho(\varepsilon^{-1}x)$ and set $\xi_\varepsilon=\rho_\varepsilon\ast\xi$. Consider then the renormalised approximating initial / boundary value problem
\begin{equs}[2][eq:0PAM approx]
        \partial_t u^\varepsilon &=\Delta u^\varepsilon+f_{ij}(u^\varepsilon)(\partial_iu^\varepsilon\partial_j u^\varepsilon-\delta_{ij}C_\varepsilon g^2(u^\varepsilon)) &&\\
        &\quad +g(u^\varepsilon)(\xi_\varepsilon-2C_\varepsilon g'(u^\varepsilon))\quad & \text{on } &\R_+\times D,\\
        u^\varepsilon&=0, & \text{on } &\R_+\times \partial D,\\
        u^\varepsilon&=u_0, & \text{on } &\{0\}\times D,
\end{equs}
for some constants $C_\varepsilon$. One can solve \eqref{eq:0PAM approx} in the classical sense, and in the $\varepsilon\rightarrow0$ limit this provides a concept of local solution to \eqref{eq:0PAM} in the following sense.
\begin{theorem}\label{thm:PAM}
There exists a choice of diverging constants $C_\varepsilon$ and a random time $T>0$ such that the sequence $u^\varepsilon\bone_{[0,T]}$ converge in probability to a continuous function $u$. Furthermore,
provided that the constants $C_\varepsilon$ are suitably chosen, the limit does not depend on the choice of the mollifier $\rho$.
\end{theorem}

\begin{remark}
We believe that the choice $D = (-1,1)^2$ is not essential, the restriction to the square case
is mostly for the sake of convenience: it is easier to verify our conditions when the explicit form
of the Greens function is known.
\end{remark}

\begin{remark}
One could easily deal with inhomogeneous Dirichlet data of the type $u^\varepsilon=g$ on $\partial D$
by considering the equation for $u^\eps - \hat g$, where $\hat g$ is the harmonic extension of $g$ to 
all of $D$.
\end{remark}

Our next example is the KPZ equation with $0$ Dirichlet boundary condition.
Write this time  $\xi$ for space-time white noise and choose $u_0\in\cC^\delta([-1,1])$ for some $\delta>0$ with 
$u_0(\pm 1) = 0$.
Taking a smooth, 
compactly supported function $\rho$ integrating to $1$, define $\rho_\varepsilon(t,x)=\varepsilon^{-3}\rho(\varepsilon^{-2}t,\varepsilon^{-1}x)$ and set $\xi_\varepsilon=\rho_\varepsilon\ast\xi$. The approximating equations then read as
\begin{equs}[eq:0KPZ approx]
        \partial_t u^\varepsilon &= \half \d_x^2 u^\varepsilon+ (\partial_x u^\varepsilon)^2-C_\varepsilon+\xi_\varepsilon \quad & \text{on } &\R_+ \times [-1,1],\\
        u^\varepsilon &=0, & \text{on } &\R_+ \times \{\pm 1\},\\
        u^\varepsilon&=u_0 & \text{on } & \{0\}\times[-1,1].
\end{equs}

\begin{remark}
We have chosen to include the arbitrary constant $\half$ in front of the term $\d_x^2 u$ so that the corresponding
semigroup at time $t$ is given by the Gaussian with variance $t$. 
\end{remark}

We then have the following analogous result on local solvability.
\begin{theorem}\label{thm:KPZ D}
If $\rho$ satisfies the condition $\rho(x,t) = \rho(-x,t)$, then the 
statement of Theorem~\ref{thm:PAM} also holds for $u^\eps$ defined in \eqref{eq:0KPZ approx}.
\end{theorem}

\begin{remark}
If the additional symmetry on $\rho$ fails, then an analogous result holds, 
but an additional drift term appears in general, see for example \cite{HS15}.
\end{remark}

A more interesting situation arises when trying to define solutions to the KPZ equation with Neumann boundary 
conditions. First, in this case, it is much less clear \textit{a priori} what such a boundary condition
even means since solutions are nowhere differentiable. It is however possible to define a notion of
``KPZ equation with Neumann boundary conditions'' via the Hopf-Cole transform. Indeed, it suffices to realise
that, at least \textit{formally}, if $u$ solves
\begin{equ}[e:KPZ]
\d_t u = \half \d_x^2 u + (\d_x u)^2 + \xi\;,\qquad \d_x u(t,\pm 1) = c_\pm\;, 
\end{equ}
then the process $Z = \exp(2u)$ solves
\begin{equ}[e:SHE]
\d_t Z = \half \d_x^2 Z + 2 Z\,\xi\;,\qquad \d_x Z(t,\pm 1) = 2c_\pm Z(t,\pm 1)\;.
\end{equ}
The latter equation is well-posed as an It\^o stochastic PDE in mild form \cite{DPZ} (with the boundary condition
encoded in the choice of heat semigroup for the mild formulation), so that 
we can \textit{define} the ``Hopf-Cole solution'' to \eqref{e:KPZ} by $u = \half \log Z$ with $Z$ solving \eqref{e:SHE}.
This is the point of view that was taken in \cite{2016arXiv161004931C} where the authors showed that the height function
associated to a large
but finite discrete system of particles performing a weakly asymmetric simple exclusion 
process converges to the solutions to \eqref{e:KPZ}
with boundary conditions $c_\pm$ that are related to the boundary behaviour of the discrete system in a straightforward way. In particular, if the `net flow' of particles at each boundary is $0$, then $c_\pm=0$.  

One of the main results of the present article is to show that the values of $c_\pm$ are very ``soft'' in the 
sense that they in general depend in a rather non-trivial way on the fine details of the particular approximation
one considers for \eqref{e:KPZ}. This is not too surprising: after all, the solution itself is
not differentiable, so it is not so clear what we mean when we impose the value of its derivative at the
boundary. 
To formulate this more precisely, consider $\xi_\varepsilon=\rho_\varepsilon\ast\xi$ and $\hat u_0 \in \CC^\delta([-1,1])$ 
as before (except that we do not impose that $\hat u_0$ vanishes at the boundaries)
and let $\hat u^\eps$ be the solution to 
\begin{equs}[eq:0KPZ N approx]
        \partial_t \hat u^\varepsilon &=\half\d_x^2 \hat u^\varepsilon+(\partial_x \hat u^\varepsilon)^2+\xi_\varepsilon \quad & \text{on } &\R_+\times [-1,1],\\
        \partial_x \hat u^\varepsilon &= \hat b_\pm, & \text{on } &\R_+\times \{\pm 1\},\\
        \hat u^\varepsilon &=\hat u_0 &\text{on }&\{0\}\times[-1,1].
\end{equs}
We then have the following result.

\begin{theorem}\label{thm:KPZ N}
There exist constants $C_\eps$ with $\lim_{\eps \to 0} C_\eps = \infty$, as well as
 constants $a, c \in \R$ such that, setting
\begin{equ}[e:renormu]
 u^\eps(t,x) = \hat u^\eps(t,x) - C_\eps t - cx\;,
\end{equ}
the sequence $u^\eps$ converges, locally uniformly and in probability, to a limit $u$
solving the KPZ equation \eqref{e:KPZ} in the Hopf-Cole sense with boundary data 
$b_\pm = \hat b_\pm - c \pm a$ and with initial condition $u_0(x)=\hat u_0(x)-cx$.
In the particular case where $\rho(x,t) = \rho(-x,t)$, one has $c = 0$. 
\end{theorem}

\begin{remark}
Even in the symmetric case, one can have $a \neq 0$, so that one can end up 
with non-zero boundary conditions
in the limit, although one imposes zero boundary conditions for the approximation.
\end{remark}

\begin{remark}
The effect of subtracting $cx$ in \eqref{e:renormu} is the same as that of adding a drift 
term $2c \d_x u^\eps$ to the right hand side of \eqref{eq:0KPZ N approx} and changing the boundary condition
$\hat c_\pm$ into $\hat c_\pm - c$, which is the reason for the form of the constants $c_\pm$.
\end{remark}

\begin{remark}
At first sight, this may appear to contradict the results of \cite{Ismael} where the authors consider
the three-dimensional parabolic Anderson model in a rather general setting which covers
that of domains with boundary. Since this scales in exactly the same way as the KPZ equation
(after applying the Hopf-Cole transform), one would expect to observe a similar ``boundary renormalisation''
in this case. The reason why there is no contradiction with our results is that there is no
statement on the behaviour of the renormalisation term $\lambda^\eps$ in \cite[Thm~1]{Ismael} as a 
function of position. What our result suggests is that, at least in the flat case, one should
be able to take $\lambda^\eps$ of the form $\lambda^\eps = C_\eps + \mu$, where $C_\eps$ is a
constant and $\mu$ is some measure concentrated on the boundary of the domain.
\end{remark}

\begin{remark}
The recent result \cite{Nicolas} is consistent with our result in the sense that it shows 
that the ``natural'' notion of solution to \eqref{e:KPZ} with homogeneous Neumann boundary condition
(i.e.\ $c_\pm = 0$) does \textit{not} coincide with the Hopf-Cole solution with homogeneous
boundary data.
In this particular case, one possible interpretation is that, for any fixed time, the solution
to the KPZ equation is a forward / backwards semimartingale (in its own filtration) near the
right / left boundary point. It is then natural to define the ``space derivative'' at the boundary 
to be the derivative of its bounded variation component. When performing the Hopf-Cole
transform, one then picks up an It\^o correction term, which is precisely what one sees in \cite{Nicolas}.
Note however that it is not clear at all whether the homogeneous Neumann solution of \cite{Nicolas} can 
be obtained by considering \eqref{eq:0KPZ N approx} with $\hat b_\pm = 0$ for some mollifier 
$\rho$. This is because, with our conventions for units, this corresponds to the Hopf-Cole solution
with $b_\pm = \pm 1$, while in our case one has $|a| \le {1\over 2}$ as a consequence of
the explicit formula \eqref{eq:constant a} for typical choices of the mollifier, i.e.\ those
with $\rho \ge 0$.
\end{remark}

One has explicit expressions for $c$ and $a$  in terms of $\rho$:
with the notation $\bar\rho(s,y)=\rho(-s,-y)$ and $\Erf$ standing for the error function,
one has the identities
\begin{equs}
a&=\int_{\R^2}(\bar\rho\ast\rho)(s,y)\Big({1\over 2} - {1\over 2}\Erf\Big(\frac{|y|}{\sqrt{2|s|}}\Big) - 2|y| \CN(y,s)\Big)\,ds\,dy\;,
\label{eq:constant a}
\\
c&=2\int_{\R^2}(\bar\rho\ast\rho)(s,y)\, y\CN(y,s)\,ds\,dy\;,
\label{eq:constant c}
\end{equs}
where $\CN$ denotes the heat kernel, see Section~\ref{sec:KPZNeumann} below. Note that in both cases the function
integrated against $\bar\rho\ast\rho$ vanishes at $s = 0$ for any fixed value of $y$, so that 
$a = c = 0$ if we consider the KPZ equation driven by purely spatial regularisations 
of white noise. To the best of our knowledge, this is the first observed instance of ``boundary renormalisation''
for stochastic PDEs. On the other hand, it is somewhat similar to the effects one observes in the
analysis of (deterministic) singularly perturbed problems in the presence of boundary layers, see for example 
\cite{Hinch,Holmes}.

The remainder of the article is structured as follows. After recalling some elements of the theory of
regularity structures in Section~\ref{sec:H0}, mostly to fix our notations, we introduce in Section~\ref{sec:def} 
the spaces of modelled distributions that are relevant for solving singular stochastic PDEs on domains.
Section~\ref{sec:calculus} is then devoted to a rederivation of the calculus developed in \cite{H0}, adapted to these
spaces, with an emphasis on those aspects that actually differ in the present context. 
In Section~\ref{sec:FPP}, we then ``package'' these results into a rather general fixed point theorem, 
which is finally applied to the above examples in Section~\ref{sec:applications}.

\subsection*{Acknowledgements}

{\small
MH gratefully acknowledges support by the Leverhulme Trust and by an ERC consolidator grant,
project 615897.
MG thanks the support of the LMS Postdoctoral Mobility Grant.
}

\section{Elements of the theory of regularity structures}\label{sec:H0}

First let us summarise the relevant definitions, constructions, and results from the theory of regularity structures that we will need in the sequel.
\subsection{Main definitions}
\begin{definition}
A regularity structure $\scT=(A,T,G)$ consists of the following elements.
\begin{claim}
\item An index set $A\subset\R$ which is locally finite and bounded from below.
\item A graded vector space $T=\bigoplus_{\alpha\in A}T_\alpha$ with each $T_\alpha$ a finite-dimensional
normed vector space.
\item A group $G$ of linear operators $\Gamma:T\rightarrow T$, such that, for all $\Gamma\in G$, $\alpha\in A$, $a\in T_\alpha$, one has
$
\Gamma a-a=\bigoplus_{\beta<\alpha}T_\beta
$.
\end{claim}
We will furthermore always consider situations where $T_0$ contains a distinguished element $\bone$
of unit norm which is fixed by the action of $G$.
\end{definition}
\begin{definition}
Given a regularity structure and $\alpha\leq 0$, a sector $V$ of regularity $\alpha$ is a $G$-invariant subspace of $T$ of the form $V=\bigoplus_{\beta\in A}V_{\beta}$ such that $V_\beta\subset T_\beta$ and $V_\beta=\{0\}$ for $\beta<\alpha.$
\end{definition}
With $V$ as above, we will always use the notations
$
V_\alpha^+=\bigoplus_{\gamma\geq\alpha}V_\gamma$ and $V_{\alpha}^-=\bigoplus_{\gamma<\alpha}V_\gamma
$,
with the convention that the empty direct sum is $\{0\}$.
Some further notations will be useful. For $a\in T$, its component in $T_\alpha$ will be denoted either by $\cQ_\alpha a$ or by $(a)_\alpha$ and the norm of $(a)_\alpha$ in $T_\alpha$ is $\|a\|_\alpha$. The projection onto $T_\alpha^-$ is denoted by $\cQ_\alpha^-$. The coefficient of $\bone$ in $a$ is denoted by $\langle \bone,a\rangle$. 

We henceforth fix a \textit{scaling} $\frs$ on $\R^d$, which is just an element of $\mathbb{N}^d$. 
We use the notations $|\frs|=\sum_{i=1}^d\frs_i$, and, for any $d$-dimensional 
multiindex $k$, we write $|k|_\frs=\sum_{i=1}^d\frs_i k_i$. 
A scaling also induces a metric on $\R^d$ by
$d_\frs(x,y)=\sum_{i=1}^d|x_i-y_i|^{1/\frs_i}$,
and this quantity will also sometimes be denoted by $\|x-y\|_\frs$. This is homogeneous under the mappings $\cS_\frs^\delta$ defined by
$$
\cS_\frs^\delta(x_1,\ldots,x_d)=(\delta^{-\frs_1}x_1,\ldots,\delta^{-\frs_d}x_d)
$$
in the sense that $\|\cS_\frs^\delta x\|_\frs=\delta^{-1}\|x\|_\frs$. The ball with center $x$ and radius $r$, in the above sense, is denoted by $B(x,r)$. We also define the mapping $\cS_{\frs,x}^\delta$, acting on $L_1(\R^d)$ by
$$
(\cS_{\frs,x}^\delta \varphi)(y)=\delta^{-|\frs|}\varphi(\cS_\frs^\delta(y-x)).
$$
We will also sometimes use the shortcut $\varphi_x^{\delta} = \cS_{\frs,x}^\delta \varphi$.

One important regularity structure is that of the polynomials in $d$ commuting variables, 
which we denote by $X_1,\ldots,X_d$. For any nonzero multiindex $k$, we denote
$$
X^k=X_1^{k_1}\cdots X_d^{k_d},
$$
and also use the notation $X^0=\bone.$ We define the index set $\bar A=\mathbb{N}$, for any $n\in\mathbb{N}$, the subspaces
$$
\bar T_n=\text{span}\{X^k:|k|_\frs=n\},
$$
and for any $h\in\R^d$, the linear operator $\bar\Gamma_h$ by
$$
(\bar\Gamma_h P)(X)=P(X+h).
$$
It is straightforward to verify that this defines a regularity structure $\bar \scT$, with structure group $\bar G=\{\bar\Gamma_h:h\in\R^d\}\approx\R^d$.

In most of the following we consider $d$, $\scT$, and $\frs$ to be fixed.
We will always assume that our regularity structures contain $\bar \scT$ in the sense of \cite[Sec.~2.1]{H0}.
A concise definition of the H\"older spaces of all (non-integer) exponents that are used in the sequel is the following.
\begin{definition}
A distribution $\xi \in \CD'(\R^d)$ is said to be of class $\cC^\alpha$, if for every compact set $\frK\subset\R^d$ it holds that
\begin{equation}\label{eq:holder def}
|\xi (\varphi_x^{\delta})|\lesssim\delta^\alpha
\end{equation}
uniformly over $\delta\leq 1$, $x\in\frK$, and over test functions $\varphi$ supported on $B(0,1)$ that furthermore
have all their derivatives up to order $(\lceil-\alpha\rceil+1) \vee 0$ bounded by $1$ and satisfy
$\int\varphi(x)x^k\,dx=0$ for every multiindex $|k| < \alpha$.
The best proportionality constant in \eqref{eq:holder def} is denoted by $\|\xi\|_{\alpha;\frK}$.
\end{definition}
We shall also use the notation $\cB^r$ for smooth functions $\varphi$ supported on $B(0,1)$ and having derivatives up to order $r$ bounded by $1$.

\begin{definition}
A model for a regularity structure $\scT$ on $\R^d$ with a scaling $\frs$ consists of the following elements.
\begin{claim}
\item A map $\Gamma:\R^d\times\R^d\rightarrow G$ such that $\Gamma_{xy}\Gamma_{yz}=\Gamma_{xz}$ for all $x$, $y$, $z\in\R^d$.
\item A collection of continuous linear maps $\Pi_x: T\rightarrow\cS'(\R^d)$ such that $\Pi_x=\Pi_y\circ\Gamma_{xy}$ for all $x$, $y\in\R^d$.
\end{claim}
Furthermore, for every $\gamma>0$ and compact $\frK\subset\R^d$, the bounds
\begin{equation}\label{eq: model bounds}
|(\Pi_x a)(\cS_{\frs,x}^\delta\varphi)|\lesssim\|a\|_l\delta^l,\quad\quad
\|\Gamma_{xy}a\|_m\lesssim\|a\|_l\|x-y\|_\frs^{l-m}
\end{equation}
hold uniformly in $x,y\in\frK$, $\delta\in(0,1]$, $\varphi\in\cB^r$, $l<\gamma$, $m<l$, and $a\in T_l$. Here, $r$ is the smallest integer such that $l>-r$ for all $l\in A$. 

The best proportionality constants in \eqref{eq: model bounds} are denoted by $\|\Pi\|_{\gamma,\frK}$ and $\|\Gamma\|_{\gamma,\frK}$, respectively.
\end{definition}
We shall always assume that all models under consideration are compatible with the polynomials in the sense that $(\Pi_x X^k)(y)=(y-x)^k$ for any multiindex $k$.
A central notion of the theory is that of a modelled distribution, spaces of which are defined as follows.
\begin{definition}
Let $V$ be a sector and $(\Pi,\Gamma)$ be a model. Then, for $\gamma\in\R$, the space $\cD^\gamma(V;\Gamma)$ consists of all functions $f:\R^d\rightarrow V_\gamma^-$ such that, for every compact set $\frK$,
\begin{equ}\label{eq:vnorm}
\vn{f}_{\gamma,\frK}= 
\sup_{\substack{x,y\in\frK \\ \|x-y\|_\frs\leq1}}\sup_{l<\gamma}\frac{\|f(x)-\Gamma_{xy}f(y)\|_l}{\|x-y\|_\frs^{\gamma-l}}<\infty,
\end{equ}
where the supremum in $l$ runs over elements of $A$. 
\end{definition}
Although the spaces $\cD^\gamma$ depend on $\Gamma$, in many situation, where there can be no confusion about the model, this dependence will be omitted in the notation. The name `modelled distribution' is justified by the following result.
\begin{theorem}\label{thm: standard reco}
Let $V$ be a sector of regularity $\alpha$ and let $r=\lceil-\alpha+1\rceil$. Then for any $\gamma>0$ there exists a continuous linear map $\cR:\cD^\gamma(V)\rightarrow\cC^\alpha$ such that for every $C>0$, the bound
\begin{equation}\label{eq:standard reco estimate}
|(\cR f-\Pi_yf(y))(\psi_x^{\lambda})|\lesssim\lambda^{\gamma}\vn{f}_{\gamma,\supp \psi_x^\lambda}\;,
\end{equation}
holds locally uniformly over $x \in \R^d$ and uniformly over $\psi\in\cB^r$, over $\lambda\in(0,1]$, over
$y \in \supp \psi_x^\lambda$, and over models satisfying $\|\Pi\|_{\gamma,B(x,2)}\leq C$. 
Furthermore, \eqref{eq:standard reco estimate} specifies $\cR f$ uniquely.
\end{theorem}
It is clear from \eqref{eq:standard reco estimate} that the reconstruction operator $\cR$ is local, 
so in particular one can `reconstruct' modelled distributions that only locally lie in $\cD^\gamma$.

\begin{remark} 
While in \cite{H0} in the bound \eqref{eq:standard reco estimate}, $y=x$ is assumed, this version is essentially equivalent: for all $y\in\supp\psi_x^\lambda$, one can simply rewrite $\psi_x^\lambda$ as $\bar\psi_y^{2\lambda}$ with some $\bar\psi\in\cB^r$.

Let us also note that in the literature the use of the notation $\vn{\cdot}$ is slightly inconsistent: sometimes it is defined as in \eqref{eq:vnorm}, 
in some other instances it includes the term $\sup_{x\in\frK}\sup_{l<\gamma}\|f(x)\|_l$.
We will also be guilty of this: while for now, in the unweighted setting, \eqref{eq:vnorm} is convenient since that is what appears in 
the bounds for reconstructions like \eqref{eq:standard reco estimate} above and \eqref{eq:reco bound away negative m} below, the weighted versions of $\vn{\cdot}$ introduced in Section \ref{sec:def} \emph{do} include controls over $\|f(z)\|$.
\end{remark}

\begin{definition}
A continuous bilinear map $\star:T\times T\rightarrow T$ is called a product if,
for $a\in T_\alpha$ and $b\in T_\beta$, one has $a\star b\in T_{\alpha+\beta}$,
and $\bone\star a=a\star\bone$ for all $a\in T$. The products arising in this article
will always be associative and commutative, at least on some sufficiently large subspace.

A pair of sectors $(V,W)$ is said to be $\gamma$-regular with respect to the product $\star$ 
if $(\Gamma a)\star(\Gamma b)=\Gamma(a\star b)$ for all $\Gamma\in G$ and $a\in V_\alpha$, $b\in W_\beta$, 
satisfying $\alpha+\beta<\gamma$. A sector is called $\gamma$-regular, if the pair $(V,V)$ is $\gamma$-regular. 
Given two $T$-valued functions $f$ and $\bar f$, we also denote by $f\star_\gamma\bar f$ the function $x\rightarrow\cQ_\gamma^-(f(x)\star\bar f(x))$.
\end{definition}

%

For $\gamma>0$, a sector $V$ of regularity 0, a product $\star$ such that $V\star V\subset V$, and a smooth function $F:\R^n\rightarrow\R$ one can then define a function $\hat F_\gamma:V^n\rightarrow V$ by setting
\begin{equation}\label{def: F}
\hat F_\gamma(a)=\cQ_\gamma^-\sum_{k}\frac{D^kF(\bar a)}{k!}\tilde a^{\star k},
\end{equation}
where the sum runs over all possible $n$-dimensional multiindices, with the conventions $\bar a=\langle\bone,a\rangle$, $\tilde a=a-\bar a$, $k!=k_1!\cdots k_n!$, $\tilde{a}^{\star k}=\tilde a_1^{\star k_1}\star\cdots\star \tilde a_n^{\star k_n}$ for $k\neq 0$, and  $\tilde{a}^{\star0}=\bone$.

The abstract version of differentiation is quite straightforward.
\begin{definition}
Given a sector $V$, a family of operators $\scD_i: V\rightarrow V$ with $i=1,\ldots,d$
is called an abstract gradient if for every $i$, every $\alpha$ and every $a\in V_\alpha$, one has $\scD_i a\in T_{\alpha-\frs_i}$ 
and $\Gamma \scD_i a=\scD_i \Gamma a$ for all $\Gamma\in G$.

A model $(\Pi,\Gamma)$ is called compatible with $\scD$, if for all $a\in V$, $x\in\R^d$, and for all $i$, it holds that
$$
D_i\Pi_x a=\Pi_x\scD_i a,
$$
where $D_i$ is the usual distributional differentiation in the $i$-th unit direction.
\end{definition}

The final important operation on modelled distribution is the integration against singular kernels, the aim of which is to `lift' convolutions with Green functions to the abstract setting. The first ingredient is the abstract integral operator.
\begin{definition}
Given a sector $V$, a linear map $\cI:V\rightarrow T$ is an abstract integration map of order $\beta>0$ if:
\begin{claim}
\item $\cI(V_\alpha)\subset T_{\alpha+\beta}$ for all $\alpha\in A$.
\item $\cI a=0$ for all $a\in V\cap \bar T$.
\item $\cI\Gamma a-\Gamma\cI a\in \bar{T}$ for all $a\in V$ and $\Gamma\in G$. 
\end{claim}
\end{definition}

In our applications $\beta$ will always be 2, but for most of the analysis the one important property  required of $\beta$ is that for each $\alpha\in A$, $\alpha+\beta\in\mathbb{Z}$ implies $\alpha\in\mathbb{Z}$. In particular, under this assumption, $\cI$ does not produce any components in integer homogeneities. The class of kernels we will want to lift is characterised as follows.
\begin{definition}\label{def: K}
For $\beta>0$ the class $\scK_\beta$ of functions $\R^d\times\R^d\setminus\{x=y\}\rightarrow\R$ consists of elements $K$ that can be decomposed as
$
K(x,y)=\sum_{n\geq0}K_n(x,y)
$,
where the functions $K_n$ have the following properties:
\begin{claim}
\item For all $n\geq0$, $K_n$ is supported on $\{(x,y):\|x-y\|_\frs\leq2^{-n}\}$.
\item For any two multiindices $k$ and $l$,
$
|D_1^kD_2^lK_n(x,y)|\lesssim 2^{n(|\frs|+|k+l|_\frs-\beta)}
$,
where the proportionality constant only depends on $k$ and $l$, but not on $n$, $x$, $y$.
\item For any two multiindices $k$ and $l$, $y\in\R^d$, $i=1,2$, it holds, for all $n\geq0$,
$$
\Big|\int_{R^d}(x-y)^lD_i^kK_n(x,y)dx\Big|\lesssim 2^{-\beta n}
$$
where the proportionality constant only depends on $k$ and $l$.
\item For a given $r>0$,
$
\int_{\R^d}K_n(x,y)P(y)dy=0
$,
for all $n\geq0$, $x\in\R^d$, and every polynomial $P$ of (scaled) degree at most $r$.
\end{claim}
\end{definition}

To introduce the appropriate `remainder' terms, we set $\cJ(x) a$, for $a\in T_\alpha$ as 
\begin{equation}
\cJ(x) a=\sum_{n\geq 0}\cJ^{(n)}(x)a=\sum_{n\geq 0}\sum_{|k|_\frs<\alpha+\beta}\frac{X^k}{k!}(\Pi_x a)(D_1^kK_n(x,\cdot)).
\end{equation}

\begin{definition}
Given a sector $V$ and an abstract integration map $\cI$ acting on $V$ we say that a model $(\Pi,\Gamma)$ realises $K$ for $\cI$ if, for every $\alpha\in A$, every $a\in V_\alpha$, every $x\in\R^d$ one has the identity
$$
\Pi_x\cI a=\int_{\R^d}K(\cdot,z)(\Pi_xa)(dz)-\Pi_x\cJ(x)a.
$$
\end{definition}
Note that both sides are distributions, so the equality should be understood in the distributional sense. 
For $\gamma>0$ we also define an operator $\cN_\gamma$ which maps any $f\in\cD^\gamma$ into a $\bar T$-valued function by 
\begin{equation}
(\cN_\gamma f)(x)=\sum_{n\geq 0}(\cN_\gamma^{(n)} f)(x)=\sum_{n\geq 0}\sum_{|k|_\frs<\gamma+\beta}\frac{X^k}{k!}(\cR f-\Pi_x f(x))(D_1^kK_n(x,\cdot)).
\end{equation}

The key result on a Schauder-type estimate for integration on $\cD^\gamma$ then reads as follows.
\begin{theorem}\label{thm:standard int}
Let $K\in\scK_\beta$ for some $\beta>0$, let $\cI$ be an abstract integration map acting on $V$, and let $(\Pi,\Gamma)$ be a model realising $K$ for $\cI$. Then, for $\gamma>0$, the operator $\cK_\gamma$ defined by
\begin{equation}
(\cK_\gamma f)(x)=\cI f(x)+\cJ(x)f(x)+(\cN_\gamma f)(x),
\end{equation}
maps $\cD^{\gamma}(V)$ into $\cD^{\gamma+\beta}$ and the identity
\begin{equation}
\cR\cK_\gamma f=K\ast\cR f
\end{equation}
holds for every $f\in\cD^\gamma$.
\end{theorem}

\subsection{Preliminaries}

For negative values of $\gamma$, a statement similar to Theorem~\ref{thm: standard reco} still holds, but the
``uniqueness'' part is lost. 
It will be useful for our purposes to have a family of ``reconstruction operators'' defined
similarly to \cite[Eq.~3.38]{H0}, but depending additionally on some small cut-off scale. 
We define the sets
$
\Lambda^n_\frs=\big\{\sum_{j=1}^d2^{-n\frs_j}k_je_j:k_j\in\mathbb{Z}\big\},
$
where $e_j$ is the $j$-th unit vector of $\R^d$, $j=1,\ldots,d$, and we use the notation 
$$
\eta_x^{n,\frs}=2^{-n|\frs|/2}\eta_x^{2^{-n}}
$$
for locally integrable functions $\eta$. Then, as shown in \cite{Daub}, for any integer $r>0$, 
there exist a compactly supported $\cC^r$ function $\varphi$ and a finite family of compactly 
supported $\cC^r$ functions $\Psi$ with the following properties.
\begin{claim}
\item For each $m$, the set
$
\{\varphi_x^{m,\frs}:x\in\Lambda^m_\frs\}\cup\{\psi_x^{n,\frs}:n\geq m,x\in\Lambda_\frs^n,\psi\in\Psi\}
$
forms an orthonormal basis of $L^2(\R^d)$.
\item For every $\psi\in\Psi$ and polynomial $P$ of degree at most $r$, one has
$
\int\psi(x)P(x)dx=0
$.
\end{claim}
In fact much more is known about these functions, but this will suffice for our purposes.
We then set
\begin{equation}\label{eq:Rm}
\cR^mf=\sum_{n\geq m}\sum_{x\in\Lambda_\frs^n}\sum_{\psi\in\Psi}(\Pi_xf(x))(\psi_x^{n,\frs})\psi_x^{n,\frs}+\sum_{x\in\Lambda_\frs^m}(\Pi_xf(x))(\varphi_x^{m,\frs})\varphi_s^{m,\frs}.
\end{equation}
With this notation, we have the following result which is a strengthening of the $\gamma < 0$ part of
\cite[Thm~3.10]{H0}.
\begin{lemma}
Let $\gamma<0$, $m\geq0$ be an integer, $f\in\cD^\gamma(V)$ with a sector $V$ of regularity $\alpha \le 0$. Then $\cR^mf\in\cC^\alpha$ and for every $r>|\alpha|$ there exists $c$ such that, uniformly over $\eta\in\cB^r$ and $\lambda\in(0,1]$ and locally uniformly over $x$, one has the bound
\begin{equation}\label{eq:reco bound away negative m}
|(\cR^mf-\Pi_xf(x))(\eta_x^\lambda)|\lesssim\lambda^{\gamma-\alpha}(\lambda\wedge2^{-m})^\alpha
\vn{f}_{\gamma,B(x,c\lambda+2^{-m})}\;.
\end{equation}
\end{lemma}
\begin{proof}
The fact that $\cR^mf\in\cC^\alpha$ is immediate, since the above construction only differs by a $\cC^r$ function from the reconstruction operator given in \cite[Eq.~3.38]{H0}. 
To show \eqref{eq:reco bound away negative m}, we assume without loss of generality that 
$\vn{f}_{\gamma,B(x,\lambda+2^{-m})} \le 1$. Note first that
$$
|(\psi_y^{n,\frs},\eta_x^\lambda)|\lesssim 2^{n|\frs|/2} \bigl(2^n \lambda \vee 1\bigr)^{-|\frs|-r}\;,
$$
and that $(\psi_y^{n,\frs},\eta_x^\lambda)=0$ for $\|x-y\|_\frs\geq\lambda+c2^{-n}$ for some fixed constant $c$. 
We also have, for $n\geq m$, and for $\|x-y\|_\frs\leq\lambda+2^{-n}$,
\begin{align}
|(\cR^m f-\Pi_x f(x))(\psi_y^{n,\frs})|&=|(\Pi_yf(y)-\Pi_x f(x))(\psi_y^{n,\frs})|
=|(\Pi_y(f(y)-\Gamma_{yx} f(x))(\psi_y^{n,\frs})|
\nonumber\\&
\lesssim\sum_{l<\gamma}\|x-y\|_\frs^{\gamma-l}2^{-n|\frs|/2-nl}.\label{eq:Rm4}
\end{align}
Denoting the first (triple) sum in \eqref{eq:Rm} by $\cR^m_0$, and the projection of $\Pi_xf(x)$ to $\text{span}\{\psi_y^{n,\frs}:y\in\Lambda_\frs^n,n\geq m\}$ by $(\Pi_xf(x))_0$, we can write
\[
|(\cR^m_0f-(\Pi_xf(x))_0)(\eta_x^\lambda)|=\sum_{n\geq m}\sum_{y\in\Lambda_\frs^n}\sum_{\psi\in\Psi}|(\cR^mf-\Pi_xf(x))(\psi_y^{n,\frs})(\psi_y^{n,\frs},\eta_x^\lambda)|=:\sum_{n\geq m}I_n.
\]
We consider the cases $2^{-m}\gtrless\lambda$ separately. If $\lambda< 2^{-m}$, then considering that for $2^{-n}\leq\lambda$, the number of nonzero terms in the sum over $y\in\Lambda_\frs^n$ is of order $\lambda^{|\frs|}2^{n|\frs|}$, by estimating each of them using the bounds above, we have
\begin{equation}\label{eq:Rm0}
\sum_{2^{-n}\leq\lambda}I_n\lesssim
\sum_{2^{-n}\leq\lambda}\lambda^{|\frs|}2^{n|\frs|}2^{-n|\frs|/2-nr}\lambda^{-|\frs|-r}\sum_{l<\gamma}(\lambda+2^{-n})^{\gamma-l}2^{-n|\frs|/2-nl}\lesssim\lambda^\gamma,
\end{equation}
due to $r+l>0$. On the other hand, for $\lambda<2^{-n}$, the number of nonzero terms in the sum over $y$ is of order $1$, so we can write
\begin{equation}\label{eq:Rm1}
\sum_{\lambda<2^{-n}\leq2^{-m}}I_n\lesssim\sum_{\lambda<2^{-n}\leq2^{-m}}2^{n|\frs|/2}\sum_{l<\gamma}(\lambda+2^{-n})^{\gamma-l}2^{-n|\frs|/2-nl}\lesssim\lambda^\gamma,
\end{equation}
where we used the negativity of $\gamma$, and this bound is of the required order.

In the case $2^{-m}\leq\lambda$, then similarly to before
\begin{align}
\sum_{n\geq m}I_n&\lesssim
\sum_{n\geq m}\lambda^{|\frs|}2^{n|\frs|}2^{-n|\frs|/2-nr}\lambda^{-|\frs|-r}\sum_{l<\gamma}(\lambda+2^{-n})^{\gamma-l}2^{-n|\frs|/2-nl}\nonumber
\\&\lesssim\sum_{l<\gamma}2^{-m(r+l)}\lambda^{\gamma-l-r}\leq\sum_{l<\gamma}2^{-ml}\lambda^{\gamma-l},\label{eq:Rm2}
\end{align}
and since $l\geq\alpha$, this gives the required bound.

For the second sum in \eqref{eq:Rm}, denoted for the moment by $\cR^m_1$ and the projection of $\Pi_xf(x)$ to $\text{span}\{\varphi_y^{m,\frs}:y\in\Lambda_\frs^m\}$, denoted by $(\Pi_xf(x))_1$, we proceed similarly. This time, one has
$$
|(\varphi_y^{m,\frs},\eta_x^\lambda)|\lesssim  2^{m|\frs|/2} \bigl(2^m \lambda \vee 1\bigr)^{-|\frs|}\;,
$$
and $(\varphi_y^{m,\frs},\eta_x^\lambda)=0$ for $\|x-y\|_\frs\geq\lambda+c2^{-m}$, that is, for all but of order $2^{m|\frs|}\lambda^{|\frs|}$ instances of $y\in\Lambda_{\frs}^m$ in the case $2^{-m}\leq\lambda$, and for all but of order $1$ instances of $y\in\Lambda_\frs^m$ in the case $\lambda<2^{-m}$. The quantity $(\cR^m f-\Pi_x f(x))(\varphi_y^{m,\frs})$ can 
then be bounded exactly as in \eqref{eq:Rm4}. Combining these bounds, we arrive at
\begin{align}
|(\cR^m_1f-(\Pi_xf(x))_1)(\eta_x^\lambda)|&=\sum_{y\in\Lambda_\frs^m}|(\cR^mf-\Pi_yf(y))(\varphi_y^{m,\frs})(\varphi_y^{m,\frs},\eta_x^\lambda)|
\nonumber\\
&\lesssim2^{m|\frs|/2}\sum_{l<\gamma}(\lambda+2^{-m})^{\gamma-l}2^{-m|\frs|/2-ml}\nonumber
\\&
\lesssim \lambda^{\gamma-\alpha}2^{-m\alpha} \vee 2^{-m\gamma}
\lesssim \lambda^{\gamma-\alpha}2^{-m\alpha} \vee \lambda^\gamma\;,
\end{align}
as required. Here, the last inequality comes from the fact that $\gamma < 0$ and that the second term
dominates when $2^{-m} \ge \lambda$, so that $2^{-m\gamma} \le \lambda^\gamma$.
\end{proof}

Next we recall some results on extending dual elements of a space of smooth functions that are supported away from a submanifold, to distributions, at least locally. This is essentially the content of \cite[Prop.~6.9]{H0}, but we slightly reformulate the statements in order to fit the needs of Section~\ref{subsec:reco} below better.

Whenever here and in the sequel we refer to a `boundary' $P$, we mean the following. Assume that $\R^d$ is decomposed as 
$\R^d=\R^{d_1}\times\cdots\times\R^{d_m}$, such that $\frs_1=\cdots=\frs_{d_1}$, $\frs_{d_1+1}=\cdots=\frs_{d_2}$, etc. 
We then assume $P$ to be of the form
$$
P=M_1\times\cdots\times M_m
$$
where each $M_i$ is either $\R^{d_i}$ or is a piecewise $\cC^1$ boundary of a domain, satisfying the strong cone condition. Denoting the codimension of $M_i$ by $m_i$, the codimension of $P$ is then defined to be $\sum_{i=1}^m m_i\frs_{d_{i-1}+1}$, with the convention $d_0=0$. We will need the following version of a well-known 
``folklore'' fact:

\begin{proposition}\label{prop:extension}
Let $P$ be a boundary of codimension $\frm$, $D\subset\R^d$ be a bounded domain and let $\xi$ be an element of the dual of smooth functions compactly supported in $D\setminus P$. Suppose furthermore that $0\geq\alpha>-\frm$ and for an integer $r>|\alpha|$ one has
\begin{equation}\label{eq:extension}
|\xi(\psi_x^\lambda)|\lesssim\lambda^\alpha
\end{equation}
uniformly over $x \in D\setminus P$, over $\psi\in\cB^r$, and over $\lambda\in(0,1]$ satisfying
furthermore $2\lambda\leq d_\frs(x,P)$ and $\supp\psi_x^\lambda\subset D$. 
Then there exists a unique element $\xi'$ in the dual of smooth functions compactly supported in $D$ that agrees with $\xi$ on test functions supported away from $P$ and for which the bound \eqref{eq:extension} holds with $\xi'$ in place of $\xi$, uniformly in $x$, in $\psi\in\cB^r$, and in $\lambda\in(0,1]$ satisfying $\supp \psi_x^\lambda\subset D$.
\end{proposition}

\begin{proof}
By considering a suitable partition of unity,
thanks to the strong cone condition,
 we see that for any compact set $K \subset D$ with diameter $\lambda$,
and any $n$ with $2^{-n} \le \lambda$, we can find  
smooth functions $\Phi_n \colon K \to [0,1]$ such that
$\Phi_n(y) = 1$ if $d_\frs(y,P) \in [2^{1-n},2^{2-n}]$, $\Phi_n(y) = 0$ if $d_\frs(y,P) \notin [2^{-n},2^{3-n}]$, and satisfying the following
property.
For every $n \ge 1$, one can find sequences $\{x_k\}_{k=1}^{N}$ with $N \le C \lambda^{|\frs|-\frm} 2^{(|\frs|-\frm)n}$ 
and functions $\phi_k , \tilde \phi_k \in \cB^r$ such that, setting $\mu = 2^{-n}$, one has
\begin{equ}[e:propPhi]
\Phi_n = \mu^{|\frs|}\sum_{k=1}^N \phi_{k,x_k}^\mu\;,\qquad
\Phi_n - \Phi_{n+1} = \mu^{|\frs|}\sum_{k=1}^N \tilde \phi_{k,x_k}^\mu\;.
\end{equ}

Fix now a test function of the type $\psi_x^\lambda$ with support $K \subset D$, then the sequence  
$\xi(\psi_x^\lambda (1-\Phi_n))$ is Cauchy since
\begin{equs}
|\xi(\psi_x^\lambda (\Phi_{n+1}-\Phi_n))| &\le \sum_{k=1}^N\mu^{|\frs|} \xi(\tilde\psi_x^\lambda \phi_{k,x_k}^\mu)
\lesssim \lambda^{-|\frs|} N \mu^{\alpha + |\frs|}
\le C \lambda^{-|\frs|} \lambda^{|\frs|-\frm} 2^{(|\frs|-\frm)n} \mu^{\alpha+|\frs|} \\
& = C \lambda^{-m} 2^{-(\frm + \alpha)n}\;\label{corr},
\end{equs}
where in the second inequality we made use of the bound \eqref{eq:extension}. Thanks to the assumption $\alpha+\frm>0$, the right-hand side of \eqref{corr} which converges to $0$ exponentially fast, as claimed.
 The same bound also shows that the limit is bounded
by some constant times $\lambda^{-\alpha}$ as required. The uniqueness of $\xi'$ follows in a similar way
by comparing $\xi'(\psi (1-\Phi_n))$ to $\xi'(\psi)$ and using the first identity of \eqref{e:propPhi}.
\end{proof}

\section{Definition of \texorpdfstring{$\cD_P^{\gamma,w}$}{DPgw} and basic properties}\label{sec:def}

Our main tool for dealing with domains is to introduce spaces analogous to the spaces $\CD^{\gamma,\eta}$
used in \cite{H0} to deal with initial conditions, but allowing for blow-ups at the boundary of 
the domain as well. One subtlety arises in the handling of the ``double singularity'' arising
on the boundary at time $0$.  
Let $P_0$ and $P_1$ be two fixed boundaries with respective codimensions $\frm_0$, $\frm_1$
and such that $P_\cap = P_0 \cap P_1$ is itself a boundary of codimension $\frm=\frm_0+\frm_1$.
We also write $P=P_0\cup P_1$ and we assume that $P$ satisfies the (uniform) cone condition, which 
forces the two boundaries to intersect in a transverse manner. 
For $i=1,2$, denote
$$
|x|_{P_i}=1\wedge d_\frs(x,P_i),\quad\quad|x,y|_{P_i}=|x|_{P_i}\wedge|y|_{P_i},
$$
and for any compact set $\frK$,
$$
\frK_P=\{(x,y)\in(\frK\setminus P)^2:\,x\neq y\,\,\,\text{and}\,\,\,2\|x-y\|_\frs\leq|x,y|_{P_0}\wedge|x,y|_{P_1}\}.
$$
To slightly ease notation, in the following $w$ will always stand for an element in $\R^3$, with coordinates $w=(\eta,\sigma,\mu)$, corresponding to exponents for the `weights' at $P_0$, $P_1$, and their intersection, respectively.
\begin{remark} It might be at first sight surprising to have not two, but three different orders of singularity. While in the subsequent calculus the use of exponent $\mu$ will become clear, it is worth mentioning a simple example when the singularities at the different boundaries do not in any way determine the one at the intersection: Consider the solution of $\partial_t u=\Delta u$, $u_0\equiv 1$, with $0$ Dirichlet boundary conditions on some domain $D$. Then, while away from the ``corner'' $\{(0,x):x\in\partial D\}$, all derivatives of $u$ are continuous up to both the temporal and the spatial boundaries, the $k$-th derivative exhibits a blow-up of order $|k|_\frs$ at the corner.
\end{remark}

\begin{definition}
Let $V$ be a sector, $\gamma>0$ and $w=(\eta,\sigma,\mu)\in\R^3$. Then the space $\cD_P^{\gamma,w}(V)$ consists of all functions $f:\R^d\setminus P\rightarrow V_\gamma^-$ such that for every compact set $\frK\subset\R^d$ one has
\begin{align}
\vn{f}_{\gamma,w;\frK}&:=
\sup_{(x,y)\in\frK_P}\sup_{l<\gamma}
\frac{\|f(x)-\Gamma_{xy}f(y)\|_l}
{\|x-y\|_{\frs}^{\gamma-l}|x,y|_{P_0}^{\eta-\gamma}|x,y|_{P_1}^{\sigma-\gamma}(|x,y|_{P_0}\vee|x,y|_{P_1})^{\mu-\eta-\sigma+\gamma}}\nonumber
\\
&\quad+\sup_{x\in\frK:\,0<|x|_{P_0}\leq|x|_{P_1}}\sup_{l<\gamma}\frac{\|f(x)\|_l}{|x|_{P_1}^{\mu-l}\left(\tfrac{|x|_{P_0}}{|x|_{P_1}}\right)^{(\eta-l)\wedge0}}\nonumber
\\
&\quad+\sup_{x\in\frK:\,0<|x|_{P_1}\leq|x|_{P_0}}\sup_{l<\gamma}\frac{\|f(x)\|_l}{|x|_{P_0}^{\mu-l}\left(\tfrac{|x|_{P_1}}{|x|_{P_0}}\right)^{(\sigma-l)\wedge0}}\,\,<\infty.\label{def:spaces}
\end{align}
The sum of the second and third term above will also be denoted by $\|f\|_{\gamma,w;\frK}$.
Similarly to before, these spaces do depend on the model, but if no confusion can arise, this dependence will not be denoted. 
For two models $(\Pi,\Gamma)$ and $(\bar\Pi,\bar\Gamma)$, and for $f\in\cD_P^{\gamma,w}(V;\Gamma)$ and $\bar f\in\cD_P^{\gamma,w}(V;\bar \Gamma)$, we also set
\begin{align*}
\vn{f;\bar f}_{\gamma,w;\frK}&=\sup_{(x,y)\in\frK_P}\sup_{l<\gamma}
\frac{\|f(x)-\bar f(x)-\Gamma_{xy}f(y)+\bar\Gamma_{xy}\bar f(y)\|_l}
{\|x-y\|_{\frs}^{\gamma-l}|x,y|_{P_0}^{\eta-\gamma}|x,y|_{P_1}^{\sigma-\gamma}(|x,y|_{P_0}\vee|x,y|_{P_1})^{\mu-\eta-\sigma+\gamma}}
\\
&\quad+\|f-\bar f\|_{\gamma,w;\frK}.
\end{align*}
\end{definition}

This notation is slightly ambiguous since the knowledge of $P$ does of course not imply the knowledge of 
$P_0$ and $P_1$. One should therefore really interpret the instance of $P$ appearing in $\cD_{P}^{\gamma,w}$
as meaning $P = \{P_0,P_1\}$ rather than $P = P_0 \cup P_1$, which is used whenever we view $P$
as a subset of $\R^d$.
It will also sometimes be useful to consider functions in $\cD_{P}^{\gamma,w}$ that are slightly better behaved when
approaching one of the two boundaries. This is the purpose of the following definition.
\begin{definition}
We denote by $\cD_{P,\{0\}}^{\gamma,w}$ the set of those elements $f \in \cD_{P}^{\gamma,w}$ for which the map $x\mapsto \cQ_\eta^- f(x)$ extends continuously to $\R^d\setminus P_1$ in such a way that $\cQ_\eta^-f(x)=0$ for all $x\in P_0\setminus P_1$. 
The space $\cD_{P,\{1\}}^{\gamma,w}$ is defined analogously. 
Finally, writing $\frK_0 = \{x\in\frK:\, 0<|x|_{P_0}\leq|x|_{P_1}\}$ and similarly for $\frK_1$,
we set
\begin{align*}
\bn{f}_{\gamma,w,\{0\};\frK}&=\sup_{x\in\frK_0}\sup_{l<\gamma}\frac{\|f(x)\|_l}{|x|_{P_1}^{\mu-l}\left(\tfrac{|x|_{P_0}}{|x|_{P_1}}\right)^{\eta-l}}\nonumber
+\sup_{x\in\frK_1}\sup_{l<\gamma}\frac{\|f(x)\|_l}{|x|_{P_0}^{\mu-l}\left(\tfrac{|x|_{P_1}}{|x|_{P_0}}\right)^{(\sigma-l)\wedge0}},
\end{align*}
and also define $\bn{f}_{\gamma,w,\{1\};\frK}$ in the same way, but with the exponents $\eta-l$ 
and $(\sigma-l)\wedge 0$ replaced
by $(\eta-l)\wedge 0$ and $\sigma-l$ respectively.
\end{definition}

We shall assume throughout the article that these exponents satisfy $\eta\vee\sigma\vee\mu\leq\gamma$.

\begin{remark}\label{remark:mu}
Denoting the regularity of the sector $V$ by $\alpha$, the definition is set up so that, when $\mu\leq\alpha$ and there exists an $x$ with $|x|_{P_0}\sim|x|_{P_1}\sim 1$ and
$
\sup_{l<\gamma}\|f(x)\|_{l}\sim1
$,
then the first term in \eqref{def:spaces} bounds the second and third. For $\mu>\alpha$, one would actually need to add $|x|_{P_1}^{(\mu-l)\wedge0}$ to the denominator in the second term and $|x|_{P_0}^{(\mu-l)\wedge0}$ in the third. As this would 
make the calculations significantly longer, we omit this modification and deal with the slight difficulties arising from this 
restriction later.

\end{remark}

\begin{proposition}\label{prop:1}
Let $V$ be a sector of regularity $\alpha$, and $f\in\cD_{P,\{1\}}^{\gamma,w}(V)$. Suppose furthermore that $\frK$ is a compact set such that for each $x\in\frK$ the line connecting $x$ and the closest point to $x$ on $P_1$ is contained in $\frK$. Then it holds that 
\begin{equation}\label{eq:vmi0}
\bn{f}_{\gamma,w,\{1\};\frK}\lesssim\vn{f}_{\gamma,w;\frK}.
\end{equation}
If $(\bar \Pi,\bar\Gamma)$ is another model for $\scT$ and $\bar f\in\cD_{P,\{1\}}^{\gamma,w}(V;\bar\Gamma)$, then one has
\begin{equation}\label{eq:vmi1}
\bn{f-\bar f}_{\gamma,w,\{1\};\frK}\lesssim\vn{f;\bar f}_{\gamma,w;\frK}+\|\Gamma-\bar\Gamma\|_{\gamma;\frK}(\vn{f}_{\gamma,w;\frK}+\vn{\bar f}_{\gamma,w;\frK}),
\end{equation}
and, for any $\kappa \in [0,1]$,
\begin{equation}\label{eq:vmi2}
\vn{f;\bar f}_{\bar\gamma,\bar w;\frK}\lesssim\bn{f-\bar f}_{\bar \gamma,w,\{1\};\frK}^\kappa(\vn{f}_{\gamma,w;\frK}+\vn{\bar f}_{\gamma,w;\frK})^{1-\kappa},
\end{equation}
where $\bar\gamma=(1-\kappa)\gamma+\kappa\alpha$ and $\bar w=(\bar \eta,\sigma,\mu)$ with $\bar\eta=\eta+\kappa((\alpha-\eta)\wedge0)$.
\end{proposition}
\begin{proof}
We prove separately for $\frK_i=\frK\cap\{|x|_{P_i}\leq|x|_{P_{1-i}}\}$.

For $\frK_1$, further introducing $\frK_1^n=\frK_1\cap\{2^{-n}\leq|x|_{P_0}\leq 2^{-n+1}\}$, the bounds for $\frK_1^n$ in place of $\frK$ follow immediately from Lemmas 6.5 and 6.6, \cite{H0}, uniformly in $n$. Since there is no dependence on $n$ in the bounds, and for any pair $(x,y)\in(\frK_1)_P$, the indices $n_x$ and $n_y$ for which $x\in\frK_1^{n_x}$, $y\in\frK_1^{n_y}$, differ by at most $1$, the estimates carry through for $\frK_1$.

For $\frK_0$, the bounds \eqref{eq:vmi0} and \eqref{eq:vmi1} are trivial. As for \eqref{eq:vmi2}, we have
$$
\|f(x)-f(x)-\Gamma_{xy}f(y)+\bar\Gamma_{xy}f(y)\|_l\leq(\vn{f}_{\gamma,w;\frK_0}+\vn{\bar f}_{\gamma,w;\frK_0})\|x-y\|_\frs^{\gamma-l}|x,y|_{P_0}^{\eta-\gamma}|x,y|_{P_1}^{\mu-\eta}
$$
as well as
$$
\|f(x)-f(x)-\Gamma_{xy}f(y)+\bar\Gamma_{xy}f(y)\|_l\lesssim\bn{f-\bar f}_{\gamma,w,\{1\};\frK_0}|x,y|_{P_1}^{\mu-l}\left(\frac{|x,y|_{P_0}}{|x,y|_{P_1}}\right)^{(\eta-l)\wedge0}.
$$
Therefore, we can bound the quantity $\|f(x)-f(x)-\Gamma_{xy}f(y)+\bar\Gamma_{xy}f(y)\|_l$ by the right-hand side of \eqref{eq:vmi2} times
\begin{align*}
\|x-y&\|_\frs^{(1-\kappa)(\gamma-l)}|x,y|_{P_1}^{(1-\kappa)(\mu-\eta)+\kappa(\mu-l)-\kappa((\eta-l)\wedge0)}|x,y|_{P_0}^{(1-\kappa)(\eta-\gamma)+\kappa((\eta-l)\wedge0)},
\\
&\lesssim \|x-y\|_\frs^{\bar\gamma-l}\|x-y\|_{\frs}^{\kappa(l-\alpha)}|x,y|_{P_1}^{\mu-\eta-\kappa(l-\eta+(\eta-l)\wedge0)}|x,y|_{P_0}^{\eta-\bar\gamma+\kappa(\alpha-\eta+(\eta-l)\wedge0)}.
\end{align*}
Considering that $\|x-y\|_\frs\leq|x,y|_{P_0}$ and that the minimum value of $a_l:=(l-\eta+(\eta-l)\wedge0)$ is $a_\alpha=(\alpha-\eta)\wedge 0$, we can estimate the right-hand side above by
$$
\|x-y\|_\frs^{\bar\gamma-l}|x,y|_{P_1}^{\mu-\bar\eta}|x,y|_{P_1}^{\kappa(a_\alpha-a_l)}|x,y|_{P_0}^{\bar\eta-\bar\gamma}|x,y|_{P_0}^{\kappa(a_l-a_\alpha)},
$$
and since we are in the situation $|x,y|_{P_0}\leq|x,y|_{P_1}$, this gives the required bound. The estimate for $\|f(x)-\bar f(x)\|_l$, $x\in\frK_0$ is straightforward, since one has the bound
$$
\frac{\|f(x)-\bar f(x)\|_l}{|x|_{P_1}^{\mu-l}\left(\tfrac{|x|_{P_0}}{|x|_{P_1}}\right)^{(\eta-l)\wedge0}}\lesssim \bn{f-\bar f}_{ \gamma,w,\{1\};\frK}\wedge(\vn{f}_{\gamma,w,\frK}+\vn{\bar f}_{\gamma,w;\frK})\;,
$$
thus concluding the proof.
\end{proof}

\begin{proposition}\label{prop3}
If $f\in\cD_{P,\{1\}}^{\gamma,w}$ then, for any $\delta>0$ and compact $\frK\subset\{|x|_{P_1}\vee\delta\leq|x|_{P_0} \le 2\delta\}$, it holds that
\begin{equation}\label{eq:bode1}
\vn{\hat f}_{\sigma;\frK}\lesssim\delta^{\mu-\sigma}\vn{f}_{\gamma,w;\frK}\;,
\end{equation}
with $\hat f = \cQ_\sigma^- f$. In particular, away from $P_0$, $\hat f$ locally belongs to $\cD^\sigma$.
\end{proposition}
\begin{proof}
We assume without loss of generality $\vn{f}_{\gamma,w;\frK}\leq 1$.
For $2\|x-y\|_\frs\leq|x,y|_{P_1}$, simply by the definition of the spaces $\cD_{P}^{\gamma,w}$ we get
\begin{equation}\label{eq:bode2}
\frac{\|\hat f(x)-\Gamma_{xy}\hat f(y)\|_{l}}{\|x-y\|_\frs^{\sigma-l}}\lesssim\|x-y\|_\frs^{\gamma-\sigma}|x,y|_{P_1}^{\sigma-\gamma}|x,y|_{P_0}^{\mu-\sigma}\leq\delta^{\mu-\sigma}.
\end{equation}
since $\sigma\leq\gamma$. In the case $|x,y|_{P_1}\leq 2\|x-y\|_\frs$, then noting that $|x|_{P_1}\vee|y|_{P_1}\leq3\|x-y\|_\frs$ we can write, using the estimate \eqref{eq:vmi0} again
\begin{align*}
\|\hat f(x)-\Gamma_{xy}\hat f (y)\|_l&\leq \|\hat f(x)\|_l+\sum_{l\leq m<\sigma}\|x-y\|_\frs^{m-l}\|\hat f(y)\|_m
\\
&\leq \delta^{\mu-\sigma}|x|_{P_1}^{\sigma-l}+\sum_{l\leq m<\sigma}\|x-y\|_\frs^{m-l}\delta^{\mu-\sigma}|y|_{P_1}^{\sigma-m}
\lesssim\delta^{\mu-\sigma}\|x-y\|^{\sigma-l},
\end{align*}
as required.

The fact that $\hat f$ is locally in $\cD^\sigma$ then follows, since on 
$\{\delta\leq|x|_{P_0}\leq|x|_{P_1}\}$, $ f$ actually belongs to $\cD^\gamma$, 
so its projection $\hat f$ belongs to $\cD^\sigma$,
and $\delta>0$ was arbitrary.
\end{proof}

%

\begin{remark}
One simplification that we will often use is based on the fact that for pairs $(x,y)\in\frK_P$, we have
$$
|x|_{P_i}\sim|y|_{P_i}\sim|x,y|_{P_i}
$$
for $i=0,1$. As a consequence, in the proofs of Section~\ref{sec:calculus} below we will repeatedly interchange the above quantities without much explanation. Also, for such pairs, even though $|x|_{P_0}\leq|x|_{P_1}$ does not imply $|y|_{P_0}\leq|y|_{P_1}$ or $|x,y|_{P_0}\leq|x,y|_{P_1}$, it holds that
$$
\|f(y)\|_l\lesssim\|f\|_{\gamma,w,\frK}|y|_{P_1}^{\mu-l}\left(\frac{|y|_{P_0}}{|y|_{P_1}}\right)^{(\eta-l)\wedge0},
$$
and
$$
\|f(x)-\Gamma_{xy}f(y)\|_l\lesssim\vn{f}_{\gamma,w,\frK}\|x-y\|_\frs^{\gamma-l}|x,y|_{P_0}^{\eta-\gamma}|x,y|_{P_1}^{\mu-\eta}.
$$
This, and the corresponding symmetric implications (swapping the roles of $P_0$ and $P_1$), 
will also often be used.
\end{remark}

\section{Calculus of the spaces \texorpdfstring{$\cD_P^{\gamma,w}$}{DPgw}}\label{sec:calculus}

In order to reformulate our stochastic PDEs as fixed point problems in $\cD_P^{\gamma,w}$, 
one first needs to know how the standard operations like multiplication, differentiation, 
or convolution with singular kernels, act on these spaces. The aim of this section is to 
recover the calculus of \cite{H0} in the present context.

\begin{remark}
This of course means that repetition of arguments to a certain degree is inevitable. We shall try to 
minimise the overlap and concentrate on the aspects that are different due to the additional 
weights and don't just follow trivially from \cite{H0}. This in particular applies to the 
continuity statements: since the space of models is not linear, boundedness of the operations do 
not imply their continuity. However, in practice they usually follow from the same principles, with 
an added level of notational inconvenience. We therefore only give the complete proof of 
continuity for the multiplication, after which the reader is hopefully convinced that obtaining 
the other similar continuity results is a lengthy but straightforward combination of the 
corresponding arguments in \cite{H0} and the treatment of the additional weights as described 
in the `boundedness' part of the corresponding statements. Alternatively, the continuity statements
can also be obtained by using the trick introduced in the proof of \cite[Prop.~3.11]{HP15}, which
allows to some extent to ``linearise'' the space of models.
\end{remark}

\begin{remark}
Let us mention an important point on how integration against singular kernels will be handled.
While Green's functions of boundary value problems are not translation invariant, they typically can be decomposed into a translation invariant part and a smooth one, which however is singular at the boundary. The most simple example of this is the $1+1$-dimensional Neumann heat kernel on $(\R_+)^2$:
\begin{equ}
G((t,x),(s,y))=\frac{1}{\sqrt{4\pi(t-s)}}\Big(e^{-\frac{(x-y)^2}{4(t-s)}}+e^{-\frac{(x+y)^2}{4(t-s)}}\Big),
\end{equ}
for a more general discussion see Example \ref{example} below.
The advantage of such a decomposition is that only the former part plays a role in constructing the regularity structure itself and the corresponding admissible models, for which one can use the general machinery of \cite{BHZ, CH, H0}. Integration against the latter part simply produces functions described by polynomial symbols, albeit with blow-ups at the boundaries which need to be sufficiently controlled.
\end{remark}

\subsection{Multiplication}
\begin{lemma}\label{lem:mult}
For $i=1,2$, let $f_i\in \cD_P^{\gamma_i,w_i}(V_i)$ with $\gamma_i > 0$, where $V_i$ is a sector of regularity $\alpha_i \le 0$. Suppose furthermore that the pair $(V_1,V_2)$ is $\gamma:=(\gamma_1+\alpha_2)\wedge(\gamma_2+\alpha_1)$-regular with respect to the product $\star$. Then $f:=f_1\star_\gamma f_2$ belongs to $\cD_P^{\gamma,w}$, where
$
w=(\eta,\sigma,\mu)
$
with $\mu=\mu_1+\mu_2$ and
\begin{equs}
\eta &=(\eta_1+\alpha_2)\wedge(\eta_2+\alpha_1)\wedge(\eta_1+\eta_2)\;,\\
\sigma&=(\sigma_1+\alpha_2)\wedge(\sigma_2+\alpha_1)\wedge(\sigma_1+\sigma_2)\;.
\end{equs}
Moreover, if $(\bar\Pi,\bar\Gamma)$ is another model for $\scT$, and $g_i\in\cD_P^{\gamma_i,w_i}(V_i;\bar\Gamma)$ for $i=1,2$, then, for $g=g_1\star_\gamma g_2$ and any $C>0$
\begin{equation}\label{eq:multi cont}
\vn{f;g}_{\gamma,w;\frK}\lesssim\vn{f_1;g_1}_{\gamma_1,w_1;\frK}+\vn{f_2;g_2}_{\gamma_2,w_2;\frK}+\|\Gamma-\bar\Gamma\|_{\gamma_1+\gamma_2;\frK},
\end{equation}
holds uniformly in $f_i$ and $g_i$ with $\vn{f_i}_{\gamma_i,w_i;\frK}+\vn{g_i}_{\gamma_i,w_i;\frK}\leq C$ and models with $\|\Gamma\|_{\gamma_1+\gamma_2;\frK}+\|\bar \Gamma\|_{\gamma_1+\gamma_2;\frK}\leq C$.
\end{lemma}

\begin{proof}
We fix a compact $\frK$ and assume, without loss of generality, that both $f_1$ and $f_2$ are of norm $1$ on $\frK$. Then, for $|x|_{P_0}\leq|x|_{P_1}$ and $l < \gamma$,
\begin{align*}
\|f(x)\|_l&\leq\sum_{l_1+l_2=l}\|f_1(x)\|_{l_1}\|f_2(x)\|_{l_2}
\leq\sum_{l_1+l_2=l}|x|_{P_1}^{\mu_1+\mu_2-l_1-l_2}\left(\frac{|x|_{P_0}}{|x|_{P_1}}\right)^{(\eta_1-l_1)\wedge0+(\eta_2-l_2)\wedge0}
\\
&\leq|x|_{P_1}^{\mu-l}\sum_{l_1+l_2=l}\left(\frac{|x|_{P_0}}{|x|_{P_1}}\right)^{{-l+\eta_1\wedge l_1+\eta_2\wedge l_2}}.
\end{align*}
It remains to notice that, since for $i=0,1$, $l_i\geq\alpha_i$, we have $\eta_1\wedge l_1+\eta_2\wedge l_2\geq \eta\wedge l$, by construction, and hence
$$
\|f(x)\|_l\lesssim |x|_{P_1}^{\mu-l}\left(\frac{|x|_{P_0}}{|x|_{P_1}}\right)^{(\eta-l)\wedge0}.
$$

Next we bound $f(x)-\Gamma_{xy}f(y)$. As usual, we assume $|x,y|_{P_0}\leq|x,y|_{P_1}$.  
For $l < \gamma$, the triangle inequality yields
\begin{align}
\|f(x)-\Gamma_{xy}f(y)\|_l&\leq\|\Gamma_{xy}f(y)-(\Gamma_{xy}f_1(y))\star(\Gamma_{xy}f_2(y))\|_l
\nonumber\\
&\quad+\|(\Gamma_{xy}f_1(y)-f_1(x))\star(\Gamma_{xy}f_2(y)-f_2(x)\|_l
\nonumber\\
&\quad+\|(\Gamma_{xy}f_1(y)-f_1(x))\star f_2(x)\|_l
\nonumber\\
&\quad+\|f_1(x)\star\big(\Gamma_{xy}f_2(y)-f_2(x)\big)\|_l.\label{multi1}
\end{align}
Thanks to the $\gamma$-regularity of $(V_1,V_2)$, the first term in this expression can be bounded by
\begin{align}
A&:=\|\Gamma_{xy}f(y)-(\Gamma_{xy}f_1(y))\star(\Gamma_{xy}f_2(y))\|_l
\nonumber
\\&
\leq\Big\|\sum_{m+n\geq\gamma}(\Gamma_{xy}\cQ_mf_1(y))\star(\Gamma_{xy}\cQ_nf_2(y))\Big\|_l
\nonumber
\\
&\leq\sum_{m+n\geq\gamma}\sum_{\beta_1+\beta_2=l}\|\Gamma_{xy}\cQ_mf_1(y)\|_{\beta_1}\|\Gamma_{xy}\cQ_nf_2(y)\|_{\beta_2}
\nonumber
\\
&\leq\sum_{m+n\geq\gamma}\sum_{\beta_1+\beta_2=l}\|\Gamma\|_{\gamma_1+\gamma_2}^2\|f_1(y)\|_m\|f_2(y)\|_n\|x-y\|_\frs^{m+n-\beta_1-\beta_2}.\label{multi3}
\end{align}
The factor $\|\Gamma\|_{\gamma_1+\gamma_2}^2$ can course be incorporated into the proportionality constant, 
but it will be useful in the sequel to view the dependence on it as above. We can continue by writing
\begin{align}
A&\lesssim\sum_{m+n\geq\gamma}\|x-y\|_\frs^{m+n-l}\|f_1(y)\|_m\|f_2(y)\|_n
\nonumber\\
&\leq\|x-y\|_\frs^{\gamma-l}\sum_{m+n\geq\gamma}\|x-y\|_\frs^{m+n-\gamma}
|y|_{P_1}^{\mu_1+\mu_2-m-n}
\left(\frac{|y|_{P_0}}{|y|_{P_1}}\right)^{(\eta_1-m)\wedge0+(\eta_2-n)\wedge0}
\nonumber\\
&\leq\|x-y\|_\frs^{\gamma-l}|y|_{P_1}^{\mu}|y|_{P_0}^{-\gamma}\sum_{m+n\geq\gamma}|y|_{P_0}^{m+n}|y|_{P_1}^{-m-n}\left(\frac{|y|_{P_0}}{|y|_{P_1}}\right)^{(\eta_1-m)\wedge0+(\eta_2-n)\wedge0}
\nonumber\\
&=\|x-y\|_\frs^{\gamma-l}|y|_{P_1}^{\mu}|y|_{P_0}^{-\gamma}\sum_{m+n\geq\gamma}\left(\frac{|y|_{P_0}}{|y|_{P_1}}\right)^{\eta_1\wedge m+\eta_2\wedge n},\label{multi2}
\end{align}
where we used $\|x-y\|\leq|y|_{P_0}$ to get the third line. As before, we have $\eta_1\wedge m+\eta_2\wedge n\geq\eta\wedge\gamma=\eta$, and recalling that $|y|_{P_i}\sim|x,y|_{P_i}$ we see that this is indeed the bound we need in \eqref{def:spaces}.  The second term on the right-hand side of \eqref{multi1} is bounded by a constant times
\begin{align*}
\sum_{m+n=l}&\|\Gamma_{xy}f_1(y)-f_1(x)\|_m\|\Gamma_{xy}f_2(y)-f_2(x)\|_n
\\
&\leq\sum_{m+n=l}\|x-y\|_\frs^{\gamma_1+\gamma_2-m-n}|x|_{P_1}^{\mu_1+\mu_2-\eta_1-\eta_2}|x|_{P_0}^{\eta_1+\eta_2-\gamma_1-\gamma_2}
\\
&\lesssim\|x-y\|_\frs^{\gamma-l}|x|_{P_1}^{\mu-\eta_1-\eta_2}|x|_{P_0}^{\eta-\eta_1-\eta_2}|x|_{P_0}^{\eta-\gamma_1-\gamma_2}\|x-y\|_s^{\gamma_1+\gamma_2-\gamma}.
\end{align*}
Since $\gamma_1+\gamma_2\geq\gamma$, $\eta_1+\eta_2\geq\eta$, and $\|x-y\|_s\leq|x|_{P_0}\leq|x|_{P_1}$, this gives the required bound. The third term on the right-hand side of \eqref{multi1} is bounded by a constant times
\begin{align}
\sum_{m+n=l}&\|\Gamma_{xy}f_1(y)-f_1(x)\|_m\|f_2(x)\|_n
\nonumber\\
&\lesssim\sum_{m+n=l}\|x-y\|_\frs^{\gamma_1-m}|x|_{P_1}^{\mu_1-\eta_1}|x|_{P_0}^{\eta_1-\gamma_1}|x|_{P_1}^{\mu_2-n}\left(\frac{|x|_{P_0}}{|x|_{P_1}}\right)^{(\eta_2-n)\wedge0}
\nonumber\\
&\leq \|x-y\|^{\gamma-l}\sum_{m+n=l}\|x-y\|_\frs^{\gamma_1+n-\gamma}|x|_{P_1}^{\mu-\eta_1-\eta_2\wedge n}|x|_{P_0}^{\eta_1-\gamma_1+\eta_2\wedge n-n}
\nonumber\\
&\lesssim\|x-y\|_\frs^{\gamma-l}|x|_{P_1}^{\mu-\eta}
\sum_{m+n=l}\|x-y\|_\frs^{\gamma_1+n-\gamma}|x|_{P_1}^{\eta-\eta_1-\eta_2\wedge n}|x|_{P_0}^{\eta_1-\gamma_1+\eta_2\wedge n-n}.\label{multi4}
\end{align}
Inside the sum, the exponent of $\|x-y\|_\frs$ is nonnegative, due to the relation $\gamma\leq\gamma_1+\alpha_2$, while the exponent of $|x|_{P_1}$ is nonpositive, due to $\eta\leq\eta_1+\eta_2\wedge\alpha_2$. Using $\|x-y\|_\frs\leq|x|_{P_0}\leq|x|_{P_1}$ as before, we get the required bound. Finally, the fourth term on the right-hand side of \eqref{multi1} is bounded similarly, reversing the roles played by $f_1$ and $f_2$.

To prove the continuity estimate \eqref{eq:multi cont}, we of course need only consider the first part of the definition of $\vn{\cdot;\cdot}$, the bound on the second  already follows from above by linearity. We then write
\begin{align}
f(x)-g(x)&-\Gamma_{xy}f(y)+\bar\Gamma_{xy}g(y)
\nonumber\\
&=-\Gamma_{xy}f(y)+\Gamma_{xy}g(y)+\Gamma_{xy}f_1(y)\star\Gamma_{xy}f_2(y)-\bar\Gamma_{xy}g_1(y)\star\bar\Gamma_{xy}g(y)
\nonumber\\
&\quad+(f_1(x)-g_1(x)-\Gamma_{xy}f_1(y)+\bar\Gamma_{xy}g_1(y))\star f_2(x)
\nonumber\\
&\quad+\Gamma_{xy}f_1(y)\star(f_2(x)-g_2(x)-\Gamma_{xy}f_2(y)+\bar\Gamma_{xy}g_2(y))
\nonumber\\
&\quad+\bar\Gamma_{xy}(g_1(y)-f_1(y))\star(\bar\Gamma_{xy}g_2(y)-g_2(x))
\nonumber\\
&\quad+(\bar\Gamma_{xy}f_1(y)-\Gamma_{xy}f_1(y))\star(\bar \Gamma_{xy}g_2(y)-g_2(y))
\nonumber\\
&\quad+(g_1(y)-\bar \Gamma_{xy}g_1(y))\star(f_2(x)-g_2(x)).
\nonumber\\
&=:T_0+T_1+T_2+T_3+T_4+T_5
\end{align}
For $T_0$, repeating the argument in \eqref{multi3}, we need to estimate, for $m+n\geq\gamma$, terms of the form
\begin{align*}
\Gamma_{xy}\cQ_mf_1(y)\star\Gamma_{xy}\cQ_nf_2(y)&-\bar\Gamma_{xy}\cQ_mg_1(y)\star\bar\Gamma_{xy}\cQ_ng_2(y)
\\
&=\Gamma_{xy}\cQ_mf_1(y)\star(\Gamma_{xy}(\cQ_nf_2(y)-\cQ_ng_2(y))
\\
&\quad+\Gamma_{xy}\cQ_mf_1(y)\star(\Gamma_{xy}\cQ_ng_2(y)-\bar\Gamma_{xy}\cQ_ng_2(y))
\\
&\quad+\Gamma_{xy}(\cQ_mf_1(y)-\cQ_mg_1(y))\star\bar\Gamma_{xy}\cQ_ng_2(y)
\\
&\quad+(\Gamma_{xy}\cQ_mg_1(y)-\bar\Gamma_{xy}\cQ_mg_1(y))\star\bar\Gamma_{xy}\cQ_ng_2(y).
\end{align*}
Continuing as in \eqref{multi3}, we get
\begin{align*}
\|T_0\|_l\lesssim\sum_{m+n\geq\gamma}\|x-&y\|_\frs^{m+n-l}\Big[\|f_1(y)\|_m\|f_2(y)-g_2(y)\|_n+\|f_1(y)\|_m\|\Gamma-\bar\Gamma\|_{\gamma_1+\gamma_2}\|g_2(y)\|_n
\\
&\quad\quad
+\|f_1(y)-g_1(y)\|_m\|g_2(y)\|_n+\|\Gamma-\bar\Gamma\|_{\gamma_1+\gamma_2}\|g_1(y)\|_m\|g_2(y)\|_n\Big].
\end{align*}
From here we get the desired bound \eqref{eq:multi cont} by repeating the calculation in \eqref{multi2}.

For the further terms, we shall make use of the fact that for any $\bar\gamma,$ $\bar w$, $h\in\cD_P^{\bar \gamma,\bar w}$, and for pairs $(x,y)$ under consideration, $\Gamma_{xy}h(y)$ satisfies analogous bounds to $h(x)$:
\begin{align}
\|\Gamma_{xy}h(y)\|_l&\leq\sum_{m\geq l}\|x-y\|_\frs^{m-l}\|h(y)\|_m
\lesssim\sum_{m\geq l}\|x-y\|_\frs^{m-l}|y|_{P_1}^{\bar\mu-m}
\left(\frac{|y|_{P_0}}{|y|_{P_1}}\right)^{(\bar\eta-m)\wedge0}
\nonumber\\
&\lesssim |x|_{P_1}^{\bar\mu-l}\left(\frac{|x|_{P_0}}{|x|_{P_1}}\right)^{(\bar\eta-l)\wedge0}.\label{eq:multi Gamma}
\end{align}
For $T_1$, we write
\begin{align*}
\|T_1\|_l\lesssim\vn{f_1;g_1}_{\gamma_1,w_1}\sum_{m+n=l}\|x-y\|^{\gamma_1-m}|x|_{P_1}^{\mu_1-\eta_1}|x|_{P_0}^{\eta_1-\gamma_1}|x|_{P_1}^{\mu_2-n}\left(\frac{|x|_{P_0}}{|x|_{P_1}}\right)^{(\eta_2-n)\wedge0},
\end{align*}
and as we recognise the sum from \eqref{multi4}, the required bound follows.

For $T_2$, we use \eqref{eq:multi Gamma} with $h=f_1$, and then proceed just like for $T_1$, with the role of the indices reversed.

To bound $T_3$, we use \eqref{eq:multi Gamma}, this time with $h=g_1-f_1$, to get
\begin{align*}
\|T_3\|_{l}\leq\|f_1-g_1\|_{\gamma_1,w_1}\sum_{m+n=l}|y|_{P_1}^{\mu_1-m}\left(\frac{|y|_{P_0}}{|y|_{P_1}}\right)^{(\eta_1-m)\wedge0}\|x-y\|_\frs^{\gamma_2-n}|x|_{P_1}^{\mu_2-\eta_2}|x|_{P_0}^{\eta_2-\gamma_2},
\end{align*}
and the sum is again of the same form.

The bound for the term $T_5$ goes similarly to $T_3$, with the indices reversed, and so does $T_4$, with the only difference that the prefactor of the sum is $\|\Gamma-\bar\Gamma\|_{\gamma_1+\gamma_2}\vn{f_1}_{\gamma_1,w_1}$. 
\end{proof}

\subsection{Composition with smooth functions}
\begin{lemma}\label{lem:comp}
Let $V$ be a sector of regularity $0$ with $V_0 = \scal{\bone}$ 
that is $\gamma$-regular with respect to the product $\star$ and furthermore $V\star V\subset V$.

Let $f_1,\ldots,f_n\in\cD_P^{\gamma,w}(V)$ with $w=(\eta,\sigma,\mu)$ such that $\eta,\sigma,\mu\geq0.$ Let furthermore $F:\R^n\rightarrow\R$ be a smooth function. Then $\hat F_\gamma(f)$ belongs to $\cD_P^{\gamma,w}(V)$. Furthermore, $\hat F_\gamma:\cD_P^{\gamma,w}\rightarrow\cD_P^{\gamma,w}$ is locally Lipschitz continuous in any of the seminorms $\|\cdot\|_{\gamma,w;\frK}$ and $\vn{\cdot}_{\gamma,w;\frK}$.
\end{lemma}

\begin{remark}\label{remark:comp boxnorm}
If two modelled distributions $f$, $\bar f$ are such that $f-\bar f\in\cD_{P,\{1\}}^{\gamma,w}$, then clearly $\hat F_\gamma(f)-\hat F_\gamma(\bar f)$ also has $0$ limit at $P_1\setminus P_0$. In this case the analogous Lipschitz bound for $\hat F$ in the seminorms $\bn{\cdot}_{\gamma,w;\frK}$ also holds.
\end{remark}

\begin{remark}
One can use the same construction as in \cite[Prop.~3.11]{HP15} to obtain local Lipschitz continuity when 
comparing two modelled distributions modelled on two different models.
\end{remark}

\begin{proof}
We only give a sketch of the proof, as the majority of the argument is exactly the same as that of the proof of Theorem~4.16 and Proposition~6.12 in \cite{H0}. We prove the main estimates which are somewhat different due to the additional weights and refer the reader to \cite{H0} to confirm that these indeed imply the theorem. 

As usual, we consider the situation $2\|x-y\|_\frs\leq|x,y|_{P_0}\leq|x,y|_{P_1}$. We denote $L=\lfloor \gamma/\zeta\rfloor$, where $\zeta$ is either the lowest nonzero homogeneity such that $V_\zeta\neq\{0\}$, or if that index is larger than $\gamma$, then we set $\zeta=\gamma$. The essential quantities to bound are
\begin{align*}
R_1&:=\sum_{l:\sum l_i\geq\gamma}\Gamma_{xy}\cQ_{l_1}\tilde{f}(y)\star\cdots\star\Gamma_{xy}\cQ_{l_n}\tilde{f}(y),
\\
R_f&:=\Gamma_{yx}f(x)-f(y),
\\
R_2&:=\sum_{|k|\leq L}(\Gamma_{yx}\tilde{f}(x))^{\star k}-(\Gamma_{yx}\tilde{f}(x)+R_f)^{\star k},
\\
R_3&:=\sum_{|k|\leq L}|\bar f(x)-\bar f(y)|^{\gamma/\zeta-|k|}(\tilde{f}(y)-(\bar f(y)-\bar f(x))\bone)^{\star k},
\end{align*}
each of which has to be estimated in the following way, for all $\beta<\gamma$:
\begin{equation}\label{eq:comp bound}
\|R_i\|_\beta\lesssim\|x-y\|_\frs^{\gamma-\beta}|x,y|_{P_1}^{\mu-\eta}|x,y|_{P_0}^{\eta-\gamma}.
\end{equation}
Note that there is a slight abuse of notation here in that $R_f$ is vector-valued. By \eqref{eq:comp bound} we then understand that such an estimate holds for each coordinate, and this convention is applied in the other analogous situations below whenever vector-valued functions are considered.

We further invoke two elementary inequalities from the proof of \cite[~Prop 6.12]{H0}: for $\eta\geq0$, $n\in\mathbb{N}$, $l_1,\ldots,l_n\in\mathbb{N}$, we have
\begin{equation}\label{eq:comp1}
\sum_{i=1}^n (\eta-l_i)\wedge0\geq \Big(\eta-\sum_{i=1}^n l_i\Big)\wedge0,
\end{equation}
and for any multiindex $k$ with $|k|\leq L,$ integer $0\leq m\leq |k|$, real numbers $0<\zeta\leq\gamma$, $0\leq \beta,\eta\leq\gamma$, and integers $l_1,\ldots,l_m$ satisfying $\sum l_i=\beta$ and $l_i\geq\zeta$, it holds
\begin{align}
N+M:=\Big[&(|k|\zeta-\gamma-|k|\eta+(\gamma\eta/\zeta))\wedge0\Big]
\nonumber\\
&+\Big[\beta-\zeta m+(|k|-m)((\eta-\zeta)\wedge0)+\sum_{i=1}^m(\eta-l_i)\wedge0\Big]\geq\eta-\gamma.\label{eq:comp2}
\end{align}

The term $R_1$ looks very similar to what we encountered in \eqref{multi3}, and indeed by the same argument we can write
\begin{align*}
\|R_1\|_\beta&\lesssim\sum_{\sum l_i\geq\gamma}\|x-y\|_\frs^{\sum l_i-\beta}\prod_i\|\tilde{f}(y)\|_{l_i}
\\
&\lesssim\|x-y\|_\frs^{\gamma-\beta}\sum_{\sum l_i\geq\gamma}\|x-y\|_\frs^{\sum l_i-\gamma}\prod_i|y|_{P_1}^{\mu-l_i}\left(\frac{|y|_{P_0}}{|y|_{P_1}}\right)^{(\eta-l_i)\wedge0}
\\
&\lesssim\|x-y\|_\frs^{\gamma-\beta}\sum_{\sum l_i\geq\gamma}|y|_{P_0}^{-\gamma}|y|_{P_1}^{n\mu}\left(\frac{|y|_{P_0}}{|y|_{P_1}}\right)^{\sum(\eta-l_i)\wedge0+\sum l_i}.
\end{align*}
By \eqref{eq:comp1}, the exponent of the fraction above is bounded from below by $\eta\wedge\sum l_i=\eta$, and since $n\mu\geq\mu$ due to $\mu$ being nonnegative, this yields the required bound.

The bound for $R_f$ follows from the definition. For $R_2$, notice that
\begin{align*}
\|\Gamma_{yx}\tilde{f}(x)\|_l&\lesssim\sum_{l'\geq l}\|x-y\|_\frs^{l'-l}\|\tilde{f}(x)\|_{l'}
\\&
\lesssim\sum_{l'\geq l}\|x-y\|_\frs^{l'-l}|x|_{P_1}^{\mu-l'}\left(\frac{|y|_{P_0}}{|y|_{P_1}}\right)^{(\eta-l')\wedge0}
\lesssim\|x-y\|_\frs^{-l}.
\end{align*}
Therefore, for any nonzero multiindex $m$ and any multiindex $m'$,
\begin{align*}
\|R_f^{\star m}\star(\Gamma_{yx}\tilde{f}(x))^{\star m'}\|_\beta&\lesssim\sum_{\substack{l_1+\ldots +l_m \\ +l'_1+\ldots+l'_{m'}=\beta}}\prod_{i=1}^{|m|}\|x-y\|_\frs^{\gamma-l_i}|x|_{P_1}^{\mu-\eta}|x|_{P_0}^{\eta-\gamma}\prod_{i'=1}^{|m'|}\|x-y\|^{-l'_{i'}}
\\
&\lesssim\|x-y\|_\frs^{\gamma-\beta}|x|_{P_1}^{\mu-\eta}|x|_{P_0}^{\eta-\gamma}\left(\|x-y\|_\frs^{\gamma}|x|_{P_1}^{\mu-\eta}|x|_{P_0}^{\eta-\gamma}\right)^{|m|-1},
\end{align*}
and since the quantity in the parentheses is of order one due to $\gamma,\eta,\mu\geq0$ and $\|x-y\|_\frs\leq|x,y|_{P_0}\leq|x,y|_{P_1}$, the bound \eqref{eq:comp bound} for $R_2$ follows.

For $R_3$, fix $k$ and first write
\begin{align}
|\bar f(x)-\bar f(y)|&\leq\|\Gamma_{xy}\tilde{f}(y)\|_0+\|f(x)-\Gamma_{xy}f(y)\|_0
\nonumber\\
&\lesssim\sum_{\zeta\leq l\leq\gamma}\|x-y\|_\frs^l|x|_{P_1}^{\mu-l}\left(\frac{|x|_{P_0}}{|x|_{P_1}}\right)^{(\eta-l)\wedge0}\label{eq:semmmi}
,\end{align}
where $l$ runs over indices in $A\cup\{\gamma\}$ in the specified range. If the exponent of $\|x-y\|_\frs$ were $l-\zeta$ instead of $l$, we would be in the exact same situation as in \eqref{eq:multi Gamma}. Taking this extra $\|x-y\|_\frs^\zeta$ out of the sum, we therefore get the bound
\begin{equation}\label{eq:comp4}
|\bar f(x)-\bar f(y)|\lesssim\|x-y\|^{\zeta}|x|_{P_1}^{\mu-\zeta}\left(\frac{|x|_{P_0}}{|x|_{P_1}}\right)^{(\eta-\zeta)\wedge0},
\end{equation}
and, recalling the notation $N$ from \eqref{eq:comp2},
\begin{equation}\label{eq:comp3}
|\bar f(x)-\bar f(y)|^{\gamma/\zeta-|k|}\lesssim\|x-y\|_\frs^{\gamma-|k|\zeta}|x|_{P_1}^{(\gamma/\zeta-|k|)(\mu-\zeta)}\left(\frac{|x|_{P_0}}{|x|_{P_1}}\right)^{N}.
\end{equation}
Moving to the other constituent of $R_3$, by \eqref{eq:comp4} and the bounds on $\tilde{f}(y)$ from the definition of the spaces $\cD_P^{\gamma,w}$, the we can write
\begin{align*}
\|(\tilde{f}&(y)-(\bar f(y)-\bar f(x))\bone)^{\star k}\|_\beta
\\
&\lesssim\sum_{0\leq m\leq|k|}\sum_{\substack{\sum_{i=1}^m l_i=\beta \\ l_i\geq\zeta}}
\|x-y\|_\frs^{\zeta(|k|-m)}|x|_{P_1}^{(\mu-\zeta)(|k|-m)}\left(\frac{|x|_{P_0}}{|x|_{P_1}}\right)^{((\eta-\zeta)\wedge0)(|k|-m)}
\\
&\quad\quad\times\prod_{i=1}^m|x|_{P_1}^{\mu-l_i}\left(\frac{|x|_{P_0}}{|x|_{P_1}}\right)^{(\eta-l_i)\wedge0}.
\end{align*}
As the sum has finitely many terms, it suffices to treat them separately, and therefore we fix $m$ and $l_i$ as above. Then, since $\beta=\sum l_i\geq\sum\zeta=m\zeta$, we can get a bound
\begin{align*}
\|x-y\|_\frs^{\zeta|k|-\beta}|x|_{P_0}^{\beta-m\zeta}|x|_{P_1}^{|k|\mu-\zeta|k|}|x|_{P_1}^{m\zeta-\beta}\left(\frac{|x|_{P_0}}{|x|_{P_1}}\right)^{((\eta-\zeta)\wedge0)(|k|-m)+\sum(\eta-l_i)\wedge0}
\end{align*}
Moving the second and fourth factor into the fifth one, we get that the exponent of the fraction above becomes $M$, as defined in \eqref{eq:comp2}. Combining this with \eqref{eq:comp3}, we get
\begin{align*}
\|R_3\|_\beta\lesssim\|x-y\|^{\gamma-\beta}|x|_{P_1}^{(\gamma/\zeta)\mu-\gamma}\left(\frac{|x|_{P_0}}{|x|_{P_1}}\right)^{N+M},
\end{align*}
and by \eqref{eq:comp2} and the fact $(\gamma/\zeta)\mu\geq\mu$, we arrive at \eqref{eq:comp bound} for $R_3$.
\end{proof}

\subsection{Reconstruction}\label{subsec:reco}
Recall that, since reconstruction is a local operation, there exists an element $\tilde{\cR}f$ in the dual of smooth functions supported away from $P$ such that the bound \eqref{eq:standard reco estimate} is satisfied if $\lambda\ll|x|_{P_0}\wedge|x|_{P_1}$. A natural guess for the target space of the extension of the reconstruction operator acting on $\cD^{\gamma,w}_P(V)$ would be $\cC^{\eta\wedge\sigma\wedge\mu\wedge\alpha}$.
While this certainly does hold, we need some finer control over the behaviour at the different boundaries. To this end, we introduce weighted versions of H\"older spaces as follows. 

\begin{definition}\label{def:weightedHolder}
Let $a = (a_0,a_1,a_\cap) \in \R_-^3$, write $a_\wedge = a_0\wedge a_1\wedge a_{\cap}$, and let $P = (P_0,P_1)$ as above. Then, we define $\cC^{a}_P$ as the set of distributions $u\in\cC^{a_\wedge}$ that furthermore satisfy the following two properties.
\begin{enumerate}[(a)]
\item For any $x\in\{|x|_{P_0}\leq2|x|_{P_1}\}$, $\lambda\in(0,1]$ satisfying $2\lambda\leq |x|_{P_1}$, and every $\psi\in\cB^r$, where $r=\lceil-a_0+1\rceil$,
\begin{equation}\label{eq:defHolder1}
|u(\psi_x^{\lambda})|\lesssim |x|_{P_1}^{a_{\cap}-a_0}\lambda^{a_0}.
\end{equation}
\item For any $x\in\{|x|_{P_1}\leq2|x|_{P_0}\}$, $\lambda\in(0,1]$ satisfying $2\lambda\leq |x|_{P_0}$, and every $\psi\in\cB^r$, where $r=\lceil-a_1+1\rceil$,
\begin{equation}\label{eq:defHolder2}
|u(\psi_x^{\lambda})|\lesssim |x|_{P_0}^{a_{\cap}-a_1}\lambda^{a_1}.
\end{equation}
\end{enumerate}
For a compact $\frK$, the maximum of the best proportionality constants in \eqref{eq:defHolder1} and \eqref{eq:defHolder2} over $x\in\frK$ is denoted by $\|u\|_{a;\frK}$.
\end{definition}

\begin{proposition}\label{prop:extension3}
Let $u \in \CD'(\R^d\setminus (P_0\cap P_1))$ be such that the bounds \eqref{eq:defHolder1}-\eqref{eq:defHolder2} are satisfied. 
Then, provided $a_\wedge>-\frm$, there exists a unique distribution 
$u'\in\cC^{a}_P$ that agrees with $u$ on test functions supported away from $P_0\cap P_1$. 
\end{proposition}
\begin{proof}
Such a $u'$ clearly satisfies (a)-(b) of Definition~\ref{def:weightedHolder}, so it only needs to be shown that there exists a unique extension of $u$ in $\cC^{a_\wedge}$. 
By Proposition~\ref{prop:extension}, it suffices to obtain the bound
\begin{equation}\label{eq:reco2}
|u(\psi_x^{\lambda})|\lesssim \lambda^{a_\wedge}\;,
\end{equation}
uniformly over $\psi \in \cB^r$ (for some fixed large enough $r$)
and $\lambda \in (0,1]$, for $c\lambda\leq d_\frs(x,P_0\cap P_1)$ with some fixed $c>1$. For sufficiently large $c$ (depending only on the dimension), one can find smooth functions $\phi_i^{(\lambda)}$ with $i = 0,1$ with the following properties: 
\begin{enumerate}[(i)]
\item The $\phi_i^{(\lambda)}$ are supported on $\{x: |x|_{P_i}\geq4\lambda,2 |x|_{P_i}\geq |x|_{P_{1-i}}\}$.
\item If $x\in\R^d$ is such that $d_\frs(x,P_0\cap P_1)\geq (c-1)\lambda,$ then $
\phi_0^{(\lambda)}(x)+\phi_1^{(\lambda)}(x)=1$.
\item For any multiindex $k$, the bound $|D^k\phi_i^{(\lambda)}(x)|\lesssim \lambda^{-|k|_\frs}$ is satisfied for all $x\in\R^d$. 
\end{enumerate}
The functions $\psi_x^\lambda\phi_i^{(\lambda)}$ then satisfy the bounds
$$
\sup_{y\in\R^d}|D^k(\psi_x^\lambda\phi_i^{(\lambda)})(y)|\lesssim \lambda^{-|\frs|-|k|_\frs}
$$
and have support with diameter less than $2\lambda^{|\frs|}$. One can therefore find points $z_i$ with $2|z_i|_{P_i}\geq|z_i|_{P_{1-i}} \vee 8\lambda$, as well as 
functions $\xi^{(i,\lambda)} \in \cB^r$ such that $
\psi_x^\lambda\phi_i^{(\lambda)}=\xi_{z_i}^{(i,\lambda),2\lambda}
$. Applying the estimates \eqref{eq:defHolder1} and \eqref{eq:defHolder2} to $\xi^{(1,\lambda)}$ and $\xi^{(0,\lambda)}$, respectively, we get
$$
|u(\psi_x^\lambda)| \le |u(\xi_{z_0}^{(0,\lambda),2\lambda})|+|u(\xi_{z_1}^{(1,\lambda),2\lambda})|\lesssim\lambda^{(a_\cap-a_1)\wedge0+a_1}+\lambda^{(a_\cap-a_0)\wedge0+a_0},
$$
and since the minimum of the two exponents on the right-hand side is $a_\wedge$, \eqref{eq:reco2} holds indeed.
\end{proof}

\begin{theorem}\label{thm:reco}
Let $f\in\cD_P^{\gamma,w}(V)$, where $V$ is a sector of regularity $\alpha$ and suppose that 
\begin{equation}\label{eq:exponents}
\eta\wedge\alpha>-\frm_0,\quad\sigma\wedge\alpha>-\frm_1,\quad\mu>-\frm\;.
\end{equation}
Then, setting $a=(\eta\wedge\alpha,\sigma\wedge\alpha,\mu)$, there exists a unique distribution 
$$
\cR f\in \cC^{a}_P
$$ 
such that $(\cR f)(\psi)=(\tilde{\cR} f)(\psi)$ for smooth test functions that are compactly supported away from $P$. In particular, $\cR f\in\cC^{a_\wedge}$.

Moreover, if $(\bar\Pi,\bar\Gamma)$ is another model for $\scT$ and $f\in\cD_{P}^{\gamma,w}(V,\Gamma)$, $\bar f\in\cD_{P}^{\gamma,w}(V,\bar\Gamma)$, then one has the bounds, for any $C>0$ and $\frK$ compact
\begin{equation}\label{eq:reco9}
\|\cR f-\cR\bar f\|_{a;\frK}\lesssim \vn{f;\bar f}_{\gamma,w;\bar\frK}+\|\Pi-\bar\Pi\|_{\gamma,\bar\frK}+\|\Gamma-\bar\Gamma\|_{\gamma,\bar\frK},
\end{equation}
uniformly in $f,\bar f$, and the two models being bounded by $C$, where $\bar\frK$ denotes the $1$-fattening of $\frK$.
\end{theorem}
\begin{proof}
By virtue of Proposition~\ref{prop:extension3}, it suffices to extend $\tilde\cR f$ to an element of 
$\CD'(\R^d \setminus (P_0\cap P_1))$ in such a way that \eqref{eq:defHolder1}-\eqref{eq:defHolder2} hold with the 
desired exponents.

By \eqref{eq:standard reco estimate}, it holds, uniformly in $x\in\{|x|_{P_0}\leq2|x|_{P_1}\}$ over compacts, uniformly in $\psi\in\cB^r$, and uniformly in $\lambda\in(0,1]$ such that $4\lambda\leq|x|_{P_0}$, that
\begin{equation}\label{eq:reco6}
|(\tilde{\cR}f-\Pi_xf(x))(\psi_x^\lambda)|\lesssim \lambda^\gamma|x|_{P_0}^{\eta-\gamma}|x|_{P_1}^{\mu-\eta}\lesssim\lambda^{\eta}|x|_{P_1}^{\mu-\eta}.
\end{equation}
Also, in the same situation, we have
\begin{equation}
|(\Pi_xf(x))(\psi_x^\lambda)|\lesssim \sum_l\lambda^l|x|_{P_1}^{\mu-l}\left(\frac{|x|_{P_0}}{|x|_{P_1}}\right)^{(\eta-l)\wedge0}.
\end{equation}
Since $\lambda\lesssim|x|_{P_0}\wedge|x|_{P_1}$, this sum is of the same form that we encountered before, for example in \eqref{eq:semmmi}. By the same argument we get
\begin{equation}\label{eq:reco again-1}
|(\Pi_xf(x))(\psi_x^\lambda)|\lesssim \lambda^\alpha|x|_{P_1}^{\mu-\alpha}\left(\frac{|x|_{P_0}}{|x|_{P_1}}\right)^{(\eta-\alpha)\wedge0}\lesssim \lambda^{\eta\wedge\alpha}|x|_{P_1}^{\mu-(\eta\wedge\alpha)}.
\end{equation}
Combining this with \eqref{eq:reco6}, by Proposition~\ref{prop:extension} we can extend $\tilde\cR f$ to an element $\tilde\cR_0 f\in \CD'(\R^d \setminus P_1)$ such that the bound
\begin{equation}\label{eq:rec again0}
|(\tilde\cR_0f)(\psi_x^\lambda)|\lesssim\lambda^{\eta\wedge\alpha}|x|_{P_1}^{\mu-(\eta\wedge\alpha)}
\end{equation}
holds uniformly in $x\in\{|x|_{P_0}\leq2|x|_{P_1}\}$ over compacts, uniformly in $\psi\in\cB^r$, and uniformly in $\lambda\in(0,1]$ such that $2\lambda\leq|x|_{P_1}$.

One can similarly construct $\tilde{\cR}_1 f \in \CD'(\R^d \setminus P_0)$ such that
$
|(\tilde\cR_1f)(\psi_x^\lambda)|\lesssim\lambda^{\sigma\wedge\alpha}|x|_{P_1}^{\mu-(\sigma\wedge\alpha)}
$
holds in the symmetric situation. Since $\tilde \cR_0 f$ and $\tilde\cR_1 f$ agree on the intersection of their domains, they can be pieced together to get the claimed extension of $\tilde{\cR} f$.
The proof of continuity is again analogous and is omitted here.
\end{proof}

Keeping in mind that our goal will be to apply this calculus for singular SPDEs with  boundary conditions on some domain $D$, $P_1$ will typically stand for $\R\times \partial D$. With a parabolic scaling we have $\frm_1=1$ and so condition \eqref{eq:exponents}, in particular requiring $\sigma\wedge\alpha>-1$ is rather strict and will often be violated. In these situations, a $\cC^{(\eta\wedge\alpha,\sigma\wedge\alpha,\mu)}_P$ extension $\tilde\cR f$ is not unique and hence sometimes it will be more suggestive to write $\hat\cR f$ for particular choices of such extensions. 
On some occasions this choice will be made `by hand', but there is also another generic situation when a canonical choice can be made, as follows.

\begin{theorem}\label{thm:reco hat'}
Let $f\in\cD_{P,\{1\}}^{\gamma,w}$, where $V$ is a sector of regularity $\alpha$ and let $\gamma>0$ and $w$ be such that 
\begin{equs}[eq:exponents']
0>\sigma>-\frm_1\geq\alpha,\quad
\eta\wedge\alpha>-\frm_0,\quad\mu>-\frm\;.
\end{equs}
Then  there exists a unique distribution 
$
\hat \cR f\in\cC^{(\eta\wedge\alpha,\alpha,\mu)}_P
$ 
such that for smooth functions $\psi$ compactly supported away from $P$, $\hat\cR f(\psi)=\tilde \cR f(\psi)$ and that furthermore, 
\begin{equation}\label{eq:reco on P1}
|\hat\cR f(\psi_x^\lambda)|\lesssim \lambda^\sigma|x|_{P_0}^{\mu-\sigma}
\end{equation}
holds uniformly in $x$ over relatively compact subsets of $P_1\setminus P_0$, in $\psi\in\cB^r$, and in $\lambda\in(0,1]$ such that $2\lambda\leq|x|_{P_0}$.  

 Moreover, if $(\bar\Pi,\bar\Gamma)$ is another model for $\scT$ and $f\in\cD_{P,\{1\}}^{\gamma,w}(V,\Gamma)$, $\bar f\in\cD_{P,\{1\}}^{\gamma,w}(V,\bar\Gamma)$, then one has the bound, for all $C>0$ and compact $\frK$
\begin{equation}\label{eq:reco7'}
\|\hat\cR f-\hat{\bar{\cR}} \bar f\|_{\eta\wedge\alpha,\alpha,\mu;\frK}\lesssim \vn{f;\bar f}_{\gamma,w;\bar\frK}+\|\Pi-\bar\Pi\|_{\gamma,\bar\frK}+\|\Gamma-\bar\Gamma\|_{\gamma,\bar\frK}.
\end{equation}
uniformly in $f,\bar f$, and the two models being bounded by $C$, where $\bar\frK$ denotes the $1$-fattening of $\frK$.

Finally, if for all $a\in V$, $\Pi_x a$ is a continuous function, then 
\begin{equation}\label{eq:reco8'}
\hat\cR f(\psi)=\int_{\R^d\setminus P}(\Pi_xf(x))(x)\,\psi(x)\,dx\;.
\end{equation} 

\end{theorem}
\begin{proof}
First notice that such a $\hat\cR f$ has to be unique: any two extensions of $\tilde{\cR} f$ differ by
a distribution concentrated on $P$, which, due to the conditions on the exponents and the constraint \eqref{eq:reco on P1}, has to vanish.

An extension $\tilde\cR_0 f$ with the `right behaviour' on $\R^d\setminus P_1$ is constructed in the proof of Theorem~\ref{thm:reco}. Concerning the behaviour outside $P_0$ we claim that, with $\hat f = \cQ_\sigma^- f$, it suffices to construct an extension $\hat \cR_1 f\in\cD'(\R^d\setminus P_0)$ of $\tilde{\cR} f$ that satisfies the bound
\begin{equation}\label{eq:reco01}
|(\hat\cR_1 f-\Pi_x\hat f(x))(\psi_x^\lambda)|\lesssim \lambda^\sigma|x|_{P_0}^{\mu-\sigma}
\end{equation}
uniformly in $x\in\{|x|_{P_1}\leq2|x|_{P_0}\}$ over compacts, uniformly in $\psi\in\cB^r$, and uniformly in $\lambda\in(0,1]$ such that $2\lambda\leq|x|_{P_0}$.
Indeed, \eqref{eq:reco on P1} then follows from the fact that $\hat f(x) = 0$ for $x \in P_1\setminus P_0$
by the definition of $\cD_{P,\{1\}}^{\gamma,w}$.
Furthermore, by Propositions~\ref{prop3} and~\ref{prop:1}, we have
$$
|\Pi_x\hat f(x)(\psi_x^\lambda)|\lesssim\sum_{\alpha\leq l<\sigma}|x|_{P_0}^{\mu-\sigma}|x|_{P_1}^{\sigma-l}\lambda^l
\lesssim|x|_{P_0}^{\mu-\alpha}\lambda^\alpha\;,
$$
where the last bound follows from the facts that  $|x|_{P_1} \le |x|_{P_0}$, $\alpha \le l$, and $\lambda \le |x|_{P_0}$. 
Therefore, by \eqref{eq:reco01}, the same bound holds for $\hat \cR_1 f$, and so piecing $\tilde\cR_0 f$ and $\hat \cR_1 f$ together, the resulting element of $\cD'(\R^d\setminus (P_0\cap P_1))$ satisfies the conditions of Proposition~\ref{prop:extension3} with $a_0=\eta\wedge\alpha$, $a_1=\alpha$, and $a_\cap=\mu$. Applying the proposition, we get the claimed $\hat\cR f$. Further notice, that in fact it is enough to show \eqref{eq:reco01} for each $m\in\mathbb{N}$ in the case where $x$ is further restricted to run over $A_m:=\{|x|_{P_0}\in[2^{-m-2},2^{-m}]\}$. Indeed, all functions $\psi_x^\lambda$ that are considered in \eqref{eq:reco01} have support that intersects at most two $A_m$'s, and therefore a straightforward partition of unity argument, like for instance the one in the proof of Proposition~\ref{prop:extension3} completes the proof.

To get $\hat\cR_1f$ on $A_m$, first consider $\cR^m\hat f$ defined as in \eqref{eq:Rm}, 
which is a meaningful expression thanks to Proposition~\ref{prop3}. Furthermore, by
\eqref{eq:reco bound away negative m} and using Proposition~\ref{prop3}, one has the bound
\begin{equation}\label{eq:rec again2'}
|(\cR^m \hat f-\Pi_x\hat f(x))(\psi_x^\lambda)|\lesssim\lambda^{\sigma-\alpha}(\lambda\wedge|x|_{P_0})^{\alpha}|x|_{P_0}^{\mu-\sigma}\lesssim\lambda^\sigma|x|_{P_0}^{\mu-\sigma},
\end{equation}
uniformly in $x\in\{|x|_{P_1}\leq2|x|_{P_0}\}\cap A_m$ over compacts, uniformly over $\psi\in\cB^r$, and uniformly over $\lambda\in(0,1]$ such that $4\lambda\leq|x|_{P_1}$. One also has,  by \eqref{eq:standard reco estimate} and the basic properties of the model,
\begin{align}
|(\tilde{\cR}f-\Pi_x\hat f(x))(\psi_x^\lambda)|&\leq|(\tilde{\cR}f-\Pi_x f(x))(\psi_x^\lambda)|+|(\Pi_xf(x)-\Pi_x\hat f(x))(\psi_x^\lambda)|\nonumber\\
&\lesssim \lambda^\gamma|x|_{P_1}^{\sigma-\gamma}|x|_{P_0}^{\mu-\sigma}+\sum_{l>\sigma}\lambda^l|x|_{P_0}^{\mu-\sigma}|x|_{P_1}^{\sigma-l}
\lesssim \lambda^{\sigma}|x|_{P_0}^{\mu-\sigma}\label{eq:rec again1'}
\end{align}
with the same uniformity. Thus the same bound holds for the difference $\tilde\cR f-\cR^m\hat f$, which therefore, by Proposition~\ref{prop:extension}, has a unique extension $\Delta_m \cR f\in\cD'(\{|x|_{P_1}\leq2|x|_{P_0}\}\cap A_m)$ for which the same bound holds even when $\lambda$ is only restricted by $2\lambda\leq|x|_{P_0}$. Hence $\cR^m f+\Delta_m\cR f$ satisfies the required bound \eqref{eq:reco01} (on $A_m$), and it trivially agrees with $\tilde{\cR} f$ on functions supported away from $P$.

As for the last statement of the theorem, one simply has to check that the right-hand side of \eqref{eq:reco8'} satisfies the claimed properties. It trivially coincides with $\tilde\cR f$ away from $P$, and the bound \eqref{eq:reco on P1} follows from the fact that, thanks to Proposition~\ref{prop:1}
$$
|(\Pi_x f(x))(x)|\lesssim|x|_{P_0}^{\mu-\sigma}|x|_{P_1}^\sigma
$$
if $|x|_{P_1}\leq|x|_{P_0}$, where in this particular case the proportionality constant also depends on the local supremum bounds of the continuous functions $\Pi_x a$. Since this additional dependency doesn't affect the uniqueness part of the statement, the proof is complete.
\end{proof}

\subsection{Differentiation}

\begin{lemma}\label{lem:diff}
Let $\mathscr{D}$ be an abstract gradient and let $f\in\cD_P^{\gamma,w}(V)$, where $\gamma>\frs_i$ and $w=(\eta,\sigma,\mu)\in\R^3$. Then $\mathscr{D}_{i} f\in\cD_P^{\gamma-\frs_i,(\eta-\frs_i,\sigma-\frs_i,\mu-\frs_i)}$.
\end{lemma}

This lemma is a direct consequence of the definition of abstract gradients, and since the proof is a trivial modification of that of \cite[~Prop 5.28]{H0}, it is omitted here.

\subsection{Integration against singular kernels}
\label{sec:kernel}

As seen above, in certain situations the distribution $\cR f$ is not uniquely defined as there might be many distributions $\zeta$ with the appropriate regularity that extend $\tilde\cR f$. For any such $\zeta$, let us denote by $\cN_\gamma^\zeta f$ and $\cK^\zeta_\gamma f$ the modelled distributions defined analogously to $\cN_\gamma f$ and $\cK_\gamma f$, but with $\cR f$ replaced by $\zeta$.

Before stating the result on the integration operator in the weighted spaces, let us recall the following identities from \cite{H0}, which hold for any multiindex $k$,
with the usual convention that empty sums vanish
\begin{equs}[eq:rearrange K 1]
(\Gamma_{xy}\cN_\gamma^{\zeta,(n)}f(y))_k&=\frac{1}{k!}\sum_{|k+l|_\frs<\gamma+\beta}
\frac{(x-y)^l}{l!}(\zeta-\Pi_x f(x))(D_1^{k+l}K_n(y,\cdot)),
\\
(\Gamma_{xy}\cJ^{(n)}(y)f(y))_k&=(\cJ^{(n)}(x)\Gamma_{xy}f(y))_k=\frac{1}{k!}\sum_{\delta>|k|_\frs-\beta}(\Pi_x\cQ_\delta\Gamma_{xy}f(y))(D_1^kK_n(x,\cdot)).\nonumber
\end{equs}
In particular, choosing $x=y$, these identities also cover the formulas for the coefficient of $X^k$ in $\cN_{\gamma}^{\zeta,(n)} f(x)$ and $\cJ^{(n)}(x)f(x)$, respectively.

Another nontrivial rearrangement of terms gives
\begin{align}
k!(\Gamma_{xy}\cN_\gamma^{\zeta,(n)}f(y)&+\Gamma_{xy}\cJ^{(n)}(y)f(y)-\cN_\gamma^{\zeta,(n)}f(x)-\cJ^{(n)}(x)f(x))_k
\nonumber\\
&=(\Pi_y f(y)-\zeta)(K_{n;xy}^{k,\gamma})
\nonumber\\
&\quad-\sum_{\delta\leq|k|_\frs-\beta}(\Pi_x\cQ_\delta(\Gamma_{xy}f(y)-f(x)))(D_1^k K_n(x,\cdot)),
\label{eq:rearrange K 0}
\end{align}
where we define, for $\alpha\in\R$,
$$
K_{n;xy}^{k,\alpha}(z)=D_1^kK_n(y,z)-\sum_{|k+l|_\frs<\alpha+\beta}\frac{(y-x)^l}{l!}D_1^{k+l}K_n(x,z).
$$
We will also make use of the fact that following Taylor remainder formula holds:
\begin{equation}\label{eq:rearrange K Taylor}
K_{n;xy}^{k,\alpha}(z)=\sum_{l\in \partial A_{\alpha}}\int_{\R^d}D_1^{k+l}K_n(\bar y,z)Q^l(x-y,d\bar y),
\end{equation}
where all we need from the yet undefined objects is that $\partial A_{\alpha}$ is a finite set of multiindices $l$ which all satisfy $|l|_\frs\geq\alpha+\beta-|k|_\frs$ and that $Q^l(x-y,\cdot)$ is a measure supported on the set $\{\bar y:\|x-\bar y\|_\frs\leq\|x-y\|_\frs\}$, with total mass bounded by a constant times $\|x-y\|_\frs^{|l|_\frs}$.
For a proof of this, see for example  \cite[Appendix~A]{H0}.

\begin{lemma}\label{lem:int}
Fix $\gamma > 0$, $w = (\eta,\sigma,\mu)$, let $V$ be a sector of regularity $\alpha$, and set $a = (\eta\wedge\alpha,\sigma\wedge\alpha,\mu)$.

(i) Let $f\in\cD_P^{\gamma,w}(V)$ and let $K$ be as in Theorem~\ref{thm:standard int} for some $\beta>0$ and abstract integration map $\cI$. 
Let $\zeta\in\cC^{a}$ such that $\zeta(\psi)=(\tilde \cR f)(\psi)$ for all $\psi\in C_0^\infty(\R^d\setminus P)$
and set
\begin{equation}\label{eq:exponents2}
\bar \gamma = \gamma + \beta,\quad
\bar\eta=(\eta\wedge\alpha)+\beta, \quad
\bar\sigma=(\sigma\wedge\alpha)+\beta,\quad
\bar\mu\leq(a_\wedge+\beta)\wedge0,\quad
\bar\alpha=(\alpha+\beta)\wedge0.
\end{equation}
Suppose furthermore that none of  $\bar \gamma$, $\bar\eta$, $\bar\sigma$, or $\bar\mu$ are integers and that 
these exponents satisfy the condition \eqref{eq:exponents}. Then $\cK^\zeta_\gamma f\in\cD_P^{\bar \gamma,\bar w}$, where $\bar w=(\bar\eta,\bar\sigma,\bar\mu)$.

Furthermore, if $(\bar\Pi,\bar\Gamma)$ is a second model realising $K$ for $\cI$ and $\bar f\in\cD_P^{\gamma,w}(V,\bar\Gamma)$, $\bar \zeta\in\cC^a$ are as above, then for any $C>0$ the bound
$$
\vn{\cK^\zeta_\gamma f;\bar\cK^{\bar \zeta}_\gamma\bar f}_{\bar\gamma,\bar w;\frK}\lesssim\vn{f;\bar f}_{\gamma,w;\bar \frK}+\|\Pi-\bar\Pi\|_{\gamma;\bar\frK}+\|\Gamma-\bar\Gamma\|_{\bar\gamma;\bar\frK}+\|\zeta-\bar\zeta\|_{a,\bar\frK}
$$
holds uniformly in models and modelled distributions both satisfying $\vn{f}_{\gamma,w;\bar \frK}+\|\Pi\|_{\gamma;\bar\frK}+\|\Gamma\|_{\bar\gamma;\bar\frK}+\|\zeta\|_{a,\bar \frK}\leq C$, where $\bar\frK$ denotes the $1$-fattening of $\frK$.

Finally, the identity
\begin{equation}\label{eq:reco identity}
\cR\cK_\gamma^\zeta f=K\ast\zeta
\end{equation}
holds.

(ii) If $f\in\cD_{P,\{1\}}^{\gamma,w}$ and the coordinates of $w$ satisfy \eqref{eq:exponents'}, then choosing $\hat \cR f$ in the above in place of $\zeta$, the same conclusions hold, but with the definition of $\bar \sigma$
 in \eqref{eq:exponents2} replaced by $\bar\sigma=\sigma+\beta$.
\end{lemma}
\begin{proof}
The argument showing that $\cN^\zeta_\gamma f$ (and therefore $\cK^\zeta_\gamma f$) is actually well-defined 
is exactly the same as in \cite{H0}. Also, the fact that the required bounds trivially hold for components of $(\cK^\zeta_\gamma f)(x)$ and $(\cK^\zeta_\gamma f)(x)-\Gamma_{yx}(\cK^\zeta_{\gamma}f)(y)$,  whose homogeneity is non-integer, does not change in our setting.

For integers homogeneities, we shall make use of the decomposition of $K$ and use different arguments on different scales. We start by bounding the second term in \eqref{def:spaces}. First consider the case $2^{-n+2}\leq|x|_{P_0}\leq|x|_{P_1}$. We then have, for any multiindex $l$, due to \eqref{eq:standard reco estimate}
\begin{equation}\label{eq:int1}
|(\tilde \cR f-\Pi_x f(x))(D_1^lK_n(x,\cdot))|\lesssim 2^{n(|l|_\frs-\beta-\gamma)}|x|_{P_0}^{\eta-\gamma}|x|_{P_1}^{\mu-\eta}.
\end{equation}
After summation over the relevant values of $n$, we get a bound of order
$$
|x|_{P_0}^{\eta+\beta-|l|_\frs}|x|_{P_1}^{\mu-\eta}\leq|x|_{P_1}^{\mu+\beta-|l|_\frs}\left(\frac{|x|_{P_0}}{|x|_{P_1}}\right)^{\eta+\beta-|l|_\frs},
$$
as required, since $\bar\mu\leq\mu+\beta$.
As for $\cJ^{(n)}(x)f(x)$, for any integer $l$ we have
$$
\|\cJ^{(n)}(x)f(x)\|_l\lesssim\sum_{\delta>l-\beta}2^{n(l-\beta-\delta)}|x|_{P_1}^{\mu-\delta}\left(\frac{|x|_{P_0}}{|x|_{P_1}}\right)^{(\eta-\delta)\wedge0}.
$$
Summing over $n$, we get
\begin{align*}
\sum_{2^{-n+2}\leq|x|_{P_0}} &\|\cJ^{(n)}(x)f(x)\|_l\lesssim\sum_{\delta>l-\beta}|x|_{P_0}^{\delta+\beta-l}|x|_{P_1}^{\mu-\delta}\left(\frac{|x|_{P_0}}{|x|_{P_1}}\right)^{(\eta-\delta)\wedge0}
\\
&=\sum_{\delta>l-\beta}|x|_{P_0}^{\beta-l}|x|_{P_1}^{\mu}\left(\frac{|x|_{P_0}}{|x|_{P_1}}\right)^{\eta\wedge\delta}
\lesssim|x|_{P_1}^{\mu+\beta-l}\left(\frac{|x|_{P_0}}{|x|_{P_1}}\right)^{\eta\wedge\alpha+\beta-l}
,\end{align*}
where we made use of $\delta\geq\alpha$ in the last step.

Next, consider the case $|x|_{P_0}\leq 2^{-n+2}\leq|x|_{P_1}$. Since then $d_{\frs}(\supp D^l_1K_n(x,\cdot),P_1)\sim|x|_{P_1}$, we can invoke part (a) of Definition~\ref{def:weightedHolder}. For any multiindex $l$, we get
\begin{align*}
|(\zeta-\Pi_x f(x))&(D_1^lK_n(x,\cdot))+(\cJ^{(n)}(x)f(x))_l|
\\
&\leq|\zeta(D_1^lK_n(x,\cdot))|+\sum_{\delta\leq|l|_\frs-\beta}|(\Pi_x\cQ_\delta f(x))(D_1^l K_n(x,\cdot)|
\\
&\lesssim 2^{n(|l|_\frs-\beta-\eta\wedge\alpha)}|x|_{P_1}^{\mu-\eta\wedge\alpha}+\sum_{\delta\leq l-\beta}2^{n(|l|_\frs-\beta-\delta)}|x|_{P_1}^{\mu-\delta}\left(\frac{|x|_{P_0}}{|x|_{P_1}}\right)^{(\eta-\delta)\wedge0}.
\end{align*}
Notice that here in fact we only use estimates of $\zeta$ tested against functions centred on the boundary, this observation useful in particular in the proof of part (ii) of the lemma. Let us denote the two terms above by $A_n$ and $B_n$. Summing $A_n$ over the relevant values of $n$, we have two cases, depending on the sign of $|l|_\frs-\beta-(\eta\wedge\alpha)=|l|_\frs-\bar\eta.$ If this exponent is positive, we get, after summation
$$
|x|_{P_0}^{(\eta\wedge\alpha)+\beta-|l|_\frs}|x|_{P_1}^{\mu-(\eta\wedge\alpha)}\leq
|x|_{P_1}^{\mu+\beta-|l|_\frs}\left(\frac{|x|_{P_0}}{|x|_{P_1}}\right)^{\bar\eta-|l|_{\frs}},
$$
which gives the required bound. If, on the other hand, $|l|_\frs-\bar\eta<0$ (equality cannot occur, by assumption), then the sum of the $A_n$'s over the relevant values of $n$ is bounded by a constant times
$$
|x|_{P_1}^{(\eta\wedge\alpha)+\beta-|l|_\frs+\mu-\eta},
$$
which is also of the required order. The treatment of $B_n$ is momentarily postponed.

In the final case, $|x|_{P_0}\leq|x|_{P_1}\leq 2^{-n+2}$. Similarly as above, recalling that $\zeta\in\cC^{a_\wedge}$, we get
\begin{align}
|(\zeta-\Pi_x f(x))&(D_1^lK_n(x,\cdot))+(\cJ^{(n)}(x)f(x))_l|
\nonumber\\
&\leq|\zeta(D_1^lK_n(x,\cdot))|+\sum_{\delta\leq|l|_\frs-\beta}|(\Pi_x\cQ_\delta f(x))(D_1^l K_n(x,\cdot)|
\nonumber\\
&\lesssim 2^{n(|l|_\frs-\beta-a_\wedge)}+\sum_{\delta\leq l-\beta}2^{n(|l|-\beta-\delta)}|x|_{P_1}^{\mu-\delta}\left(\frac{|x|_{P_0}}{|x|_{P_1}}\right)^{(\eta-\delta)\wedge0}.\label{eq:int0}
\end{align}
Recognising the second term as $B_n$, we consider its sum over the values of $n$ in both this and in the second case. Notice that the exponent of $2^n$ is strictly positive: indeed, $\delta+\beta\in\mathbb{N}$ implies $\delta\in\mathbb{N}$, but since $K_n$ and its derivatives annihilate polynomials, such terms have no contribution to the sum. The resulting quantity is bounded by a constant times
\begin{align*}
\sum_{\delta\leq l-\beta}|x|_{P_0}^{\beta+\delta-|l|_\frs}|x|_{P_1}^{\mu-\delta}\left(\frac{|x|_{P_0}}{|x|_{P_1}}\right)^{(\eta-\delta)\wedge0}&\leq 
\sum_{\delta\leq l-\beta}|x|_{P_1}^{\mu+\beta-|l|_\frs}\left(\frac{|x|_{P_0}}{|x|_{P_1}}\right)^{(\eta\wedge\delta)+\beta-|l|_\frs}
\\
&\lesssim |x|_{P_1}^{\mu+\beta-|l|_\frs}\left(\frac{|x|_{P_0}}{|x|_{P_1}}\right)^{(\eta\wedge\alpha)+\beta-|l|_\frs}
\end{align*}
as required. Moving on to the first term on the right-hand side of \eqref{eq:int0},  recall that $\bar\mu\leq a_\wedge+\beta$, and hence
\begin{equation}\label{eq:mu1}
\sum_n2^{n(|l|_\frs-\beta-a_\wedge)}\leq\sum_n2^{n(|l|_\frs-\bar\mu)}\lesssim|x|_{P_1}^{\bar\mu-|l|_\frs},
\end{equation}
where the sum runs over the relevant values of $n$, and we also made use of the fact that $\bar\mu\leq0$ holds, and in fact, by assumption, with strict inequality. This concludes the estimation of the second, and by symmetry, third term in \eqref{def:spaces}.

Turning to bounding $\|\cK^\zeta_\gamma f(x)-\Gamma_{xy}\cK_\gamma^\zeta f(y)\|$, recall that we need only consider pairs $(x,y)$ where $2\|x-y\|_\frs\leq|x,y|_{P_0}\leq|x,y|_{P_1}$. As before, this implies $|x|_{P_i}\sim|y|_{P_i}\sim|x,y|_{P_i}$.

We separate into different scales again, starting by $2^{-n+2}\leq2\|x-y\|_\frs\leq|x,y|_{P_0}\leq|x,y|_{P_1}$. As in \eqref{eq:int1}, we have
$$
|(\cN_{\gamma}^{\zeta,(n)}f(x))_l|\lesssim 2^{n(|l|_\frs-\beta-\gamma)}|x|_{P_0}^{\eta-\gamma}|x|_{P_1}^{\mu-\eta}.
$$
Summing over the relevant values of $n$, we get a bound of order
$$
\|x-y\|_s^{\gamma+\beta-|l|_\frs}|x|_{P_0}^{\eta-\gamma}|x|_{P_1}^{\mu-\eta},
$$
as required. Similarly,
$$
|(\Gamma_{xy}\cN_\gamma^{\zeta,(n)}f(x))_l|\lesssim \sum_{|k+l|_\frs<\gamma+\beta}\|x-y\|_\frs^{|k|_\frs}2^{n(|k+l|_\frs-\beta-\gamma)}|x|_{P_0}^{\eta-\gamma}|x|_{P_1}^{\mu-\eta},
$$
which, after summation, yields an estimate of order
$$
\sum_{|k+l|_\frs<\gamma+\beta}\|x-y\|_\frs^{|k|_\frs}\|x-y\|_\frs^{\gamma+\beta-|k+l|_\frs}|x|_{P_0}^{\eta-\gamma}|x|_{P_1}^{\mu-\eta},
$$
which is again of the required order. Next, using \eqref{eq:rearrange K 1}, we have
\begin{align*}
|(\cJ^{(n)}(x)f(x)-\Gamma_{xy}\cJ^{(n)}(y)f(y))_l|&\leq\sum_{\delta>|l|_\frs-\beta}(\Pi_x\cQ_\delta(f(x)-\Gamma_{xy}f(y))(D_1^{l}K_n(x,\cdot))
\\
&\lesssim\sum_{\delta>|l|_\frs-\beta}\|x-y\|^{\gamma-\delta}_\frs|x|_{P_0}^{\eta-\gamma}|x|_{P_1}^{\mu-\eta}2^{n(|l|_\frs-\beta-\delta)}.
\end{align*}
Summing over the relevant values $n$, we get the bound
$$
\sum_{\delta>|l|_\frs-\beta}\|x-y\|^{\gamma-\delta}_\frs|x|_{P_0}^{\eta-\gamma}|x|_{P_1}^{\mu-\eta}\|x-y\|^{\delta+\beta-|l|_\frs},
$$
as required.

Moving on to larger scales, we will then use the identity \eqref{eq:rearrange K 0}. Starting with the second term,
\begin{align*}
|\sum_{\delta\leq|l|_\frs-\beta}(\Pi_x\cQ_\delta(\Gamma_{xy}f(y)-f(x)))(D_1^l K_n(x,\cdot))|\lesssim\sum_{\delta\leq|l|_\frs-\beta}\|x-y\|^{\gamma-\delta}|x|_{P_0}^{\eta-\gamma}|x|_{P_1}^{\mu-\eta}2^{n(|l|_\frs-\beta-\delta)}.
\end{align*}
This can be treated for all the remaining scales at once: summing over $n$ such that $\|x-y\|_\frs\leq2^{-n+2}$ (the strict positivity of the exponent of $2^n$ can be argued exactly as in the previous similar situation), we get a bound of order
$$
\sum_{\delta\leq|l|_\frs-\beta}\|x-y\|^{\gamma-\delta}|x|_{P_0}^{\eta-\gamma}|x|_{P_1}^{\mu-\eta}\|x-y\|^{\delta+\beta-|l|_\frs},
$$
which is of required order.

We are left to estimate
$$
|(\Pi_y f(y)-\zeta)(K_{n;xy}^{k,\gamma})|.
$$
Rewriting the above quantity as in the formula \eqref{eq:rearrange K Taylor}, and making use of the properties mentioned following it, we have
\begin{align*}
|(\Pi_y f(y)&-\zeta)(K_{n;xy}^{l,\gamma})|
\\
&\leq\sum_{|k|_\frs\geq\gamma+\beta-|l|_\frs}\|x-y\|^{|k|_\frs}_\frs\sup_{\|x-\bar y\|_\frs\leq\|x-y\|_\frs}|(\Pi_y f(y)-\zeta)(D_1^{k+l}K_n(\bar y,\cdot))|
\\
&\leq \|x-y\|^{\gamma+\beta-|l|_\frs}_\frs\sum_{|k|_\frs\geq\gamma+\beta-|l|_\frs}\|x-y\|_{\frs}^{|k+l|_\frs-\gamma-\beta}\sup_{\|x-\bar y\|_\frs\leq\|x-y\|_\frs}|(\Pi_y f(y)-\zeta)(D_1^{k+l}K_n(\bar y,\cdot))|.
\end{align*}
Therefore it remains to show that, for any $k$ multiindex satisfying $|k|_\frs\geq\gamma+\beta-|l|_\frs$ and any $\bar y$ satisfying $\|x-\bar y\|_\frs\leq\|x-y\|_\frs$, the following bound holds.
\begin{equation}\label{eq:int2}
\|x-y\|_{\frs}^{|k+l|_\frs-\gamma-\beta}|(\Pi_y f(y)-\zeta)(D_1^{k+l}K_n(\bar y,\cdot))|\lesssim|x|_{P_0}^{\bar\eta-\bar\gamma}|x|_{P_1}^{\bar\mu-\bar\eta}.
\end{equation}
Notice that in particular, as before, $|x|_{P_i}\sim|\bar y|_{P_i}\sim|x,\bar y|_{P_i}$. To show \eqref{eq:int2} we again treat the remaining different scales separately. First, take $n$ such that $\|x-y\|_\frs\leq2^{-n+2}\leq|x,y|_{P_0}\leq|x,y|_{P_1}$. We write 
\begin{align}
|(\Pi_y f(y)-\zeta)(D_1^{k+l}K_n(\bar y,\cdot))|&\leq|(\Pi_{\bar y}f(\bar y)-\zeta)(D_1^{k+l}K_n(\bar y,\cdot)|\nonumber
\\
&\quad+|(\Pi_{\bar y}(\Gamma_{\bar y y}f(y)-f(\bar y))(D_1^{k+l}K_n(\bar y,\cdot))|.\label{eq:int3}
\end{align}
Summing the first term over the relevant values of $n$, we get a bound of order
$$
\sum_{n}2^{n(|k+l|_\frs-\beta-\gamma)}|x|_{P_0}^{\eta-\gamma}|x|_{P_1}^{\mu-\eta}\lesssim \|x-y\|_{\frs}^{-|k+l|_\frs+\gamma+\beta}|x|_{P_0}^{\eta-\gamma}|x|_{P_1}^{\mu-\eta},
$$
so the prefactor in \eqref{eq:int2} cancels and we get the required bound. Similarly to before, we used that while we only required $|k|_\frs\geq\gamma+\beta-|l|_\frs$, in fact equality can not occur due to the assumptions of the theorem, so the exponent of $2^n$ is strictly positive. The second term in \eqref{eq:int3} is estimated by 
$$
\sum_{\delta\leq\gamma}\|x-y\|_\frs^{\gamma-\delta}|x|_{P_0}^{\eta-\gamma}|x|_{P_1}^{\mu-\eta}2^{n(|k+l|_\frs-\beta-\delta)}.
$$
After summation over $n$, we get the bound
$$
\sum_{\delta\leq\gamma}\|x-y\|_\frs^{\gamma-\delta}|x|_{P_0}^{\eta-\gamma}|x|_{P_1}^{\mu-\eta}\|x-y\|_\frs^{-|k+l|_\frs+\beta+\delta},
$$
which, just as before, is of required order.

Turning to the scale $\|x-y\|_\frs\leq|x,y|_{P_0}\leq2^{-n+2}$, we estimate the the actions of the two distributions acting on the left-hand side of \eqref{eq:int2} separately. First,
$$
|(\Pi_y f(y))(D_1^{k+l}K_n(\bar y,\cdot))|\lesssim\sum_{\alpha\leq\delta\leq\gamma}|x|_{P_1}^{\mu-\delta}\left(\frac{|x|_{P_0}}{|x|_{P_1}}\right)^{(\eta-\delta)\wedge0}2^{n(|k+l|_\frs-\beta-\delta)}.
$$
As before, the exponent of $2^n$ is strictly positive. Therefore
\begin{align*}
\sum_{|x,y|_{P_0}\leq 2^{-n+2}}&\|x-y\|_{\frs}^{|k+l|_\frs-\gamma-\beta}(\Pi_y f(y))(D_1^{k+l}K_n(\bar y,\cdot))
\\
&\lesssim\sum_{\alpha\leq\delta\leq\gamma}|x|_{P_0}^{|k+l|_\frs-\gamma-\beta}|x|_{P_1}^{\mu-\delta}\left(\frac{|x|_{P_0}}{|x|_{P_1}}\right)^{(\eta-\delta)\wedge0}|x|_{P_0}^{-|k+l|_\frs+\beta+\delta}
\\
&\lesssim\sum_{\alpha\leq\delta\leq\gamma}|x|_{P_1}^{\mu-\delta-(\eta-\delta)\wedge0}|x|_{P_0}^{\eta\wedge\delta-\gamma}
\lesssim|x|_{P_1}^{\mu-\eta}|x|_{P_0}^{\eta\wedge\alpha-\gamma}\;,
\end{align*}
as required.

To treat the other distribution in \eqref{eq:int2}, we further divide the scales, and consider first $\|x-y\|_\frs\leq|x,y|_{P_0}\leq2^{-n+2}\leq|x,y|_{P_1}$. In this case the support of $K_n(\bar y,\cdot)$ is separated away from $P_1$, so we have
$$
|\zeta (D_1^{k+l}K_n(\bar y,\cdot))|\lesssim 2^{n(|k+l|_\frs-\beta-(\eta\wedge\alpha))}|x|_{P_1}^{\mu-(\eta\wedge\alpha)}.
$$
After summation on $n$ and multiplying by the prefactor in \eqref{eq:int2}, using $\|x-y\|_\frs\leq|x|_{P_0}$, we obtain a bound of order
$$
|x|_{P_0}^{|k+l|-\gamma-\beta}|x|_{P_0}^{(\eta\wedge\alpha)+\beta-|k+l|_\frs}|x|_{P_1}^{\mu-(\eta\wedge\alpha)},
$$
which is again of required order.

Finally, when $\|x-y\|_\frs\leq|x,y|_{P_0}\leq|x,y|_{P_1}\leq2^{-n+2}$, we can write
$$
|\zeta(D_1^{k+l}K_n(\bar y,\cdot))|\lesssim 2^{n(|k+l|_\frs-\beta-a_\wedge)}\leq 2^{n(|k+l|_\frs-\bar \mu)}.
$$
Summing over $n$ and multiplying by the prefactor in \eqref{eq:int2}, we arrive at the bound
\begin{equation}\label{eq:mu2}
|x|_{P_0}^{|k+l|-\gamma-\beta}|x|_{P_1}^{\bar \mu-|k+l|_\frs}=
|x|_{P_0}^{\bar\eta-\bar\gamma}|x|_{P_0}^{|k+l|_\frs-\bar\eta}|x|_{P_1}^{\bar \mu-|k+l|_\frs},
\end{equation}
and since $|k+l|_\frs-\bar\eta\geq0$, the middle term can be estimated by $|x|_{P_1}^{|k+l|_\frs-\bar\eta}$, and the proof is finished.

The proof of continuity again goes in an analogous way and is omitted here.

As for the identity \eqref{eq:reco identity}, inspecting the proof of \cite[~Thm 5.12]{H0}, one can notice that this boils down to obtaining the estimate
$$
\Big|\sum_{n\geq 0}\int(\Pi_x f(x)-\tilde\cR f)(\cK_{n,yx}^{0,\gamma})\psi_{x}^\lambda(y)\,dy\Big|\lesssim\lambda^{\gamma+\beta}
$$
for $\lambda\ll|x|_{P_0}\wedge|x|_{P_1}$. This however is a local statement and therefore the argument in \cite{H0} carries through for our case virtually unchanged.

(ii) In the $f\in\cD_{P,\{1\}}^{\gamma,w}$ case, when repeating the above arguments, one should only pay attention in order to get the improved exponent $\bar\sigma=\sigma+\beta$ in place of $(\sigma\wedge\alpha)+\beta=\alpha+\beta$. This improvement is the consequence of the improved bound on $\|f(x)\|_l$ near $P_1$, thanks to Proposition~\ref{prop:1}, and of the improved regularity of $\hat \cR f$ when tested against functions centred on $P_1$, thanks to \eqref{eq:reco on P1}.
\end{proof}

\begin{remark}\label{remark:mu2}
The ``slight difficulty'' foreshadowed in Remark~\ref{remark:mu} is the constraint $\bar\mu\leq0$ in the above 
lemma. Indeed, in all three of the concrete examples mentioned in the introduction, it turns out one needs to 
choose $\bar\mu > 0$. Note that the only two places in the proof where the condition $\bar\mu\leq0$ was used 
are \eqref{eq:mu1} with $l=0$ and \eqref{eq:mu2}. In the latter case one, actually only needs $\bar\mu\leq\bar\gamma$, which holds as soon as we choose $\gamma$ sufficiently large so that $\mu\leq\gamma$. Therefore, provided that $\zeta$
is such that the bound
$$
\sum_{2^{-n+2}\geq|x|_{P_1}}|\zeta(D_1^lK_n(x,\cdot))|\lesssim|x|_{P_1}^{\bar \mu-|l|_\frs}\;,
$$
holds for $|x|_{P_0}\leq|x|_{P_1}$, and the corresponding symmetric bound holds for $|x|_{P_1}\leq|x|_{P_0}$, 
for all $|l|_\frs\leq\bar\mu$, and $\bar\mu\leq a_\wedge+\beta$, then the conclusions of Lemma~\ref{lem:int} 
still hold. This appears to be a very strong condition, but in the standard case where $K$ is a non-anticipative
kernel and $\zeta$ is supported on positive times, it is actually quite reasonable, see Proposition~\ref{prop:improved mu} below. 
\end{remark}

\subsection{Integration against smooth remainders with singularities at the boundary}
\label{subsec:int remainder}

From this point on we move to a more concrete setting, and in particular $P_0$ and $P_1$ will play different roles. We shall view $\R^d$ as $\R\times\R^{d-1}$, denoting its points by either $z$ or by $(t,x)$, where $t\in\R$, $x\in \R^{d-1}$. Furthermore we assume that $P_0$ is given by $\{(0,x):x\in\R^{d-1}\}$

\begin{definition}\label{def:Z}
Denote by $\scZ_{\beta,P}$ the set of functions $Z:(\R^d\setminus P)^2 \to \R$ that can be written in the form
$
Z(z,z')=\sum_{n\geq0}Z_n(z,z')
$
where, for each $n$, $Z_n$ satisfies the following
\begin{claim}
\item $Z_n$ is supported on $\{(z,z')=((t,x),(t',x')):\,|z|_{P_1}+|z'|_{P_1}+|t-t'|^{1/\frs_0}\leq C 2^{-n}\}$, 
where $C$ is a fixed constant depending only on the domain $D$.
\item For any ($d$-dimensional) multiindices $k$ and $l$,
$$
|D_1^kD_2^lZ_n(z,z')|\lesssim2^{n(|\frs|+|k+l|_\frs-\beta)},
$$
where the proportionality constant may depend on $k$ and $l$, but not on $n$, $z$, $z'$.
\end{claim}
\end{definition}

The relevance of this definition is illustrated by the following example, which shows that if we consider a
heat kernel on a domain obtained by the reflection principle, then it can always be decomposed into
an element of $\scK_\beta$ and an element of $\scZ_{\beta,P}$.

\begin{example}\label{example}
Our main example will be of the following form. Suppose that $G^0$ is a function on $\R^d\times\R^d\setminus\{(z,z'):z=z'\}$ with the following properties:
\begin{claim}
\item We have a decomposition $G^0=K^0+R^0$, where $K^0\in\scK_\beta$, while $R^0$ is a globally smooth function.
\item For any two multiindices $k$ and $l$ and any number $a$, there exists a constant $C_{k,l,a}$ such that it holds that
$
|D_1^kD_2^lR^0(z,z')|\leq C_{k,l,a}(|x-x'|\vee 1)^a
$.
\end{claim}
As it is shown in \cite{H0}, the heat kernel in any dimension satisfies these conditions with $\beta=2$. Suppose then that we have a discrete group $\cG$ of isometries of $\R^{d-1}$ with a bounded fundamental domain $D$, and with the property that the following implication holds
$$
g\in\cG\setminus\{\id\},\; x,y\in D,\; \|x-g(y)\|_{\frs}\leq 2^{-n}\quad\Rightarrow\quad d_\frs(x,\partial D)\vee d_{\frs}(y,\partial D)\leq 2^{-n}.
$$
Let $a \colon \cG \to \{-1,1\}$ be a group morphism and write
\begin{equation}\label{eq:G}
G((t,x),(s,y))=\sum_{g\in\cG}a_g G^0((t,x),(s,g(y))).
\end{equation}
A concrete example to have in mind is when $D=[-1,1]$ and $\cG$ is generated by the maps $y\mapsto -2-y$ and 
$y\mapsto 2-y$. Then, the trivial morphism $a_g\equiv 1$ yields the Neumann heat kernel on $D$, while the 
morphism with kernel given 
by the orientation-preserving $g$'s yields the Dirichlet heat kernel. Obvious higher dimensional analogues include
the Neumann and Dirichlet heat kernels on $(d-1)$-dimensional cubes.

For functions $f$ and $g$ on $(\R^d)^2$, write $f \sim g$ if $f(z) = g(z)$ for $z \in ([0,1]\times D)^2$.
We claim that, setting $P_1=\R\times\partial D$, there exist $K\in\scK_\beta$, $Z\in\scZ_{\beta,P}$, such that
$G \sim K + Z$.
First, due to the decay properties of $R^0$, the sum $\tilde R = \sum_g a_gR^0((t,x),(s,g(y)))$ converges and defines a globally smooth function which we can truncate in a smooth way outside of $([0,1]\times D)^2$, so that it belongs to 
$\scZ_{\beta,P}$. For $K^0$, we divide the sum
$$\sum_{g\in\cG}a_g K^0((t,x),(s,g(y)))$$
into three parts. For $g=\id$, we simply set $K=K^0$ which belongs to $\scK_\beta$ by assumption. 
The terms with $g$ such that $y\in D$ 
implies $d_\frs(g(y),D)>1$ may safely be discarded since they are supported outside of $([0,1]\times D)^2$. 
For the remaining finitely many terms, say $g_1,g_2,\ldots,g_m$, we use our assumption on $\cG$, by which we can write
\begin{equ}
K^0_n((t,x),(s,g_m(y))) \sim \varphi_n(x,y) K^0_n((t,x),(s,g_m(y)))\;.
\end{equ}
where $\varphi_n$ is $1$ on $\{(x,y):d_\frs(x,\partial D)\vee d_{\frs}(y,\partial D)\leq 2^{-n}\}$, is supported on $\{(x,y):d_\frs(x,\partial D)\vee d_{\frs}(y,\partial D)\leq 2^{-n+1}\}$, and for all multiindices $k$ and $l$, $D_1^kD_2^l\varphi$ is bounded by $2^{n(|k+l|_\frs)}$, up to a universal constant. Let furthermore $\varphi$ be a smooth compactly supported function that equals $1$ on $D\times D$. We can then set
$$
Z_0((t,x),(s,y))=\sum_{i=1}^m\varphi(x,y)K_0^0((t,x),(s,g_i(y)))+\varphi(x,y)\tilde R,
$$
and for $n>0$
$$
Z_n((t,x),(s,y))=\sum_{i=1}^m\varphi_n(x,y)K_n^0((t,x),(s,g_i(y))),
$$
which does indeed yield an element of $\scZ_{\beta,P}$.
\end{example}

\begin{lemma}\label{lem:Z}
Let $a\in\R_-^3$ and $a_\wedge$ be as in Definition~\ref{def:weightedHolder}, $u\in\cC_P^{a}$ and $Z\in\scZ_{\beta,P}$. Then the function
\begin{equation}\label{eq:Z}
v \colon z\mapsto\sum_{n\geq 0}\scal{u ,Z_n(z,\cdot)}
\end{equation}
is a smooth function on $\R^d\setminus P$, and its lift to $\bar T$ via its Taylor expansion, 
which we also denote by $v$, belongs to $\cD_P^{\gamma,w}(\bar T)$, where $\sigma=a_1+\beta$, 
$\gamma\geq\sigma\vee 0$, and $\eta$ and $\mu$ satisfy
\begin{equation}\label{eq:mu3}
\eta\leq\gamma, \quad\mu\leq(a_\wedge+\beta)\wedge0,
\end{equation}
provided neither of $\sigma$ nor $\mu$ are integers. 

If $u$ furthermore satisfies $\scal{u,\psi_z^\lambda}\lesssim\lambda^{\bar a_1}|z|_{P_0}^{a_\cap-\bar a_1}$ for $z\in P_1\setminus P_0$ and $2\lambda\leq|z|_{P_0}$ with some $\bar a_1\geq a_1$, then the conclusions hold with the definition of $\sigma$ replaced by $\sigma=\bar a_1+\beta$.
\end{lemma}
\begin{proof}
Notice that in \eqref{eq:Z} only the terms where $2^{-n}\geq|z|_{P_1}$ give nonzero contributions. In particular, since the sum is finite, any differentiation on $v$ can be carried inside. If $|z|_{P_0}\leq 2|z|_{P_1}$, then we simply use the fact that $u\in\cC^{a_\wedge}$, to get, for any multiindex $l$
\begin{equation}\label{eq:badmuagain1}
|D^l v(z)|\lesssim\sum_{2^{-n}\geq|z|_{P_1}}2^{n(|l|_\frs-\beta-a_\wedge)}
\leq\sum_{2^{-n}\geq|z|_{P_1}}2^{n(|l|_\frs-\mu)}\leq|z|_{P_1}^{\mu-|l|_\frs},
\end{equation}
where we used $\mu\leq a_\wedge+\beta$ as well as $\mu<0$. If $2|z|_{P_1}\leq|z|_{P_0}$, then we distinguish two cases. First, if $2|z|_{P_1}\leq2^{-n}\leq|z|_{P_0}$, then the support of $Z_n(z,\cdot)$ is away from $P_0$, and so we make use of part (b) of the definition of $\cC_P^{a}$: 
\begin{equ}\label{eq:Zproof}
|\scal{u,D_1^lZ_n(z,\cdot)}|\leq 2^{n(|l|_\frs-\beta-a_1)}|z|_{P_0}^{a_\cap-a_1}.
\end{equ}
If $\sigma=a_1+\beta<|l|_\frs$, then the summing up yields
$$
\sum_{2|z|_{P_1}\leq2^{-n}\leq|z|_{P_0}}|\scal{u,D_1^lZ_n(z,\cdot)}|\lesssim |z|_{P_1}^{\sigma-|l|_\frs}|z|_{P_0}^{a_\cap-a_1}=|z|_{P_0}^{a_\cap-a_1+\sigma
-|l|_\frs}\left(\frac{|z|_{P_1}}{|z|_{P_0}}\right)^{\sigma-|l|_\frs},
$$
which is as required, since $-a_1+\sigma=\beta$. If, on the other hand, $\sigma>|l|_\frs,$ then
$$
\sum_{2|z|_{P_1}\leq2^{-n}\leq|z|_{P_0}}|\scal{u,D_1^lZ_n(z,\cdot)}|\lesssim|z|_{P_0}^{a_\cap+\beta-|l|_\frs}.
$$
On the scale, $2|z|_{P_1}\leq|z|_{P_0}\leq 2^{-n}$, when we simply use the fact $u\in\cC^{a_\wedge}$ again in the same way as before, to get
\begin{equation}\label{eq:badmuagain2}
\sum_{|z|_{P_0}\leq 2^{-n}}|\scal{u,D_1^lZ_n(z,\cdot)}|\lesssim \sum_{|z|_{P_0}\leq 2^{-n}}2^{n(|l|_\frs-\beta-a_\wedge)}\leq|z|_{P_0}^{\mu-|l|_\frs}.
\end{equation}
Putting the above estimates together, we conclude that
\begin{equation}\label{eq:ZZ}
\|v(z)\|_{|l|_\frs} = {1\over k!} |D^l v(z)| \lesssim |z|_{P_1}^{\mu-|l|_\frs}\left(\frac{|z|_{P_0}}{|z|_{P_1}}\right)^{(\eta-|l|_\frs)\wedge0}
\end{equation}
if $|z|_{P_0}\leq|z|_{P_1}$, and the corresponding symmetric estimate holds when $|z|_{P_1}\leq|z|_{P_0}$. In particular, the second and third terms in \eqref{def:spaces} are finite for any finite $\gamma$. To bound the first term, it remains to recall that since $v$ is the lift of a smooth function, for any positive integer $\gamma$ and $(z,z')\in\frK_P$
$$
\|v(z)-\Gamma_{zz'}v(z')\|_l\leq \|z-z'\|_\frs^{\gamma-l}\sup_{\bar z\in\frK:|z|_{P_i}\sim|\bar z|_{P_i}\sim|z,z'|_{P_i}}|D^\gamma v(\bar z)|.
$$
Applying \eqref{eq:ZZ} (and its symmetric counterpart) with $l=\gamma,$ we get
$$
|D^\gamma v(\bar z)|\lesssim|z|_{P_0}^{\eta-\gamma}|z|_{P_1}^{\sigma-\gamma}(|z|_{P_0}\vee|z|_{P_1})^{\mu-\eta-\sigma+\gamma},
$$
as required. For $\gamma$ non-integer, it suffices to apply the above with $\gamma$ 
replaced by $\bar \gamma = \lceil \gamma \rceil$ and to note that, for every $\gamma \in (\bar\gamma-1,\bar \gamma)$,
one has $\cD_P^{\bar \gamma,w} \subset \cD_P^{\gamma,w}$.
(To see this, write $f = f \star_{\gamma} \bone$ and
apply Lemma~\ref{lem:mult}, noting that $\bone \in \cD_P^{\gamma,\bar w}$ with
$\bar \eta = \eta \vee 0$, $\bar \sigma = \sigma \vee 0$ and $\bar \mu = 0$.)
For the last statement of the lemma, one can simply notice that in \eqref{eq:Zproof} $u$ is tested against functions centred on $P_1\setminus P_0$, and use the additional assumption on $u$.
\end{proof}

\begin{remark}
The mapping $u\rightarrow\cQ_{\gamma+\beta}^-v$, where $v$ is as in \eqref{eq:Z}, will also be denoted by $Z_\gamma$. As all models that we consider act the same on polynomials, the usual continuity estimates are in this case direct consequences of the above result.
\end{remark}
\begin{remark}\label{remark:mu3}
It is again worth pointing out that the $\mu<0$ condition, used in \eqref{eq:badmuagain1} and \eqref{eq:badmuagain2}, can be omitted if one can derive 
$$
\sum_{2^{-n+2}\geq|z|_{P_1}\vee|z|_{P_0}}|\scal{u,D_1^lZ_n(z,\cdot)}|\lesssim(|z|_{P_1}\vee|z|_{P_0})^{ \mu-|l|_\frs}
$$
for $|l|_\frs\leq\mu$ by some other means.
\end{remark}
One can easily verify that the action of $\cK_\gamma$ and $Z_\gamma$ are compatible in the following sense: take $f\in\cD_P^{\gamma,w}$ and an extension $\zeta$ of $\tilde\cR f$ as in Lemma~\ref{lem:int} (i). Then $Z_{\gamma+\beta}\zeta\in \cD_{P}^{\gamma+\beta,\bar w}$, where $\bar w$ is as in Lemma~\ref{lem:int} (i).

\section{Solving the abstract equation}\label{sec:FPP}
In addition to the setting of Section~\ref{subsec:int remainder} we now assume that, for a bounded domain $D\subset\R^{d-1}$ with a Lipschitz boundary $\partial D$ satisfying the cone condition, $P_1$ is given by $P_1=\R\times\partial D$. We shall denote by $\bar D$ the $1$-fattening of the closure of $D$, and we introduce the $T$-valued function
$$
\bR^{D}_+(t,x)=\left\{\begin{array}{lr}
        \bone, & \text{if } t>0,x\in D,\\
        0, & \text{otherwise.} 
        \end{array}\right.
$$
It is straightforward to see that $\bR^{D}_+\in\cD_P^{\infty,(\infty,\infty,0)}$, and in particular that
multiplication by $\bR^{D}_+$ maps any $\cD_P^{\gamma,w}$ space into itself.

\subsection{Non-anticipative kernels}
In a typical situation of an application of the theory to SPDEs, one important property of the kernel $K$ that we have, further to the quite general setting in Definition~\ref{def: K}, is that it is non-anticipative in the sense that
\begin{equation}
t<s\quad\Rightarrow\quad K((t,x),(s,y))=0.
\end{equation}
We shall use the notations $O=[-1,2]\times \bar D$ and $O_\tau=(-\infty,\tau]\times\bar D$ as well as the shorthand $\vn{f}_{\gamma,w;\tau}$ for $\vn{f}_{\gamma,w;O_\tau}$ and similarly for other norms involving dependence on compact sets. 

First of all, this allows us to improve our conditions on $\mu$.
\begin{proposition}\label{prop:improved mu}
\begin{enumerate}[(i)]
\item In the setting of Lemma~\ref{lem:int} (i), suppose that $K$ is non-anticipative, that $f$ is of the form $\bR^{D}_+g$ for some $g\in\cD_P^{\gamma,w}$, and that $\zeta$ annihilates test functions supported on negative times. Let furthermore $\varepsilon>0$ such that $\frm_0-\beta+\varepsilon>0$ and assume $a_\wedge+\frm_0\geq 0$. Then, modifying the condition on $\bar \mu$ from \eqref{eq:exponents2} to
$$
\bar\mu\leq a_\wedge+\beta-\varepsilon,
$$
the conclusions of Lemma~\ref{lem:int} (i) still hold.

\item The analogous statement holds for Lemma~\ref{lem:int} (ii), where the modified condition on $\bar\mu$ reads as
$$
\bar\mu\leq\eta\wedge\mu\wedge\alpha+\beta-\varepsilon.
$$
\item In the setting of Lemma~\ref{lem:Z}, suppose that $Z$ is non-anticipative and that $u$ annihilates test functions supported on negative times and let $\varepsilon>0$ be as above. Then, modifying the condition on $\mu$ from \eqref{eq:mu3} to
$$
\mu\leq a_\wedge +\beta-\varepsilon,
$$
the conclusion of Lemma~\ref{lem:Z} still hold.
\end{enumerate}
\end{proposition}
\begin{proof}
(i) By Remark~\ref{remark:mu2}, we only need to obtain the bound
\begin{equation}\label{eq:ind1}
\sum_{2^{-n+2}\geq|z|_{P_1}\vee|z|_{P_0}}|\zeta(D_1^lK_n(z,\cdot))|\lesssim(|z|_{P_1}\vee|z|_{P_0})^{\bar \mu-|l|_\frs}
\end{equation}
for $|l|_\frs\leq\bar\mu$.
For all $m\in\mathbb{N}$, define the grid 
$$
\Lambda_m=\{(s,y):s=2^{-m\frm_0 },y=\sum_{j=1}^{d-1}2^{-m\frs_j}k_j e_j, k_j\in\mathbb{Z}\},
$$
where $e_j$ is the $j$-th unit vector of $\R^{d-1}$, $j=1,\ldots,d-1$. Let furthermore $\varphi$ be a function that satisfies
$$
\sum_{y\in\Lambda_0}\varphi(t,x-y)=1\quad\forall t\in[-1,2],x\in\R^{d-1},
$$
and define $\varphi_y^{m,\frs}=2^{-m|\frs|}\varphi_y^{2^{-m}}$.

To show \eqref{eq:ind1}, we first write, with setting $2^{-m}\leq|z|_{P_0}\leq2^{-m+1}$,
$$
\zeta(D_1^lK_n(z,\cdot))=\sum_{y\in\Lambda_m}\zeta(\varphi_y^{m,\frs}(\cdot)D_1^lK_n(z,\cdot)).
$$
Indeed, the function $D_1^lK_n(z,\cdot)-\sum_{y\in\Lambda_m}\varphi_y^{m,\frs}D_1^lK_n(z,\cdot)$ is supported on strictly negative times, and therefore vanishes under the action of $\zeta$. Each of the functions $\varphi_y^{m,\frs}D_1^lK_n(z,\cdot)$ has support of size of order $2^{-m|\frs|}$ and its $k$th derivative is bounded by $2^{n(|\frs|+|l|_\frs-\beta)}2^{m|k|_\frs}$. Recalling that $\zeta\in\cC^{a_\wedge}$, this yields
$$
|\zeta(\varphi_y^{m,\frs}(\cdot)D_1^lK_n(z,\cdot))|\lesssim 2^{-m a_\wedge}2^{-m|\frs|}2^{n(|\frs|+|l|_\frs-\beta)}.
$$
Combining this with the fact that the number of points $y\in\Lambda_m$ for which the support of $\varphi_y^{m,\frs}$ actually intersects the support of $D_1^lK_n(z,\cdot)$, is of order $2^{-n(|\frs|-\frm_0)}2^{m(|\frs|-\frm_0)}$, we get
$$
|\zeta(D_1^lK_n(z,\cdot))|\lesssim 2^{-m(a_\wedge+\frm_0)}2^{n(\frm_0+|l|_\frs-\beta)}.
$$
By multiplying with $2^{n\varepsilon}$, we only increase the right-hand side, and by our assumptions this guarantees that the exponent of $2^n$ becomes positive. Therefore, recalling that $2^{-m}\sim|z|_{P_0}$, we obtain
$$
\sum_{2^{-n+2}\geq|z|_{P_1}\vee|z|_{P_0}}|\zeta(D_1^lK_n(z,\cdot))|\lesssim|z|_{P_0}^{a_\wedge+\frm_0}(|z|_{P_1}\vee|z|_{P_0})^{\beta+\varepsilon-\frm_0-|l|_\frs},
$$
which, using $a_\wedge+\frm_0\geq0$, gives the required bound.

The proof of (ii) goes in the same way, and, in light of Remark~\ref{remark:mu3}, so does that of (iii).
\end{proof}

The other important consequence of the non-anticipativity of our kernel is the following short-time control.

\begin{lemma}\label{lem:short time K}
In the setting of Proposition~\ref{prop:improved mu} (i), suppose that $K$ is non-anticipative. Set, for a $\kappa>0,$ $w'=(\eta',\sigma',\mu'):=(\bar\eta-\kappa,\bar\sigma,\bar\mu-\kappa)$. Then it holds, for any $C>0$
\begin{align}
\vn{\cK^\zeta_\gamma\bR^{D}_+g}_{\bar\gamma,w';\tau}&\lesssim \tau^{\kappa/\frs_0}(\vn{g}_{\gamma,w;\tau}+\|\zeta\|_{a;\tau}), 
\nonumber\\
\vn{\cK^\zeta_\gamma\bR^{D}_+g;\bar\cK^{\bar\zeta}_\gamma\bR^{D}_+\bar g}_{\bar\gamma,w';\tau}&\lesssim \tau^{\kappa/\frs_0}(\vn{g;\bar g}_{\gamma,w;\tau}+\|\Pi-\bar\Pi\|_{\gamma,O}+\|\Gamma-\bar\Gamma\|_{\gamma,O}
\nonumber\\
&\quad\quad\quad+\|\zeta-\bar\zeta\|_{a;\tau})\label{eq:short time}
\end{align}
uniformly in $\tau\in(0,1]$ and in models bounded by $C$. For the second bound, $g$ and $\bar g$ are also assumed to be bounded by $C$.

If we are instead in the situation of Proposition~\ref{prop:improved mu} (ii), then the analogous statement holds, with $\zeta$ replaced by $\hat \cR f$, and hence the last term on the right-hand side of \eqref{eq:short time} can be omitted.
\end{lemma}

\begin{proof}
First, by the fact that $K$ is non-anticipative, using \eqref{eq:standard reco estimate} 
we can improve Lemma~\ref{lem:int} to
$$
\vn{\cK^\zeta_\gamma\bR^{D}_+g}_{\bar\gamma,\bar w;\tau}\lesssim \vn{g}_{\gamma,w;\tau}+\|\zeta\|_{a;\tau}.
$$
This already takes care of bounding the first and third term in \eqref{def:spaces}, since, using the shorthand $F=\cK^\zeta_\gamma\bR^{D}_+g$, for $(z,z')\in (O_\tau)_P$
$$
\frac{\|F(z)-\Gamma_{zz'}F(z')\|_l}
{\|z-z'\|_{\frs}^{\bar\gamma-l}|z,z'|_{P_0}^{\eta'-\bar\gamma}|z,z'|_{P_1}^{\bar\sigma-\bar\gamma}(|z,z'|_{P_0}\vee|z,z'|_{P_1})^{\mu'-\eta'-\bar\sigma+\bar\gamma}}
\lesssim|z,z'|_{P_0}^{\bar\eta-\eta'}\vn{F}_{\bar\gamma,\bar w;\tau},
$$
where we used that $\mu'-\eta'=\bar\mu-\bar\eta$. Similarly, for $z\in O_\tau\cap\{|z|_{P_1}\leq|z|_{P_0}\}$, 
$$
\frac{\|F(z)\|_l}{|z|_{P_0}^{\mu'-l}\left(\tfrac{|z|_{P_1}}{|z|_{P_0}}\right)^{(\bar\sigma-l)\wedge0}}\lesssim|z|_{P_0}^{\mu'-\bar\mu}\vn{F}_{\bar\gamma,\bar w;\tau}.
$$
Keeping in mind that $|z|_{P_0}\leq t^{1/\frs_0}$, by the definition of the exponents $w'$, these are indeed the required bounds. Similarly, we have for $z\in O_\tau\cap\{|z|_{P_0}\leq|z|_{P_1}\}$
$$
\frac{\|F(z)\|_l}{|z|_{P_1}^{\mu'-l}\left(\tfrac{|z|_{P_0}}{|z|_{P_1}}\right)^{(\eta'-l)\wedge0}}\leq
\frac{\|F(z)\|_l}{|z|_{P_1}^{\mu'-\eta'}|z|_{P_0}^{\eta'-l}}\lesssim
|z|_{P_0}^{\eta'-\bar\eta}\bn{F}_{\bar\gamma,\bar w,\{0\};\tau},
$$
and hence, by virtue of Proposition~\ref{prop:1}, the proof is complete if we can show that $F=\cK^\zeta_\gamma\bR^{D}_+g\in\cD_{P,\{0\}}^{\bar\gamma,\bar w}$. This, on the other hand, follows from the proof of \cite[~Thm 7.1]{H0}, given that away from $P_1$, $\zeta$ belongs to $\cC^{\eta\wedge\alpha}$, which is exactly the situation considered therein.
The bound on the difference again follows in an analogous way.
\end{proof}
The corresponding results hold for the singular remainder as well.

\begin{lemma}\label{lem:short time Z,R}
Let $Z\in\scZ_{\beta,P}$, $f$, $\zeta$, $\gamma$, $\bar\gamma$, $w$, and $w'$ be as in Lemma~\ref{lem:short time K}.
Then it holds, for any $C>0$
\begin{align*}
\vn{Z_{\gamma}\zeta}_{\bar\gamma,w';\tau}&\lesssim \tau^{\kappa/\frs_0}\|\zeta\|_{a;\tau},
\end{align*}
uniformly in $\tau\in(0,1]$.
\end{lemma}
\begin{proof}
The proof goes precisely as in the previous lemma, with the only difference that we cannot refer to \cite{H0} 
to argue that $F:=Z_{\bar\gamma}\zeta\in\cD_{P,\{0\}}^{\gamma,\bar w}$. We therefore need to show that $(F)_k$ has limit $0$ at points of 
$P_0\setminus P_1$ whenever $|k|_\frs\leq\eta\wedge\alpha+\beta$. This is simply due to the fact that, for such $k$, the function
$$
z\rightarrow\zeta(Z(z,\cdot))
$$
is continuous away from $P_1$, and is $0$ for negative times. 
\end{proof}

\subsection{On initial conditions}

The class of admissible initial conditions depends on the particular choice of the kernel in that in addition to the regularity, some boundary behaviour may be required. In the setting of Example~\ref{example}, which is general enough to cover all of our examples, this can be formalised as follows.

\begin{lemma}\label{lem:initial 2}
Let $\cG$ and $G$ be as in Example~\ref{example} and let $u_0$ be a function on $D$ such that the function $\bar u_0$ defined by
$$
\bar u_0(x)=a_g\bar u_0(g^{-1}x)
$$
for the $g\in\cG$ such that $g^{-1}x\in D$, has a continuous extension that belongs to $\cC^\alpha(\R^{d-1})$. Then the function
$$
v(t,x)=\int_DG((t,x),(0,y)) u_0(y)dy
$$
is smooth on $(0,\infty)\times D$ and extending it by $0$ to $\R^d\setminus(0,\infty)\times D$, for any multiindex $l$, the pointwise lift of its $l$-th derivative via its Taylor expansion belongs to $\cD_P^{\gamma,(\alpha-|l|_\frs,\sigma,(\alpha-|l|_\frs)\wedge0)}$ for any  $0\leq\sigma\leq\gamma$. 
\end{lemma}
\begin{proof}
We can write
$$
v(t,x)=\int_{\R^d}G^0((t,x),(0,y))\bar u_0(y)dy.
$$
By assumption, the conditions of \cite[~Lem 7.5]{H0} are satisfied, and hence $v$ satisfies the bounds
$$
|D^l v(t,x)|\lesssim|z|_{P_0}^{(\alpha-|l|_\frs)\wedge0}.
$$
This already gives the right bounds for $\|D^l v(z)\|_k$, $k=0,1,\ldots$. From this one can deduce the bound for the quantity $\|D^lv(z)-\Gamma_{zz'}D^lv(z')\|_k$ precisely as in the proof of Lemma~\ref{lem:Z}.
\end{proof}

\subsection{The fixed point problem}

At this point everything is in place to solve the abstract equations that will arise as `lifts' of equations 
similar to the ones in Section~\ref{subsec:applications}.
As the notation is already quite involved, we refrain from the full generality concerning the kernel $K+Z$ and the 
scaling $\frs$ and only state the result in a form that is sufficient to treat nonlinear perturbations of 
the stochastic heat equation with some boundary conditions. Our main goal is to formulate a fixed point 
argument that is just general enough to cover the examples mentioned in the introduction, as well as some
related problems.

Our setup will involve families of Banach spaces
depending on some parameter $\tau>0$ (which will represent the time over which we solve our equation).
We will henceforth talk of a ``time-indexed space $\CV$'' for a family $\CV = \{\CV_\tau\}_{\tau > 0}$ of Banach
spaces as well as contractions $\pi_{\tau'\leftarrow \tau}\colon \CV_\tau \to \CV_{\tau'}$ for all $\tau' < \tau$ with the property
that $\pi_{\tau''\leftarrow \tau'} \circ \pi_{\tau'\leftarrow \tau} = \pi_{\tau''\leftarrow \tau}$. We consider $\CV$ itself as a
Fr\'echet space whose elements are collections $\{v_\tau\}_{\tau > 0}$ satisfying the consistency
condition $v_{\tau'} = \pi_{\tau'\leftarrow \tau} v_\tau$ and with the topology given by the collections
of seminorms $\|\cdot\|_\tau$ inherited by the spaces $\CV_\tau$. We will write $\pi_\tau\colon \CV \to \CV_\tau$ for the
natural projection. 

Given a bounded and piecewise $\cC^1$ domain $D \subset \R^{d-1}$,
a typical example of a time-indexed space
is given by the space $\CV = \CD^{\gamma,w}_P$ with $\pi_\tau$ given by the restriction to $[0,\tau] \times D$
and norms $\|\cdot\|_\tau$ given by $\vn{\cdot}_{\gamma,w;D_\tau}$, where $D_\tau = [0,\tau]\times D$.
Similarly, we write again $\cC^w_P$ for the time-indexed space consisting of distributions
on $\R^d$ which vanish outside of $\R_+ \times D$, endowed with the norms 
of Definition~\ref{def:weightedHolder}, but restricted to test functions $\psi$, points $x$
and constants $\lambda$ such that the support of $\psi_x^\lambda$ lies in $(-\infty,\tau]\times \R^{d-1}$.

Given two time-indexed spaces $\CV$ and $\bar \CV$, we call a map
$A \colon \CV \to \bar \CV$ `adapted' if there are
maps $A_\tau \colon \CV_\tau \to \bar \CV_\tau$ such that $\pi_\tau A = A_\tau \pi_\tau$.
If $A$ is linear, we will furthermore assume that the norms of $A_\tau$ are uniformly bounded
over bounded subsets of $\R_+$. Similarly, we call $A$ ``locally Lipschitz'' if each of the $A_\tau$
is locally Lipschitz continuous and, for every $K>0$ and $\tau>0$, the Lipschitz constant of  
$A_{\tau'}$ over the centred ball of radius $K$ in $A_{\tau'}$ is bounded, uniformly over $\tau' \in (0,\tau]$.

With these preliminaries in place, our setup is the following.
\begin{claim}
\item Fix $d\geq 2$, $\beta=2$, the scaling $\frs=(2,1,\ldots,1)$ on $\R^d=\{(t,x):t\in\R,x\in\R^{d-1}\}$, and a regularity structure $\scT$.
\item Let $\gamma$, $\gamma_0$ be two positive numbers satisfying $\gamma<\gamma_0 + 2$ and
let $V$ be a sector of regularity $\alpha \le 0$ and such that $\bar T \subset V$. 
\item Set $P_0=\{(0,x):x\in\R^{d-1}\}$ and $P_1=\{(t,x):t\in\R,x\in\partial D\}$, where $D$ is a domain in $\R^{d-1}$ with a piecewise $\cC^1$ boundary, satisfying the cone condition.
\item We assume that we have an abstract integration map $\cI$ of order $2$ as well as non-anticipative
kernels $K\in\scK_2$ and $Z\in\scZ_{2,P}$. We then construct the operator $Z_\gamma$ and, for every admissible model $(\Pi,\Gamma)$, 
the operator $\cK_\gamma$ as in Sections~\ref{sec:kernel} and~\ref{subsec:int remainder}.
\item We fix a family $((\Pi^\varepsilon,\Gamma^\varepsilon))_{\varepsilon\in(0,1]}$ of admissible models converging to $(\Pi^0,\Gamma^0)$ as $\varepsilon\rightarrow0$.

\item We fix a collection of time-indexed spaces $\cV_\eps$ with $\eps \in [0,1]$ endowed
with adapted linear maps $\hat \CR^\eps \colon \cV_\eps \to \bigoplus_{i=0}^n\cC^{w_i}_P$ 
and $\iota_\eps\colon \cV_\eps \to \bigoplus_{i=0}^n\cD^{\gamma_0,w_i}_P(V_i,\Gamma^\eps)$, 
where $V_i$ are sectors of regularity $\alpha_i$, satisfying $\cI(V_i) \subset V$ and $w_i\in\R^3$.
Finally, we assume that for every $\eps \in [0,1]$ and every  $v \in \cV_\eps$, one has
\begin{equ}\label{eq:hatR0}
\bigl(\tilde \cR\bR^D_+\iota_\eps v\bigr)(\psi) = \bigl(\hat \CR^\eps v\bigr)(\psi)
\end{equ}
for any $\psi \in \cC_0^\infty(\R^d \setminus P)$. 
Denote $\tilde \cC=\bigoplus_{i=0}^n\cC^{w_i}_P$ and $\tilde \cD=\bigoplus_{i=0}^n\cD^{\gamma_0,w_i}_P(V_i,\Gamma^\eps)$, which are themselves time-indexed spaces equipped with the natural norms.
\item We fix a collection of time-indexed spaces $\CW_\eps$ of modelled distributions
such that the linear maps
\begin{equ}
\cP_\gamma^{(\eps)} v = \sum_{i=0}^n \bigl(\cK_\gamma^{(\hat\cR^\varepsilon v)_i} (\bR_+^D\iota_\eps v)_i + Z_\gamma (\hat\cR^\varepsilon v)_i \bigr)\;,
\end{equ}
are bounded from $\cV_\eps$ into $\cW_\eps$ with a bound of order $\tau^\theta$ for some $\theta>0$ for its restriction to time $\tau \in (0,1]$,
uniformly over $\eps \in [0,1]$.
\item For $\eps \in [0,1]$, we fix a collection of adapted locally Lipschitz continuous maps
$F_\eps:\cD^{\gamma,w}_P(V,\Gamma^\eps) \to \cV_\eps$.
\item There are `distances' $\vn{\cdot;\cdot}_{\cW;\tau}$ (possibly also depending on $\eps$) 
defined on $\cW_\eps \times \cW_0$
that are compatible 
with the maps $F_\eps$ and $\cP_\gamma$ in the sense that, for $u\in\cV_\eps$, $v\in\cV_0$, 
and $\tau\in(0,1]$, one has
\begin{equ}
\tau^{-\theta}\vn{\cP^{(\eps)}_\gamma u;\cP^{(0)}_\gamma v}_{\cW;\tau}\lesssim
\vn{\iota_\eps u; \iota_0v}_{\tilde\cD; D_\tau} 
+ \|\hat\CR^\eps u- \hat \CR^0v\|_{\tilde\cC;D_\tau} +o(1)\;,
\end{equ}
as $\eps \to 0$. Similarly, uniformly over modelled distributions $f\in\cW_\eps$, $g\in\cW_0$
bounded by an arbitrary constant $C$ and uniformly over $\tau\in(0,1]$, one has
\begin{equ}[e:conteps]
\vn{\iota_\eps F_\eps(f); \iota_0 F_0(g)}_{\tilde\cD; D_\tau} 
+ \|\hat\CR^\eps F_\eps(f)- \hat \CR^0 F_0(g)\|_{\tilde\cC;D_\tau} \lesssim\vn{f;g}_{\cW;\tau} + o(1)
\;,
\end{equ}
as $\eps \to 0$.
\end{claim}

\begin{remark}
The reader may wonder what the point of this rather complicated setup is. By choosing for
$\CV_\eps$ a direct sum of spaces of the type defined in Section~\ref{sec:def}, it 
allows us to decompose the right hand side
of our equation into a sum of terms with well-controlled behaviour at the boundary.
This gives us the flexibility to exploit different features of each term to control the
corresponding ``reconstruction operator'' $\hat \CR^\eps_i$. 
For example, in the case of 2D gPAM, the term  $\hat f_{ij}(u) \star \cD_i(u)\star \cD_j(u)$ can be 
reconstructed because the corresponding weight exponents are sufficiently large, the term $(\hat g(u) - g(0)\bone) \star \Xi$ 
can be reconstructed because it vanishes on the boundary, and 
the term $g(0)\Xi$ can be reconstructed because it corresponds to (a constant times) white noise, 
multiplied by an indicator function.
\end{remark}

We then have the following result.

\begin{theorem}\label{thm:FPP}
In the above setting, there exists $\tau>0$ such that, for every $\varepsilon\in[0,1]$ and every $v \in \cW_\eps$, the equation
\begin{equation}\label{eq:main eq}
u=\cP^{(\varepsilon)}_{\gamma_0} F_\eps(u)+ v\;,
\end{equation}
admits a unique solution $u^\varepsilon\in \cW_\eps$ on $(0,\tau)$. The solution map $\cS_\tau:(v,\varepsilon)\mapsto u^\varepsilon$ is furthermore jointly continuous at $(v,0)$. 
\end{theorem}

\begin{proof}
By assumption $\cP^{(\eps)}_{\gamma_0}$ is an adapted linear map from
$\CV_\eps$ to $\cW_\eps$ with control on its
norm that is uniform over $\eps \in [0,1]$. It has the additional property that, when restricted to time $\tau$,
its operator norm is bounded by $\CO(\tau^\theta)$ for some exponent $\theta>0$, uniformly in $\eps$.
Combining this with the uniform local Lipschitz continuity of the maps $F_\eps$, it is immediate that,
for every $C> 2\|v\|_{\cW;1}$, there exists $\tau\in (0,1]$ such that
the right hand side of \eqref{eq:main eq} is a contraction and therefore admits a unique fixed point 
in the centred ball of radius $C$ in $\cW_\eps$. 

To show that this is the unique fixed point in all of $\cW_\eps$ is also standard: 
assume by contradiction that there
exists a second fixed point $\bar u$ (which necessarily has norm strictly greater than $C$). Then, for every $\tau' < \tau$,
the restrictions of both $u$ and $\bar u$ are fixed points in $\cW_\eps$. However, since the
norm of $A_\eps$ is bounded by $\CO(\bar \tau^\theta)$, one has uniqueness of the fixed point in a ball of 
radius $\bar C(\tau')$ of $\cW_\eps$ with $\lim_{\tau' \to 0} \bar C(\tau') = \infty$,
so that one reaches a contradiction by choosing $\tau'$ small enough.
The continuity of the solution map at $(v,0)$ then follows immediately from \eqref{e:conteps}.
\end{proof}

\section{Singular SPDEs with boundary conditions}\label{sec:applications}

The next three subsections are devoted to the proofs of Theorems~\ref{thm:PAM},~\ref{thm:KPZ D}, and~\ref{thm:KPZ N}, respectively. We do rely on the results of the corresponding statements without boundary conditions from \cite{H_KPZ,H0}, in particular the specific regularity structures, models, and their convergence do not change in our setting. Therefore we only specify details about these objects to the extent that is sufficient to cover the new aspects of our setting.
\subsection{2D gPAM with Dirichlet boundary condition}

The regularity structure for the equation \eqref{eq:0PAM} is built as in \cite[~Sec 8]{H0},
and the models $(\Pi^\eps,\Gamma^\eps)_{\eps\in[0,1]}$ as in \cite[~Sec 10]{H0},
 and we will
use the notations from there without further ado.
We use the periodic model with sufficiently large period: if the truncated heat kernel $K^0$ is chosen to have support of diameter $1$,
then the periodic model on $[-2,2]^2$ suffices, since convolution with $K^0$ and with its periodic symmetrisation agrees on $[-1,1]^2$.
The homogeneity of the symbol $\Xi$ is denoted by 
$-1-\kappa$, where $\kappa\in(0,(1/3)\wedge\delta)\setminus\mathbb{Q}$, with $\delta$ being the regularity of the initial condition.

Our setup to apply Theorem~\ref{thm:FPP} is the following. The sectors we are working with are
$$
V =\cI(T)+\bar T,\quad V_0 =T_0^+\star\scD(V)\star\scD(V),\quad V_1 =T_0^+\star\Xi
,\quad V_2=\scal{\Xi}
$$
and we set the exponents $\gamma =1+2\kappa$, $\gamma_0=\kappa$,
\begin{equs}
\alpha &=0, &\quad
\eta &=\kappa&\quad \sigma &=1/2+\kappa&\quad \mu &=-\kappa;\\
\alpha_0 &=-2\kappa, &\quad
\eta_0 &=2\kappa-2,&\quad \sigma_0 &=2\kappa-1,&\quad \mu_0 &=2\kappa-2;\\
\alpha_1 &=-1-\kappa, &\quad
\eta_1 &=-1,&\quad
\sigma_1 &=-1/2,&\quad
\mu_1 &=-1-\kappa;\\
\alpha_2 &=-1-\kappa, &\quad
\eta_2 &=-1-\kappa,&\quad
\sigma_2 &=-1-\kappa,&\quad
\mu_2 &=-1-\kappa.
\end{equs}

We then set 
\begin{equ}[e:space]
\cV_\eps=\cD_P^{\gamma_0,w_0}(V_0,\Gamma^\eps)\oplus\cD_{P,\{1\}}^{\gamma_0,w_1}(V_1,\Gamma^\eps)\oplus\cD_P^{\gamma_0,w_2}(V_2,\Gamma^\eps)\;,
\end{equ}
and we let $\iota_\eps$ be the identity.
As for $\hat \cR^\eps$, it is chosen to act coordinate-wise, and in the first two coordinates there is no choice to be made, one simply applies Theorems~\ref{thm:reco}-\ref{thm:reco hat'}. The definition of the action
of $\hat \cR^\eps$ on the third coordinate is momentarily postponed.

We take $G$ to be the Dirichlet heat kernel of the domain $D=(-1,1)^2$ continued to all of $\R^2$ as 
in Example~\ref{example}. We also consider the decomposition $G \sim K + Z$ given there and construct 
$\cK_{\gamma_0}$ and $Z_{\gamma_0}$ accordingly. Furthermore, by Schauder's estimate, 
it follows that, for all $f\in\cC^\alpha$ with $\alpha>-2$, the function
$$
(t,x) \mapsto \int_{[0,t]\times D}G((t,x),(s,y))f(s,y)\,ds\,dy
$$
is continuous and vanishes on $\R_+ \times \d D$.
In particular, for any $v\in\cV_\eps$, the modelled distribution
$$
h=(\cK_{\gamma_0}^{(\eps)} + Z_{\gamma_0}\hat\cR^\eps)v
$$
satisfies $\langle \bone,h(t,x)\rangle=0$ for all $t>0$ and $x \in \d D$. Since the only basis element in $V$ 
with homogeneity lower than $\sigma$ is $\bone$, we conclude that one has $h\in\cD_{P,\{1\}}^{\gamma,w}$. 
We exploit this by setting the time-indexed space $\CW_\eps$ to be
\begin{equ}
\CW_\eps = \bigl\{u \in \cD_{P,\{1\}}^{\gamma,(\eta,\sigma,0)}\,:\, \scD_i u\in\cD_{P}^{\gamma-1,(\eta-1,\sigma-1,\kappa-1)},\,i = 1,2\bigr\}\;.
\end{equ}
The reason for only imposing a slightly weaker condition on $u$ itself (i.e.\ we use $0$ instead of $\kappa$ as the third
singularity index) is to be able to deal with initial conditions. 
Indeed, let $v$ be the lift of the solution of the linear equation
\begin{equs}\label{eq:initial}
\partial_t v=\Delta v,\quad v|_{\partial D}=0,\quad v|_{\{0\}\times D}=u_0.
\end{equs}
Combining our assumption that $u_0 \in \CC^\delta$ with Lemma~\ref{lem:initial 2} and the definition of the various
exponents, we then note that indeed $v \in \CW_\eps$ as required, but this would \textit{not} be the case
had we simply replaced $\CW_\eps$ by $\cD_{P,\{1\}}^{\gamma,(\eta,\sigma,\kappa)}$. Due to the above choice of exponents, the required estimate of 
order $T^\theta$ of the short time norm of $\cP_\gamma^{(\eps)}$ 
from $\cV_\eps$ to $\cW_\eps$ follows from 
Lemmas \ref{lem:short time K}-\ref{lem:short time Z,R}, with the choice 
$$
\vn{f;g}_{\cW;\tau}:=\vn{f;g}_{\gamma,(\eta,\sigma,0);\tau}+\vn{\scD f;\scD g}_{\gamma-1,(\eta-1,\sigma-1,\kappa-1);\tau}.
$$

We now define the functions $F_\eps$. They are given as local operations with formal
expression that do not depend on $\eps$, and we define its three components according to 
the decomposition \eqref{e:space} separately.
We first set
$$
F^{(0)}(u)=\hat f_{ij}(u)\star\scD_i(u)\star\scD_j(u).
$$
Here $\hat f_{ij}$ are the lifts of the functions $f_{ij}$ in \eqref{eq:0PAM}. By Lemmas~\ref{lem:comp},~\ref{lem:mult}, and~\ref{lem:diff}, $F^{(0)}$ is indeed a mapping from $\CW_\eps$ to $\cD_{P}^{\gamma_0,w_0}(V_0)$.
At this stage we note that the fact that the derivatives of 
elements of $\cW^\eps$ have better corner singularity than $\mu-1$ is crucial, since otherwise we would have had to choose
$\mu_0 \le -2$ which would violate the condition $\mu_0 + 2 > (\mu\vee 0)$ appearing in the conditions of
Theorem~\ref{thm:FPP}.

Next, set
$$
F^{(1)}(u)=(\hat g(u)-g(0)\bone)\star\Xi
$$
Again, using Lemmas~\ref{lem:comp} and~\ref{lem:mult}, it is easy to see that $F^{(1)}$ maps from $\cW_\eps$ to $\cD_{P}^{\gamma_1,w_1}(V_1)$. To see that it in fact maps to $\cD_{P,\{1\}}^{\gamma_0,w_1}(V_1)$, we need only check the coefficient of $\Xi$, since $\Xi$ is the only basis element in $V_1$ with homogeneity less than $\sigma_1$. Since $\langle \bone,u(z)\rangle$ has $0$ limit at $P_1\setminus P_0$, so does $\langle \bone, (\hat g(u))(z)-g(0)\bone\rangle$, and therefore so does the coefficient of $\Xi$ in $F_1(u,v)$.

Finally, the third coordinate is the constant modelled distribution
$$
F^{(2)}(u)= g(0)\Xi.
$$

It remains to define $\hat\cR^\eps$ on $\cD_P^{\gamma_0,w_2}(\scal{\Xi})$. To this end, let us recall that for the 
model constructed for this equation
 in \cite[Sec.~10.4]{H0} (which coincides with the canonical BPHZ model defined more generally in
 \cite{BHZ,CH}) $\Pi^0_x\Xi$ is the spatial white noise $\xi$ for all $x$, while 
$\Pi_x^\varepsilon\Xi$ is the smoothed noise $\xi_\varepsilon$ for all $x$. 
Also notice that any $f\in\cD_P^{\gamma_0,w_2}(\scal{\Xi})$ is necessarily constant on $\R^+\times D$, 
and therefore in fact it suffices to define $\hat\cR^\eps(\bR^D_+\Xi)$ in a way that the continuity property \eqref{e:conteps} holds. 
Defining $\hat\cR^0(\bR^{D}_+\Xi)$ as $\mathbf{1}_{[0,\infty)\times D}\xi$ (which is of course a meaningful expression) 
and $\hat\cR^\varepsilon(\bR^{D}_+\Xi)$ as $\mathbf{1}_{[0,\infty)\times D}\xi_\varepsilon$ we therefore only
need to show that the convergence
$$
\|\mathbf{1}_{[0,\infty)\times D}\xi-\mathbf{1}_{[0,\infty)\times D}\xi_\varepsilon\|_{-1-\kappa;[0,1]\times D}\xrightarrow{\varepsilon\rightarrow0}0
$$
holds in probability for \eqref{e:conteps} to hold. This however follows in a more or less standard way 
from a Kolmogorov continuity type argument, see for example \cite[Prop.~9.5]{H0} for a very similar statement.

Therefore we can apply Theorem~\ref{thm:FPP} to get that the equation
$$
u=(\cK^{(\varepsilon)}_{\gamma_0}+Z_{\gamma_0}\hat\cR^\varepsilon)\big((F^{(0)},F^{(1)},F^{(2)})(u)\big)
+v
$$
has a unique local solution $u^\eps\in\cD_{P,\{1\}}^{\gamma,w}(V,\Gamma)$ for each of the models $(\Pi^\varepsilon,\Gamma^\varepsilon)$, for $\varepsilon\in[0,1]$. 
The fact that these correspond to the approximating equations in the sense that $\cR u^\varepsilon$ is the classical solution of \eqref{eq:0PAM approx}, for $\eps>0$,  follows exactly as in \cite{H0}: 
indeed, this is a property of the models and the compatibility of the abstract integration operators with the corresponding convolutions, neither of which changed in our setting. 
One also has, by Theorem~\ref{thm:FPP}, that $u^\varepsilon$ converges to $u^0$ in probability, with respect to the `distance' $\vn{\cdot;\cdot}_{\gamma,w,T}$. 
Therefore, $\cR u^\varepsilon$ also converge to $\cR u^0$ in probability, which proves Theorem~\ref{thm:PAM}.

\begin{remark}\label{remark:initialPAM}
If we replace $u_0$ in \eqref{eq:initial} by $(\cR u^0)(s,\cdot)$, $s<\tau$, where $\tau$ is the solution time from Theorem~\ref{thm:FPP}, then $v$ still belongs to $\CW_\eps$, in fact, one even has
$$
v\in\cD_P^{\gamma,(1-\kappa,1-\kappa,0)},\quad\scD_i v\in\cD_P^{\gamma-1,(-\kappa,-\kappa,-\kappa)},\,i=1,2.
$$
Therefore the solution can be restarted from time $s$ and these solutions can be patched together by the arguments in  \cite[Sec.~7.3]{H0}. One then sees that the only way that the solution may fail to be global is if $\|\hat\cR^0 F(u^0)\|_{-1-\kappa;s}$, and consequently, $\|(\cR u^0)(s,\cdot)\|_{1-\kappa,\bar D}$ blows up in finite time.
\end{remark}

\subsection{KPZ equation with Dirichlet boundary condition}
The construction of the regularity structure and models (as before, with a sufficiently large period)
for the KPZ equation can be for example found in \cite[Sec.~15]{FH}. The homogeneity of the symbol $\Xi$ is now denoted by $-3/2-\kappa$, where $\kappa\in(0,(1/8)\wedge\delta)\setminus\mathbb{Q}$, with $\delta$ being the regularity of the initial condition.

Similarly to the previous subsection, we let $v$ be the lift of the solution to the linear problem with initial 
condition $u_0$ (and Dirichlet boundary conditions). We also choose $K\in\scK_2$ and $Z\in\scZ_{2,P}$, as obtained 
from $G$, the Dirichlet heat kernel on the domain $D=(-1,1)$ as in Example~\ref{example}. 
We also set $\gamma=3/2+\kappa$, $\gamma_0=\kappa$, and define
\begin{equ}\label{eq:Psi}
\Psi=\Psi^\varepsilon=(\cK_{\gamma_0}^{(\eps)}+Z_{\gamma_0}\hat\cR^\eps)(\bR^D_+\Xi),
\end{equ}
where we define the distributions $\hat \cR^\eps\bR_+^D\Xi$ as in the previous subsection, with the obvious modification that $\xi$ now stands for the 1+1-dimensional space-time white noise.

We then write the abstract fixed point problem for the remainder of a one step expansion
\begin{equation}\label{eq:kpz2}
u=(\cK_{\gamma_0}+Z_{\gamma_0}\hat\cR)((F^{(0)},F^{(1)},F^{(2)})(u))+v,
\end{equation}
with
\begin{equs}
F^{(0)}(u)=(\scD u)^{\star 2},\quad
F^{(1)}(u)=2(\scD\Psi)\star(\scD u),\quad
F^{(2)}(u)\equiv(\scD\Psi)^{\star 2}.
\end{equs}

We further set
$$
V =\cI(T_{-1-2\kappa}^+)+\bar T,
\quad \quad 
V_0=(\scD V)^{\star 2},
\quad \quad
V_1=(\scD V)\star T_{-1/2-\kappa}^+,
\quad \quad
V_2=T_{-1-2\kappa}^+,
$$
which obviously implies $\alpha=0$, $\alpha_0=-4\kappa$, $\alpha_1=-1/2-3\kappa$, and $\alpha_2=-1-2\kappa$. As for the weight exponents, let
\begin{equs}
\eta&=\kappa, &\quad \sigma&=1/2+2\kappa,  &\quad  \mu&=-\kappa,\\
\eta_0&=2\kappa-2, &\quad \sigma_0&=2\kappa-1, &\quad \mu_0&=2\kappa-2,\\
\eta_1&=-3/2, &\quad \sigma_1&=\kappa-1, &\quad \mu_1&=-3/2,\\
\eta_2&=-1-2\kappa, &\quad \sigma_2&=-1-2\kappa, &\quad \mu_2&=-1-2\kappa.
\end{equs}
We then set 
similarly to above
\begin{equ}
\CW_\eps = \bigl\{u \in \cD_{P}^{\gamma,(\eta,\sigma,0)}\,:\, \scD_i u\in\cD_{P}^{\gamma-1,(\eta-1,\sigma-1,\kappa-1)},\,i = 1,2\bigr\}\;,
\end{equ}
as well as
$$
\cV_\eps=\cD_P^{\gamma_0,w_0}(V_0,\Gamma^\eps)\oplus\cD_{P}^{\gamma_0,w_1}(V_1,\Gamma^\eps)\oplus\scal{\bR^D_+(\scD\Psi^\eps)^{\star 2}}\;,
$$ 
and $\iota_\eps$ to be the identity. As before, it is straightforward to check that that the conditions on $\cV_\eps$ and $\cW_\eps$ satisfied, and also that regarding the first two coordinates of $\hat\cR^\eps$ one has 
a canonical choice given by Theorem~\ref{thm:reco}.

It remains to define $\hat\cR^\eps\bR^D_+(\scD\Psi)^{\star 2}.$ Recall that $\tilde \cR$ stands for the local reconstruction operator and that the issue with the singularity of low order is that $\tilde \cR \bR^{D}_+(\scD(G_\gamma\Xi))^{\star 2}$ does not have a canonical extension as a distribution in $\cC^{-1-2\kappa}$. Of course, for the approximating models this is just a bounded function, so it could even be extended as an element of $\cC^0$, but these extensions may not converge in the $\varepsilon\rightarrow0$ limit. Therefore some modification of these natural extensions are required at the boundary.
\begin{remark}
This process is very similar to the situation when one takes the sequence of distributions $1/(|x|+\varepsilon)$. This sequence of course does not converge to any distribution as $\varepsilon\rightarrow0$, but $1/(|x|+\varepsilon)+2\log(\varepsilon)\delta_0$ does, in $\cC^{-1-\rho}$ for any $\rho>0$. Moreover, the limiting distribution agrees with $1/|x|$ on test functions supported away from $0$.
\end{remark}
First, for the models $(\Pi^\varepsilon,\Gamma^\varepsilon)$, $\eps>0$, we denote by $\cR \bR^{D}_+(\scD\Psi)^{\star 2}$ 
the natural extension of $\tilde \cR \bR^{D}_+(\scD\Psi)^{\star 2}$ which, as just mentioned, is a bounded function 
and can be written in the form
$$
(\cR \bR^{D}_+(\scD\Psi)^{\star 2})(z)=A_2^\varepsilon(z)+A_0^\varepsilon(z),
$$
where $A_i^\varepsilon(z)$ are random variables belonging to the $i$-th homogeneous Wiener chaos for $i=0,2$. 
To write them more explicitly, introduce the notations $\bar f(s,y)=f(-s,-y)$ for any function $f$, set
\begin{equs}
\tilde{K}_{Q,\varepsilon}(z,z')&=(\bar\rho_\varepsilon\ast (D_1K(z,\cdot)\mathbf{1}_{Q}(\cdot)))(z'),\\
\tilde{Z}_{Q,\varepsilon}(z,z')&=(\bar\rho_\varepsilon\ast (D_1Z(z,\cdot)\mathbf{1}_{Q}(\cdot)))(z'),
\end{equs}
and define $\tilde G_{Q,\varepsilon}=\tilde K_{Q,\eps}+\tilde Z_{Q,\eps}$ for any $Q\subset\R^d$, and with the convention that for $\varepsilon=0$ we substitute the convolution $\bar\rho_\varepsilon\ast$ with the identity. We can then write 
\begin{equs}\label{eq:A2epsilon}
A^\varepsilon_2(z)&=\int (\tilde G_{[0,\infty)\times D,\varepsilon})(z,z')(\tilde G_{[0,\infty)\times D,\varepsilon})(z,z'')\,\xi(dz')\,\xi(dz''),\\
A_0^\varepsilon(z)&=\int (\tilde G_{[0,\infty)\times D,\varepsilon}(z,z'))^2-\tilde K_{\R^d,\varepsilon}^2(z,z')\, dz'.
\label{eq:kpzz}
\end{equs}
Note that the reason for the subtraction in \eqref{eq:kpzz} is the renormalisation already built in the model $(\Pi^\varepsilon,\Gamma^\varepsilon)$. Similarly, for the limiting model $(\Pi^0,\Gamma^0)$,
$$
\tilde \cR \bR^{D}_+(\scD\Psi)^{\star 2}=A_2+A_0,
$$
where $A_2$ and $A_0$ are given by setting $\varepsilon=0$ with the above mentioned convention in \eqref{eq:A2epsilon} and \eqref{eq:kpzz}, respectively.

The convergence of $A_2^\varepsilon$ to $A_2$ in the $\varepsilon\rightarrow0$ limit in $\cC^{-1-\kappa}$ follows from essentially the same power counting argument as in the case without boundary conditions.
The term $A_0^\varepsilon(z)$ however is more delicate. While it is not difficult to show that it converges pointwise to the 
smooth function $A_0(z)$ on $(0,\infty)\times D$, the convergence in $\cC^{-1-\kappa}$ is not a priori clear. In fact, 
without using the specific form of $G$, one cannot even rule out that the limit exhibits a non-integrable singularity 
at the spatial boundary. To see how this can be `countered', first define
\begin{align}
B_0^\varepsilon(z)&=\int (\tilde G_{(-\infty,0)\times D,\varepsilon})
(\tilde G_{\R\times D,\varepsilon}+\tilde G_{[0,\infty)\times D,\varepsilon})
(z,z')
dz',\nonumber\\
C_0^\varepsilon(z)&=\int (\tilde K_{\R\times D,\varepsilon}+\tilde Z_{\R\times D,\varepsilon})^2(z,z')-\tilde K_{\R^d,\varepsilon}^2(z,z')dz'
\nonumber\\
&=\int 2\tilde K_{\R\times D,\varepsilon}\tilde Z_{\R\times D,\varepsilon}(z,z')+
\tilde Z_{\R\times D,\varepsilon}^2(z,z')-\tilde K_{\R\times D^c,\varepsilon}^2(z,z')
\nonumber\\
&\quad -2\tilde K_{\R\times D^c,\varepsilon}\tilde K_{\R\times D,\varepsilon}(z,z')\,dz',\label{eq:C0}
\end{align}
for $z\in(0,\infty)\times D$, and extending them by $0$ otherwise, we have $A_0^\varepsilon=-B_0^\varepsilon+C_0^\varepsilon$. We can similarly write $A_0=-B_0+C_0$, where $B_0$ and $C_0$ are defined by formally setting $\varepsilon=0$ in the above definitions, that is, replacing the convolution with $\rho_\varepsilon$ with the identity.

First we claim that for $z\in(0,\infty)\times D$
\begin{equation}\label{kpz:blowup t}
|B_0^\varepsilon(z)|\lesssim 1/(|z|_{P_0}+\varepsilon)=1/(t^{1/2}+\varepsilon).
\end{equation}

It is easy to see that one has the decomposition
\begin{equation}\label{kpz:decom}
(\tilde G_{(-\infty,0)\times D,\varepsilon})(z,\cdot)=\sum_{n\geq0}\tilde{G}^{(n)}(\cdot),
\end{equation}
where, for each $n$, the function $\tilde{G}^{(n)}$ is supported on $\{z':|z'|_{P_0}\leq\varepsilon,\|z-z'\|_\frs\leq 2^{-n}+\varepsilon\}$, and is bounded by $2^{-n}(\varepsilon\vee 2^{-n})^{-3}$. Furthermore, the function $(\tilde G_{\R\times D,\varepsilon}+\tilde G_{[0,\infty)\times D,\varepsilon})
(z,\cdot)$ is also bounded by $2^{-n}(\varepsilon\vee 2^{-n})^{-3}$ on the support of $\tilde G^{(n)}$. Hence in the case $|z|_{P_0}\geq 3\varepsilon$, noting that the only nonzero terms in the sum \eqref{kpz:decom} are those where $2^{-n}\geq(|z|_{P_0}/3)$, we can bound
\begin{align*}
B_0^\varepsilon(z)
&\lesssim\int\sum_{(|z|_{P_0}/3)\leq 2^{-n}}2^{-3n}2^{2n}2^{2n}\lesssim 1/|z|_{P_0}
\end{align*}
as required. On the other hand, in the case $|z|_{P_0}\leq3\varepsilon$, we have
\begin{align*}
B_0^\varepsilon(z)\lesssim \sum_{2^{-n}>\varepsilon}2^{-3n}2^{2n}2^{2n}+
\sum_{2^{-n}\leq\varepsilon}\varepsilon^32^{-n}\varepsilon^32^{-n}\varepsilon^3\lesssim1/\varepsilon,
\end{align*}
as required. The estimate
$
|B_0(z)|\lesssim 1/t^{1/2}
$
can be obtained analogously. Since $B_0^\varepsilon$ (extended by $0$ outside of $(0,\infty)\times D$) converges to $B_0$ locally uniformly on $(0,\infty)\times D$ and since by the above estimates $(B_0^\varepsilon)_{\varepsilon\in(0,1]}$
and $B_0$ are uniformly bounded in $\cC^{-1-\kappa/2}$, the convergence also holds in $\cC^{-1-\kappa}$.

Moving on to $C_0^\varepsilon$, first notice that it only depends on the variable $x$. Furthermore, by similar calculations as above, one obtains a bound analogous to \eqref{kpz:blowup t}, namely
\begin{equation}\label{kpz:blowup}
|C^\varepsilon_0(z)|\lesssim 1/(|z|_{P_1}+\varepsilon)=1/\big((x+1)\wedge(1-x)+\varepsilon\big)
\end{equation} 
for $z\in(0,\infty)\times D$.
 We then define the distribution $\hat C^\varepsilon_0$ by
\begin{equ}[e:defC0hat]
(\hat C^\varepsilon_0,\varphi):=\int C^\varepsilon_0(z)[\varphi(z)-\chi(x+1)\varphi(t,-1)-\chi(x-1)\varphi(t,1)]dz,
\end{equ}
where $\chi$ is a smooth symmetric cutoff function in the $x$ variable which is 
$1$ on $\{x':|x'|\leq 1/8\}$, and is supported on $\{x':|x'|\leq 1/4\}$. 
The estimate \eqref{kpz:blowup}, together with the local uniform convergence of $C_0^\eps$, then implies that $\hat C^\varepsilon_0$ converges in $\cC^{-1-\kappa}$ to a limit, which we denote by $\hat C^\ast_0$. 
Moreover, since $\hat C^\varepsilon_0$ agrees with $C^\varepsilon_0$ on test functions supported away from $P$, $\hat C^\ast_0$ also agrees with $C_0$ on the same class of test functions. In other words, defining
\begin{equation}\label{eq:kpz hat epsilon}
\hat \cR^\varepsilon \bR^{D}_+(\scD\Psi)^{\star 2}=A_2^\varepsilon-B_0^\varepsilon+\hat C^\varepsilon_0,
\end{equation}
as well as
\begin{equation}\label{eq:kpz hat}
\hat \cR^0 \bR^{D}_+(\scD\Psi)^{\star 2}=A_2-B_0+\hat C_0^\ast,
\end{equation}
the desired properties \eqref{eq:hatR0}-\eqref{e:conteps} of $(\hat\cR^\eps)_{\eps\in[0,1]}$ hold

Therefore by Theorem~\ref{thm:FPP} we can conclude \eqref{eq:kpz2} has a unique local solution $u^\varepsilon\in\cD_P^{\gamma,w}(V,\Gamma^\eps)$
 for each $\eps\in[0,1]$,
 and $\cR(u^\varepsilon+\Psi^\eps)$ converges to $\cR( u^0+\Psi^0)$. To conclude the proof of Theorem~\ref{thm:KPZ D}, 
it remains to confirm that for $\eps>0$,  $\cR( u^\varepsilon+\Psi^\eps)$ solves \eqref{eq:0KPZ approx}. 
This would again follow in exactly the same manner as in \cite{H_KPZ} if we used the `natural' reconstructions everywhere, which we only steered away from in the previous construction. However, since $\hat \cR^\eps$ and $\cR$ only differ by some (finite) Dirac mass on the boundary, and since $G$, the Dirichlet heat kernel, vanishes on the boundary, we have 
\begin{align}
\cR (\cK_{\gamma_0}^{(\eps)}+Z_{\gamma_0}\hat\cR^\eps)\bR^D_+(\scD\Psi^\eps)^{\star 2}&
=G\ast\hat\cR^\varepsilon(\bR^D_+(\scD\Psi^\eps)^{\star 2})
\nonumber\\&=G\ast\cR(\bR^D_+(\scD\Psi^\eps)^{\star 2}).\label{eq:kpz hat r}
\end{align}
The previous modification is therefore not visible after the application of the reconstruction operator, and this 
concludes the proof of Theorem~\ref{thm:KPZ D}.

\subsection{KPZ equation with Neumann boundary condition}
\label{sec:KPZNeumann}

Most of the arguments of the previous subsection carry through if the Dirichlet heat kernel is replaced by the Neumann heat kernel, with the sole exception of \eqref{eq:kpz hat r}. Instead, we have
\begin{align}
\cR (\cK_{\gamma_0}^{(\eps)}+Z_{\gamma_0}\hat\cR^\eps)\bR^D_+(\scD\Psi^\eps)^{\star 2}&
=G\ast\hat\cR^\varepsilon(\bR^D_+(\scD\Psi^\eps)^{\star 2})
\nonumber\\&=G\ast(\cR(\bR^D_+(\scD\Psi^\eps)^{\star 2})-c^-_\varepsilon\delta_{-1}-c^+_\varepsilon\delta_1),
\end{align}
where $\delta_{\pm1}$ is the Dirac distribution at $x=\pm1$,  and
$$
c^-_\varepsilon=\int_{[-1,-3/4]}C_0^\varepsilon(x)\chi(x+1)\,dx,\quad
c^+_\varepsilon=\int_{[3/4,1]}C_0^\varepsilon(x)\chi(x-1)\,dx.
$$
(We henceforth view $C_0^\eps$ and $C_0$ as functions of the spatial variable $x$ only, since we already noted
that these functions, as defined in \eqref{eq:C0}, do not depend on the time variable.) 
Since these Dirac masses now do not cancel, one needs more concrete information about $c^-_\varepsilon$ and $c^+_\varepsilon$, and we begin with the former. 
First, it will be convenient to shift the equation to the right, so that the left boundary is at $x=0$. 
Furthermore, we note that we can add a globally smooth component to $K$ and $Z$ in the definitions 
of $C_0^\varepsilon$ and $C_0$ without changing the conclusion that $\hat C_0^\eps$ as defined
by \eqref{e:defC0hat} converges to a limit $\hat C_0^*$. In particular, setting
\begin{equ}[e:notationN]
\CN(x,\sigma) = {\bone_{\sigma > 0} \over \sqrt{2\pi\sigma}}\exp\Big(-{x^2\over 2\sigma}\Big)\;,
\end{equ}
we can assume that for $x\in[0,1/4]$, one has 
\begin{equs}
K((0,x),(-s,y)) = \CN(x-y,s),\qquad 
Z((0,x),(-s,y)) = \CN(x+y,s)\;.
\end{equs}
With the notations
\begin{equs}
f^{(1)}_x(s,y)=\bone_{y>0}\frac{x-y}{s}\CN(x-y,s),
\quad
f^{(2)}_x(s,y)=\bone_{y>0}\frac{x+y}{s}\CN(x+y,s),
\end{equs}
as well as $f^{(3)}_x(s,y) = f^{(1)}_x(s,y) + f^{(2)}_x(s,-y)$,
and after a trivial change of variables in $s$, we can then write, recalling the notation $\bar f(s,y)=f(-s,-y)$ for any function $f$ of time and space,
\begin{equs}
C_0^\varepsilon(x)&=\int_{\R^2}(\bar\rho_\varepsilon\ast(f^{(1)}_x+f^{(2)}_x))^2(s,y)-(\bar\rho_\varepsilon\ast f^{(3)}_x)^2(s,y)\,ds\,dy,\label{eq:C0varepsilon}
\\
C_0(x)&=\int_{\R^2}(f^{(1)}_x+f^{(2)}_x)^2(s,y)-(f^{(3)}_x)^2(s,y)\,ds\,dy.
\end{equs}
Note that our modifications of $K$ and $Z$ are only valid for $x\in[0, 1/4]$, and so \eqref{eq:C0varepsilon} also holds for these values of $x$. But since other values do not play a role in computing $c_\varepsilon^-$, for the duration of this computation we can simply define $C_0^\varepsilon(x)$ as the right-hand side of \eqref{eq:C0varepsilon} for other values of $x$. We can then write the decomposition
$$
c_\varepsilon^-=\bar c_\varepsilon^--\hat c_\varepsilon^-:=\int_0^\infty C_0^\varepsilon(x)\,dx-\int_0^\infty(1-\chi(x))C_0^\varepsilon (x)\,dx,
$$
We first show that the second term in this decomposition doesn't matter.

\begin{proposition}
With the above notations, one has $C_0(x) = 0$ for every $x \neq 0$. Furthermore, for every $\kappa \in (0,1)$,
there exists a constant $C$ such that, for $|x| \ge C\eps$, one has the bound 
$|C_0^\eps(x)| \le C\eps^{1-\kappa} |x|^{\kappa-2}$.
\end{proposition}

\begin{proof}
%
The first statement follows from the second one since $C_0^\eps \to C_0$ locally uniformly, so
it remains to show that the claimed bound on $C_0^\eps(x)$ holds.
We will assume without the loss of generality that $x>C\varepsilon$ for some sufficiently large $C$
($C=6$ will do)
and we write $z=(0,x)$ and $z'=(s,y)$. Since $f^{(3)}_x=f^{(1)}_x+f^{(3)}_x\bone_{y<0}$ almost everywhere, one has
\begin{equs}
C_0^\varepsilon(x)&=\int_{\R^2}2(\bar\rho_\varepsilon\ast f^{(1)}_x)(\bar\rho_\varepsilon\ast f^{(2)}_x)\,dz'
+\int_{\R^2}(\bar\rho_\varepsilon\ast f^{(2)}_x)^2-(\bar\rho_\varepsilon\ast (f^{(3)}_x\bone_{y<0}))^2\,dz'
\\
&\quad-\int_{\R^2}2(\bar\rho_\varepsilon\ast (f^{(3)}_x\bone_{y<0}))(\bar\rho_\varepsilon\ast (f^{(3)}_x\bone_{y>0}))\,dz'
\\
&=:2J_1+J_2-2J_3.
\end{equs}
With the usual convention $\bar\rho_0\ast$ standing for the identity, we can furthermore write
\begin{equs}
J_1&=\int f^{(1)}_x f^{(2)}_x\,dz'+
2\int(\bar\rho_\varepsilon\ast f^{(1)}_x)((\bar\rho_\varepsilon-\bar\rho_0)\ast f^{(2)}_x)\,dz'
 +2\int((\bar\rho_\varepsilon-\bar\rho_0)\ast f^{(1)}_x)f^{(2)}_x\,dz'
\\
&=:I_1+I_2+I_3.
\end{equs}
The expression $I_1$ actually vanishes, since
\begin{align*}
I_1&=\int_{s>0}\frac{x^2-y^2}{s^{2}}\CN(x,s)\CN(y,s)\,dz'=\int_{s>0}{x^2-s \over s^2} \CN(x,s)\,ds \\
&=\int_{r>0}{r^2 - 1 \over |x|}\CN(r,1)\,dr = 0\;.
\end{align*}
To estimate $I_2$, we first note that it follows immediately from the scaling properties of $f^{(2)}_x$ 
and the fact that it only has a discontinuity at $y=0$, that one can write
\begin{equ}
(\bar\rho_\varepsilon-\bar\rho_0)\ast f_x^{(2)} = f_{x,\eps}^{(2,1)} + f_{x,\eps}^{(2,2)}\;,
\end{equ}
where $f_{x,\eps}^{(2,2)}$ is supported on $\R \times [-2\eps,2\eps]$ and, for any $\kappa \in [0,1]$, 
one has the bounds
\begin{equ}[e:boundsf2]
|f_{x,\eps}^{(2,1)}(z')|\lesssim {\eps^{1-\kappa} \over \|z'+z\|^{3-\kappa}}\;,\qquad
|f_{x,\eps}^{(2,2)}(z')|\lesssim {1 \over \|z'+z\|^{2}} \lesssim {1\over s+x^2}\;.
\end{equ}
It follows immediately from standard properties of convolutions (see for example
\cite[Lem.~10.14]{H0}) that
\begin{equ}
\Bigl|\int (\bar\rho_\varepsilon\ast f^{(1)}_x)f_{x,\eps}^{(2,1)}\,dz'\Bigr| \lesssim \eps^{1-\kappa} |x|^{\kappa-2}\;,
\end{equ}
as required. Regarding the term involving $f_{x,\eps}^{(2,2)}$, it follows from the support properties of that
function that 
\begin{equ}[e:boundTermEps]
\Bigl|\int (\bar\rho_\varepsilon\ast f^{(1)}_x)f_{x,\eps}^{(2,2)}\,dz'\Bigr| \lesssim \eps \int_0^\infty {ds \over (s+x^2)^2} \lesssim \eps |x|^{-2}\le \eps^{1-\kappa} |x|^{\kappa-2}\;.
\end{equ}
The term $I_3$ can be bounded in exactly the same way.

To bound $J_2$, we use the notation $\tilde \rho_\varepsilon(t,x) = \bar \rho_\varepsilon(t,-x)$. 
Since $(f^{(3)}_x\bone_{y<0})(s,y) = f^{(2)}_x(s,-y)$, we can then rewrite $J_2$ as
\begin{equ}
J_2=\int_{\R^2}((\bar\rho_\varepsilon-\tilde \rho_\varepsilon)\ast f_x^{(2)})((\bar\rho_\varepsilon+\tilde\rho_\varepsilon)\ast f_x^{(2)})
\,dz'\;.
\end{equ}
Exactly as above, we can decompose the first factor as
\begin{equ}
(\bar\rho_\varepsilon-\tilde\rho_\eps)\ast f_x^{(2)} = f_{x,\eps}^{(2,1)} + f_{x,\eps}^{(2,2)}\;,
\end{equ}
so that the bounds \eqref{e:boundsf2} hold and $f_{x,\eps}^{(2,2)}(z') = 0$ for $y \not \in [-2\eps,2\eps]$. 
This time, we exploit the fact that the second factor itself satisfies the bound
\begin{equ}
|((\bar\rho_\varepsilon+\tilde\rho_\varepsilon)\ast f_x^{(2)})(z')| \lesssim \|z+z'\|^{-2}\;,
\end{equ}
uniformly in $\eps$, and that the support of both factors is included in 
the set $\|z+z'\| \ge |x|/2$. As a consequence, the term involving $f_{x,\eps}^{(2,1)}$ is bounded
by
\begin{equ}
\int_{\|z'\| \ge |x|/2} {\eps \over \|z'\|^5}\,dz' \lesssim \eps |x|^{-2}\;,
\end{equ}
while the other term is bounded exactly as in \eqref{e:boundTermEps}.

Finally, regarding $J_3$, the product is supported on $\R \times [-\eps,\eps]$ and each
factor is bounded by $(s+x^2)^{-1}$ there, so that the corresponding integral is again bounded 
as in \eqref{e:boundTermEps}, thus concluding the proof.
\end{proof}

Let us now return to the computation of the constant $\bar c_\varepsilon^-$.
Using the identity $(f\ast g,h)=(g,\bar f\ast h)_{L^2(R^2)}$ and the commutativity of the convolution, we 
can rewrite it as
\begin{equs}\label{eq:bar c0 as a product}
\bar c_\varepsilon^-= (\bar\rho_\varepsilon\ast\rho_\varepsilon, F)_{L^2(R^2)},
\end{equs}
where
$$
F=F_1+F_2:= \int_\R\bar f^{(1)}_x\ast f^{(2)}_x\,dx+{1\over 2}\int_\R(\bar f^{(1)}_x\ast f^{(1)}_x+\bar f^{(2)}_x\ast f^{(2)}_x-\bar f^{(3)}_x\ast f^{(3)}_x)\,dx.
$$
We will use again the notation \eqref{e:notationN} and we will make use of the 
identities
\begin{equs}
\CN(x,\sigma)\,\CN(y,\eta) &= \CN(x \pm y,\sigma+\eta)\, \CN \Bigl({\eta x \mp \sigma y \over \sigma+\eta},{\sigma \eta \over \sigma+\eta}\Bigr)\;,\\
\d_x \CN(x,\sigma) &= -(x/\sigma)\CN(x,\sigma)\;.
\end{equs}
The first identity can be obtained by considering a jointly Gaussian centred random variable
$(X,Y)$ with $\Var(Y) = \sigma$, $\E (X\,|\,Y) = Y$, $\Var(X\,|\,Y) = \eta$ and noting that one then
has $\Var(X) = \sigma+ \eta$, $\E(Y\,|\,X) = {\sigma X \over \sigma+\eta}$, and $\Var(Y\,|\,X) = {\sigma \eta \over \sigma+\eta}$.
Exploiting this identity, we can rewrite $F_1$ as
\begin{equs}
F_1 &= \int \bone_{y'>y \vee 0} {x-y'+y \over s'-s}{x+y'\over s'} \CN(x-y'+y,s'-s)\CN(x+y',s')\,dz'\,dx \\
&={1\over 4}\int \bone_{y'>y \vee 0} {(2x+y)^2 - (2y'-y)^2 \over s'(s'-s)} \CN(2y'-y,2s'-s)\\
&\qquad \times \CN\Big(x+{y\over 2} - {s(2y'-y) \over 2(2s'-s)},{s'(s'-s) \over 2s'-s}\Big)\,dz'\,dx\;.
\end{equs}
We now perform the change of variables $2y'-y \mapsto y'$ and $2s'-s \mapsto s'$ which in
particular maps $dz'$ to ${1\over 4}dz'$ and $s'(s'-s)$ to $((s')^2-s^2)/4$ so that
\begin{equs}
F_1 &={1\over 4}\int \bone_{y'>|y|} {(2x+y)^2 - (y')^2 \over (s'+s)(s'-s)} \CN(y',s')
 \CN\Big(x+{y\over 2} - {s y' \over 2s'},{(s'-s)(s'+s) \over 4 s'}\Big)\,dz'\,dx\\
 &= {1\over 4}\int \bone_{y'>|y| \atop s' > |s|}{1\over s'} \Big({1- {(y')^2 \over s'}}\Big) \CN(y',s')\,dz'\\
 &= {1\over 4}\int \bone_{y'>|y| \atop s' > |s|}{1\over \sqrt{s'}} \Big({1- {(y')^2 \over s'}}\Big) \CN(y'/\sqrt{s'},1)\,{dz' \over s'}\;.
\end{equs}
At this stage, for fixed $y'$, we perform the change of variables $r = y'/\sqrt{s'}$, so that
$dz'/s' = 2 dy'\,dr/r$, thus yielding
\begin{equs}
F_1(z) &= {1\over 2}\int_{|y|}^\infty {1\over y'} \int_0^{y' \over \sqrt{|s|}} \big({1- r^2}\big) \CN(r,1)\,dr\,dy'
= -{1\over 2}\int_{|y|}^\infty {1\over y'} \int_0^{y' \over \sqrt{|s|}} \d_r^2 \CN(r,1)\,dr\,dy' \\
&= -{1\over 2}\int_{|y|}^\infty {1\over y'} (\d_1 \CN)\Big({y' \over \sqrt{|s|}},1\Big) \,dy'
= {1\over 2}\int_{|y|}^\infty {1\over \sqrt{|s|}} \CN\Big({y' \over \sqrt{|s|}},1\Big) \,dy' \\
&= {1\over 2}\int_{|y| \over \sqrt{|s|}}^\infty \CN(q,1) \,dq = {1\over 4} - {1\over 4}\Erf \Bigl({|y| \over \sqrt{2|s|}}\Bigr)\;.
\end{equs}
Let's now turn to $F_2$. Setting $f_x(z) = {x-y\over s}\CN(x-y,s)$, a simple calculation shows that
\begin{equs}
F_2(z) &= {1\over 2} \int f_x(z-z')f_x(-z') \bigl(\bone_{y'<(0\wedge y)} + \bone_{y'>(0  \vee y)} - 1\bigr)\,dz'\,dx \\
&= -{1\over 2} \int f_x(z-z')f_x(-z') \bone_{-|y| < 2y'-y < |y|}\,dz'\,dx\\
&= {1\over 2} \int {x-y + y'\over s - s'}{x + y'\over s'}\CN(x -y+ y', s-s') \CN(x + y', -s') \bone_{|2y'-y| < |y|}\,dz'\,dx\\
&= - {1\over 8} \int {(2x+y')^2-y^2 \over (s')^2 - s^2} \CN(y,s') \CN \Bigl(x + {y'\over 2} + {ys\over 2s'}, {(s')^2 - s^2 \over 4s'}\Bigr) \bone_{|y'| < |y|} \,dz'\,dx\\
&= - {1\over 8} \int_{-|y|}^{|y|} \int_{|s|}^\infty \Bigl({1\over s'} - {y^2 \over (s')^2}\Bigr) \CN(y,s') \,ds'\,dy'\\
&= {|y|\over 4} \int_{|s|}^\infty {1\over s'}\Bigl({y^2 \over s'} - 1\Bigr) \CN(y,s') \,ds'
= -{|y| \over 2} \CN(y,|s|)\;,
\end{equs}
where the last equality was obtained in exactly the same way as above.
Combining these identities with \eqref{eq:bar c0 as a product} and exploiting the fact that $F$ is $0$-homogeneous
under the parabolic scaling, we finally obtain
\begin{equ}\label{eq:bar c0}
\bar c_\varepsilon^-=\int_{\R^2}(\bar\rho\ast\rho)(s,y)\,\Big({1\over 4} - \frac{1}{4}\Erf\Big(\frac{|y|}{\sqrt{2|s|}}\Big)-{|y| \over 2} \CN(|y|,|s|)\Big)\,ds\,dy=\frac{a}{2},
\end{equ}
where $a$ is the quantity given in \eqref{eq:constant a}.

If momentarily one also includes the dependence of $c^{\pm}_\eps$ on $\rho$, one has, by symmetry, $c^+_\eps(\rho)=c^-_\eps(\hat \rho)$, with $\hat \rho(t,x) = \rho(t,-x)$. Therefore by \eqref{eq:bar c0}, $c_\varepsilon^+=\bar c_\eps^+-\hat c_\varepsilon^+$, where $\hat c_\varepsilon^+\rightarrow 0$ as $\varepsilon\rightarrow0$ and $\bar c_\eps^+$ is given by
\begin{equ}\label{eq:bar c1}
\bar c_\eps^+ = \int_{\R^2}(\bar{\hat\rho}\ast\hat\rho)(s,y)F(y,s)\,ds\,dy
=\int_{\R^2}(\bar\rho\ast\rho)(s,y)F(-y,s)\,ds\,dy = \frac{a}{2},
\end{equ}
since $F$ is symmetric in both of its arguments.

We can conclude that, for any fixed constants $\hat b_\pm \in \R$, setting
\begin{equation}\label{eq:kpz hat epsilon N}
\hat \cR^\varepsilon \bR^{D}_+(\scD\Psi)^{\star 2}=A_2^\varepsilon-B_0^\varepsilon+ C^\varepsilon_0- {1\over 2}\bone_{t>0}\big((a-\hat b_-)\delta_{-1}+ (a+\hat b_+)\delta_1)\Big),
\end{equation}
for the models $(\Pi^\varepsilon,\Gamma^\varepsilon)$ and
\begin{equation}\label{eq:kpz hat N}
\hat \cR^0 \bR^{D}_+(\scD\Psi)^{\star 2}=A_2-B_0+ C_0 - {1\over 2}\bone_{t>0}(\hat b_+\delta_1- \hat b_-\delta_{-1})\Big) ,
\end{equation}
for the limiting model,
the desired properties \eqref{eq:hatR0}-\eqref{e:conteps} of $(\hat\cR^\eps)_{\eps\in[0,1]}$ hold. 
Similarly to before, but accounting for the additional Dirac masses, we then see that for any 
fixed $\eps > 0$ the function $h^\eps = \cR( u^\varepsilon+\Psi^\eps)$ 
(there is no ambiguity for the reconstruction operator as far as the solution $u^\eps$ is concerned, it is trivially given
simply by the component in the direction $\bone$) solves
\begin{equs}[eq:0KPZ N approx 2]
        \partial_t h^\varepsilon &=\tfrac{1}{2}\d_x^2 h^\varepsilon+(\partial_x h^\varepsilon)^2+2c\partial_x  h^\varepsilon-C_\varepsilon+\xi_\varepsilon \quad & \text{on } &\R_+\times [-1,1],\\
        \partial_x  h^\varepsilon &= \mp a + b_{\pm} & \text{on } &\R_+\times \{\pm 1\},\\
         h^\varepsilon &=u_0 &\text{on }&\{0\}\times[-1,1],
\end{equs}
where $c$ is given by \eqref{e:defc} below. Hence, clearly, $\hat h^\eps = h^\eps +cx+(C_\varepsilon+c^2) t$ solves \eqref{eq:0KPZ N approx} with boundary data $\hat b_\pm=\mp a+ b_\pm+c$ and $\hat u_0(x)=u_0(x)+cx$.

Applying again Theorem~\ref{thm:FPP}, combined with the results of \cite{HS15} regarding the convergence of the 
corresponding admissible models, we conclude that, for any choice of $b_\pm$, the solution
to \eqref{eq:0KPZ N approx 2} (which is precisely the same as \eqref{e:renormu} provided
that the constant $C_\eps$ is adjusted in the appropriate way) converges locally
as $\eps \to 0$ to a limit which depends on the choice of $b_\pm$ but is independent of
the choice of mollifier $\rho$. It remains to show that this limit coincides with the 
Hopf-Cole solution to the KPZ equation with Neumann boundary data given by $b_\pm$.
This follows by considering the special case $\rho(t,x) = \delta(t)\hat \rho(x)$, which is covered 
by the above proof, the only minor modification being the proof of convergence of the corresponding
admissible model to the same limit, which can be obtained in a way very similar to \cite{H_KPZ,H0}. 
As already mentioned at the end of Section~\ref{subsec:applications}, one has $a = c = 0$ in this case, 
so that in particular
$\hat b_\pm = b_\pm$.
In this case, we can apply It\^o's formula to perform the Hopf-Cole transform and obtain convergence
to the corresponding limit by classical means \cite{DPZ}, which concludes the proof.

\subsubsection{Expression for the drift term}

It follows from \cite{HS15} that the constant $c$ appearing in \eqref{eq:0KPZ N approx 2}
is given by
\begin{equ}[e:defc]
c = -2\scal{\rho * \bar \rho, \d_x P * \d_x P * \overline{\d_x P} } =: \scal{\rho * \bar \rho, F_{0}}\;,
\end{equ}
where $P$ is the heat kernel.
Similarly to above, we obtain the identity
\begin{equs}
(\d_x P * \d_x P)(t,x) &= \int {x-y\over t-s} {y\over s} \CN(y,s)\CN(x-y,t-s) \,dy\,ds\\
&= \CN(x,t) \int {x-y\over t-s} {y\over s}  \CN\Bigl(y-{s x \over t},{s(t-s) \over t}\Bigr) \,dy\,ds\\
&= \CN(x,t) \int_0^t {x^2 - t \over t^2} \,ds
= \CN(x,t) {x^2 - t\over t}\;,
\end{equs}
which then implies that the function $F_0$ is indeed given by 
\begin{equs}
F_0(t,x) &= 2\int {y^2-s\over s} {x-y\over t-s} \CN(y,s)\CN(x-y,s-t) \bone_{s \ge 0\vee t}\,dy\,ds \\
&= 2  \int {y^2-s\over s} {x-y\over t-s} \CN(x,2s-t) \CN\Bigl(y-{s x \over 2s-t},{s(s-t) \over 2s-t}\Bigr) \bone_{s \ge 0\vee t}\,dy\,ds\\
&= 2  \int {(2y^2-r-t)(y-x)\over r^2-t^2} \CN(x,r) \CN\Bigl(y-{(r+t) x \over 2r},{r^2-t^2 \over 4r}\Bigr) \bone_{r \ge |t|}\,dy\,dr\\
&= \int_{|t|}^\infty {(r+t) x \over 2r^2} \Bigl(3-{x^2 \over r}\Bigr) \CN(x,r) \,dr 
= \Erf(x/\sqrt{2|t|}) + 2x\CN(x,t)\;.
\end{equs}
To obtain \eqref{eq:constant c}, it remains to note that the first term is odd under the 
substitution $(t,x) \leftrightarrow (-t,-x)$, while $\rho * \bar \rho$ is even, so that this
does not contribute to the value of $c$.

\bibliography{boundary}{}
\bibliographystyle{Martin} 

\end{document}